\documentclass[11pt,a4paper]{amsart}
\usepackage{amssymb,amsmath}
\usepackage{mathrsfs}
\usepackage{newtxtext}
\usepackage[alphabetic]{amsrefs}
\usepackage{graphicx,color}
\usepackage[utf8]{inputenc}
\usepackage{amsthm}
\usepackage{latexsym}
\usepackage{fullpage}
\usepackage{slashed}
\usepackage[all]{xy}
\usepackage{hyperref}
\usepackage[mathscr]{eucal}
\usepackage{tikz}
\usepackage{mathtools}
\usepackage{accents}
\usepackage{amsaddr}

\def\theequation{\thesection.\arabic{equation}}
\makeatletter
\@addtoreset{equation}{section}
\makeatother

\makeatletter
\newcommand{\eqnum}{\refstepcounter{equation}\textup{\tagform@{\theequation}}}
\makeatother

\newcounter{copy}
\makeatletter
\renewcommand{\thecopy}{\ifnum0=\c@section\arabic{copy}\else\thesection.\arabic{copy}'\fi}
\makeatother

\BibSpec{book}{%
    +{}  {\PrintPrimary}                {transition}
    +{.} { \textit}                     {title}
    +{.} { }                            {part}
    +{:} { \textit}                     {subtitle}
    +{,} { \PrintEdition}               {edition}
    +{}  { \PrintEditorsB}              {editor}
    +{,} { \PrintTranslatorsC}          {translator}
    +{,} { \PrintContributions}         {contribution}
    +{,} { }                            {series}
    +{,} { \voltext}                    {volume}
    +{,} { }                            {publisher}
    +{,} { }                            {organization}
    +{,} { }                            {address}
    +{,} { }                            {status}
    +{}  { \parenthesize}               {language}
    +{}  { \PrintTranslation}           {translation}
    +{;} { \PrintReprint}               {reprint}
    +{,} { \PrintDate}                  {date}
    +{.} { }                            {note}
    +{.} {}                             {transition}
}
\BibSpec{article}{%
    +{}  {\PrintAuthors}                {author}
    +{,} { \textit}                     {title}
    +{.} { }                            {part}
    +{:} { \textit}                     {subtitle}
    +{,} { \PrintContributions}         {contribution}
    +{.} { \PrintPartials}              {partial}
    +{,} { }                            {journal}
    +{}  { \textbf}                     {volume}
    +{}  { \PrintDatePV}                {date}
    +{,} { \issuetext}                  {number}
    +{,} { \eprintpages}                {pages}
    +{,} { }                            {status}
    +{,} { \PrintDOI}                   {doi}
    +{}  { \parenthesize}               {language}
    +{}  { \PrintTranslation}           {translation}
    +{;} { \PrintReprint}               {reprint}
    +{.} { }                            {note}
    +{.} {}                             {transition}
}
\BibSpec{collection.article}{%
    +{}  {\PrintAuthors}                {author}
    +{,} { \textit}                     {title}
    +{.} { }                            {part}
    +{:} { \textit}                     {subtitle}
    +{,} { \PrintContributions}         {contribution}
    +{,} { \PrintConference}            {conference}
    +{}  {\PrintBook}                   {book}
    +{,} { }                            {booktitle}
    +{,} { \PrintDateB}                 {date}
    +{,} { pp.~}                        {pages}
    +{,} { }                            {publisher}
    +{,} { }                            {organization}
    +{,} { }                            {address}
    +{,} { }                            {status}
    +{,} { \PrintDOI}                   {doi}
    +{,} { \eprint}        {eprint}
    +{}  { \parenthesize}               {language}
    +{}  { \PrintTranslation}           {translation}
    +{;} { \PrintReprint}               {reprint}
    +{.} { }                            {note}
    +{.} {}                             {transition}
}

\BibSpec{misc}{
  +{}{\PrintAuthors}  {author}
  +{,}{ \textit}     {title}
  +{}{ (}             {date}
  +{),}{ }             {note}
  +{.}{}              {transition}
}

\usepackage{cleveref}
\theoremstyle{definition}
\newtheorem{defn}[equation]{Definition}
\newtheorem{notn}[equation]{Notation}
\newtheorem{assmp}[equation]{Assumption}
\newtheorem{para}[equation]{}
\theoremstyle{plain}
\newtheorem{thm}[equation]{Theorem}
\newtheorem{prp}[equation]{Proposition}
\newtheorem{lem}[equation]{Lemma}
\newtheorem{cor}[equation]{Corollary}

\theoremstyle{remark}
\newtheorem{rmk}[equation]{Remark}
\newtheorem{exmp}[equation]{Example}

\crefname{defn}{Definition}{Definitions}
\crefname{notn}{Notation}{Notations}
\crefname{assmp}{Assumption}{Assumptions}
\crefname{thm}{Theorem}{Theorems}
\crefname{prp}{Proposition}{Propositions}
\crefname{lem}{Lemma}{Lemmas}
\crefname{cor}{Corollary}{Corollaries}
\crefname{conj}{Conjecture}{Conjectures}
\crefname{rmk}{Remark}{Remarks}
\crefname{exmp}{Example}{Examples}
\crefname{section}{Section}{Sections}
\crefname{subsection}{Subsection}{Subsections}
\crefname{para}{}{}
\crefname{appendix}{Appendix}{Appendices}
\crefname{subappendix}{Appendix}{Appendices}

\newcommand{\bB}{\mathbb{B}}
\newcommand{\bC}{\mathbb{C}}
\newcommand{\bD}{\mathbb{D}}

\newcommand{\bK}{\mathbb{K}}

\newcommand{\bM}{\mathbb{M}}
\newcommand{\bN}{\mathbb{N}}

\newcommand{\bR}{\mathbb{R}}
\newcommand{\bS}{\mathbb{S}}
\newcommand{\bT}{\mathbb{T}}

\newcommand{\bZ}{\mathbb{Z}}
\newcommand{\cA}{\mathcal{A}}
\newcommand{\cB}{\mathcal{B}}

\newcommand{\cE}{\mathcal{E}}

\newcommand{\cH}{\mathcal{H}}

\newcommand{\cK}{\mathcal{K}}

\newcommand{\cM}{\mathcal{M}}

\newcommand{\cP}{\mathcal{P}}
\newcommand{\cQ}{\mathcal{Q}}

\newcommand{\cT}{\mathcal{T}}
\newcommand{\cU}{\mathcal{U}}
\newcommand{\cV}{\mathcal{V}}
\newcommand{\cW}{\mathcal{W}}
\newcommand{\cX}{\mathcal{X}}

\newcommand{\fm}{\mathfrak{m}}

\newcommand{\fs}{\mathfrak{s}}

\newcommand{\sH}{\mathscr{H}}

\newcommand{\sM}{\mathscr{M}}
\newcommand{\sN}{\mathscr{N}}

\newcommand{\sS}{\mathscr{S}}

\newcommand{\sW}{\mathscr{W}}
\newcommand{\sX}{\mathscr{X}}

\renewcommand{\Im}{\mathrm{Im} \hspace{0.1em}}
\newcommand{\pt}{\mathrm{pt}}

\newcommand{\del}{\mathrm{del}}
\newcommand{\sgn}{\mathrm{sgn}}
\newcommand{\id}{\mathrm{id}}

\newcommand{\Cl}{\text{\it C$\ell$}}

\DeclareMathOperator{\hotimes}{\hat{\otimes }}
 
\DeclareMathOperator{\Ch}{\mathrm{Ch}} 
\DeclareMathOperator{\K}{\mathrm{K}}
\DeclareMathOperator{\KR}{\mathrm{KR}}
\DeclareMathOperator{\KO}{\mathrm{KO}}

\DeclareMathOperator{\HC}{\mathrm{HC}}
\DeclareMathOperator{\HP}{\mathrm{HP}}
\DeclareMathOperator{\CC}{\mathrm{CC}}
\DeclareMathOperator{\End}{\mathrm{End}}

\DeclareMathOperator{\Hom}{\mathrm{Hom}}
\DeclareMathOperator{\ind}{\mathrm{ind}}

\DeclareMathOperator{\Ad}{\mathrm{Ad}}

\newcommand{\myhat}{\text{\hspace{-0.1em}\raisebox{-0.3ex}[0.8ex][0.0ex]{\textasciicircum}\hspace{-0.1em}}}

\newcommand{\PD}{\mathrm{PD}}

\newcommand{\lwedge}{{\textstyle\bigwedge}}

\newcommand{\Sgn}{\mathop{\mathrm{Sgn}}\nolimits}

\newcommand{\Prop}{\mathop{\mathrm{Prop}}}
\newcommand{\alg}{\mathrm{alg}}
\title{Codimension 2 transfer of higher index invariants}
\author{Yosuke kubota}
\address{Department of Mathematical Sciences, Shinshu University\\ 3-1-1 Asahi, Matsumoto, Nagano, 390-8621, Japan\\ and \\ RIKEN iTHEMS \\ 2-1 Hirosawa, Wako, Saitama, 351-0198, Japan}
\email{ykubota@shinshu-u.ac.jp}
\date{}
\begin{document}
\maketitle
\begin{abstract}
This paper is devoted to the study of the higher index theory of codimension $2$ submanifolds originated by Gromov--Lawson and Hanke--Pape--Schick. 
The first main result is to construct the `codimension $2$ transfer' map from the Higson--Roe analytic surgery exact sequence of a manifold $M$ to that of its codimension $2$ submanifold $N$ under some assumptions on homotopy groups. 
This map sends the primary and secondary higher index invariants of $M$ to those of $N$. 
The second is to establish that the codimension 2 transfer map is adjoint to the co-transfer map in cyclic cohomology, defined by the cup product with a group cocycle. 
This relates the Connes--Moscovici higher index pairing and Lott's higher $\rho$-number of $M$ with those of $N$. 
\end{abstract}
\tableofcontents 
\section{Introduction}
Higher index theory provides two differential topological invariants of manifolds, the (C*-algebraic) \emph{higher signature} and the \emph{Rosenberg index}, both taking value in the (Real) K-group of the group C*-algebra of the fundamental group. 
The higher signature of an oriented manifold is related to the rational Pontrjagin class and plays a key role for proving the Novikov conjecture for a large class of groups \cites{kasparovEquivariantKKTheory1988}. 
The Rosenberg index of a spin manifold is an obstruction to positive scalar curvature (psc) metric \cite{rosenbergAstAlgebrasPositive1983}. It is known to be a nearly complete obstruction \cite{rosenbergStableVersionGromovLawson1995}, but is not really complete due to a counterexample \cite{schickCounterexampleUnstableGromovLawsonRosenberg1998}. 
Moreover, these higher indices have the corresponding secondary invariants, the \emph{higher $\rho$-invariants} \cites{higsonMappingSurgeryAnalysis2005b,higsonMappingSurgeryAnalysis2005a,higsonMappingSurgeryAnalysis2005,piazzaRhoclassesIndexTheory2014}, both taking value in the (Real) K-group of the pseudo-local coarse C*-algebra. 
These invariants distinguish different $h$-cobordism classes of oriented homotopy equivalences of manifolds and psc cobordism classes of psc metrics respectively. 

The focus of this paper is in the codimension 2 index obstruction to positive scalar curvature, which is originated from the early work by Gromov-Lawson \cite{gromovPositiveScalarCurvature1983}*{Theorem 7.5} and is refined later by Hanke-Pape-Schick \cite{hankeCodimensionTwoIndex2015}.
Let $M$ be a closed spin manifold, let $N$ be its codimension $2$ submanifold, and let  $\Gamma := \pi_1(M)$ and $\pi := \pi_1(N)$.
According to \cite{hankeCodimensionTwoIndex2015}*{Theorem 1.1}, if $M$ and $N$ satisfies  some assumptions on homotopy groups (listed in \cref{prp:NSZ}), and if the Rosenberg index $\alpha_{\pi} (N)$ vanishes, then $M$ does not admit any psc metric. 
The corresponding result in higher signature is established by Higson--Schick--Xie \cite{higsonCalgebraicHigherSignatures2018}.

In recent researches it has turned out that this new obstruction does not go beyond the standard one, the Rosenberg index of $M$ itself. 
Namely, the non-vanishing of $\alpha_\pi(N)$ implies the non-vanishing of $\alpha_\Gamma (M)$. 
In \cites{nitscheNewTopologicalIndextheoretical2018,nitscheTransferMapsGeneralized2019}, Nitsche, Schick and Zeidler suppose a framework for proving this; the \emph{codimension $2$ transfer map}.  
They have constructed a homomorphism between K-homology groups of the classifying spaces
\begin{align}
    \tau_\sigma \colon \K_*(B\Gamma )\to \K_{*-2}(B\pi ), \label{eq:NSZ}
\end{align}  
which relates the K-homology classes of the Dirac operator of $M$ with that of $N$. 
Following this, in \cite{kubotaGromovLawsonCodimensionObstruction2020}, Schick and the author construct a homomorphism between group C*-algebra K-theory  
\begin{align}
    \tau_\sigma \colon \K_*(C^*\Gamma  )\to \K_{*-2}(C^*\pi  ), \label{eq:KS}
\end{align}  
based on the index pairing with a flat bundle of Calkin algebras. 
It is proved in \cite{kubotaGromovLawsonCodimensionObstruction2020}*{Theorems 1.2, 1.3} that $\tau_\sigma$ sends the Rosenberg index of $M$ to that of $N$  (which recovers the Hanke--Pape--Schick theorem) and the higher signature of $M$ to the twice of that of $N$.
On the other hand, a relation of maps \eqref{eq:NSZ} and \eqref{eq:KS} was not clarified.

The purpose of this paper is to generalize and provide a unified view of the codimension $2$ transfer maps of higher index invariants.
The first main theorem is the construction of the codimension 2 transfer map from the \emph{Higson--Roe analytic surgery exact sequence} of $M$, i.e., the long exact sequence in $\K$-theory associated to the short exact sequence of coarse C*-algebras $0 \to C^*(\widetilde{M})^\Gamma \to D^*(\widetilde{M})^\Gamma \to Q^*(\widetilde{M})^\Gamma \to 0$ (where $\widetilde{M}$ denotes the universal covering of $M$), to that of $N$.  
Moreover, it relates the primary and the secondary higher index invariants of $M$ with those of $N$ in the following sense.
\begin{thm}\label{thm:main1}
There is a homomorphism of long exact sequences
\[\mathclap{
\xymatrix@C=0.5cm{
\cdots  \ar[r] & \K_*(C^*(\widetilde{M})^\Gamma) \ar[r] \ar[d]^{\tau_\sigma} & \K_*(D^*(\widetilde{M})^\Gamma) \ar[r] \ar[d]^{\tau_\sigma} & \K_*(Q^*(\widetilde{M})^\Gamma) \ar[r] \ar[d]^{\tau_\sigma} & \K_{*-1}(C^*(\widetilde{M})^\Gamma ) \ar[d]^{\tau_\sigma} \ar[r] & \cdots \\
\cdots \ar[r] & \K_{*-2}(C^*(\widetilde{N})^\pi ) \ar[r] & \K_{*-2}(D^*(\widetilde{N})^\pi ) \ar[r] &  \K_{*-2}(Q^*(\widetilde{N})^\pi ) \ar[r] & \K_{*-3}(C^*(\widetilde{N})^\pi) \ar[r] & \cdots .
}}
\]
Moreover, the following hold:
\begin{enumerate}
    \item The map $\tau_\sigma \colon \K_*(C^*(\widetilde{M})^\Gamma )\to \K_{*-2}(C^*(\widetilde{N})^\pi)$ coincides with \eqref{eq:KS}.
    \item The map $\tau_\sigma\colon \K_*(Q^*(\widetilde{M})^\Gamma )\to \K_{*-2}(Q^*(\widetilde{N})^\pi)$ coincides with \eqref{eq:NSZ}.
    \item The equalities $\tau_\sigma(\rho(g_M)) = \rho(g_N)$ and $\tau_\sigma(\rho_{\mathrm{sgn}}(f_M)) = 2\rho_{\mathrm{sgn}}(f_N)$ hold.
\end{enumerate}
\end{thm}
Here, in (3), $g_M$ is a psc metric on $M$ which is of the form $g_{\bD^2} + g_N$ on a tubular neighborhood $U \cong N \times \bD^2$ of $N $, and $f_M \colon M' \to M$ is an oriented homotopy equivalence whose restriction to $f^{-1}(N)$, denoted by $f_N$, is also an oriented homotopy equivalence. 
They associate the higher $\rho$-invariants $\rho(g_M)$, $\rho(g_N)$, $\rho_\sgn (f_M) $ and $\rho_\sgn (f_N)$ (a detailed definition is reviewed in \cref{para:psc,para:homotopy}).  
We remark that  (1), (2) and the commutativity of the diagram reproves \cite{kubotaGromovLawsonCodimensionObstruction2020}*{Theorems 1.2, 1.3}.

The proof, given in \cref{thm:HigsonRoe,thm:second_tr}, consists of two steps. 
First, we construct a $\ast$-homomorphism lifting $\Gamma$-invariant operators on $\widetilde{M}$ onto the $\Pi$-covering $\overline{\sM}$ of $\sM :=(\widetilde{M}/\pi) \setminus (N \times \bD^2)$. This covering is constructed in \cite{hankeCodimensionTwoIndex2015}*{Theorem 4.3} and plays a key role for proving their main result. 
The lifting of operators to a covering space does not form a $\ast$-homomorphism of coarse C*-algebras in general, but in our setting it makes sense modulo the boundary, i.e., as a $\ast$-homomorphism to the quotient of coarse C*-algebras $C^*(\overline{\sM})^\Pi / C^*(\overline{\sN} \subset \overline{\sM})^\Pi$. 
Second, we apply the `boundary of Dirac is Dirac' and `boundary of signature is $2^\epsilon$ times signature' principles to the manifold with boundary $\overline{\sM}$. Since $\overline{\sN}$ is diffeomorphic to $\widetilde{N} \times \bR$, the boundary map in $\K$-theory sends the Dirac operator on $\overline{\sM}$ to that of $\widetilde{N} \times \bR$, which is identified with the Dirac operator on $\widetilde{N}$ by the partitioned manifold index theorem \cite{roePartitioningNoncompactManifolds1988}.

The second main theorem relates the codimension 2 transfer map and the pairing of higher index invariants with cyclic cocycles of the group algebra. 
The cyclic cohomology group of $\bC[\Gamma]$ is described in terms of the group cohomology by Burghelea~\cite{burgheleaCyclicHomologyGroup1985}. 
More specifically, $\HC^*(\bC[\Gamma])$ decomposes into the direct product of the `localized' and the `delocalized' parts, where the former is isomorphic to the cohomology of $\Gamma$ itself and the latter is isomorphic to the product of cohomology of normalizer subgroups of $\Gamma $. 
If the group $\Gamma$ is hyperbolic, then a cyclic cocycle on $\bC[\Gamma]$ induces a homomorphism $\K_*(C^*\Gamma ) \to \bC$ \cites{jolissaintKtheoryReducedAlgebras1989,puschniggNewHolomorphicallyClosed2010}. 
It plays a key role in the proof of the Novikov conjecture for hyperbolic groups by Connes-Moscovici \cite{connesCyclicCohomologyNovikov1990}. 
There is also a secondary analog of this pairing, Lott's \emph{higher $\rho$-number} or the \emph{delocalized $\eta$-invariant}~\cite{lottHigherEtainvariants1992}, in which there has been a growing interest in recent researches such as Chen--Wang--Xie--Yu~\cite{chenDelocalizedEtaInvariants2019} and Piazza--Schick--Zenobi~\cite{piazzaMappingAnalyticSurgery2019}. 

The codimension 2 co-transfer map of cohomology groups is defined in a dual way as \eqref{eq:NSZ}.
With the language of group cohomology, this is realized by the cup product with a second cohomology class $\sigma \in H^2(\Gamma ; \bZ[\Gamma /\pi])$ (by this reason it is written as $\sigma \cdot \phi$).
This $\sigma$ is an essential ingredient of the construction of codimension $2$ transfer maps in homology, studied in \cref{section:cocycle}. 
We show that this co-transfer map is adjoint to $\tau_\sigma$ up to the constant $2\pi i$ with respect to the pairing of K-theory and cyclic cohomology. For the definition of numerical higher indices $\alpha_\phi (M)$, $\Sgn_\phi (M)$ and higher $\rho$-numbers $\varrho_\phi (g_M)$, $\varrho_\sgn (f_M)$, see \cref{defn:CM_Lott}. 
\begin{thm}\label{thm:main2}
Let $\Gamma$ be a hyperbolic group and let $\pi$ be its hyperbolic subgroup. 
Let $\phi \in \HC^{q-2}_{e}(\bC[\pi] )$ and $\psi \in \HC^{q-2}_{\del}(\bC[\pi] )$. 
Let $M$ be a closed manifold and let $N$ be a codimension $2$ submanifold of $M$ satisfying (1), (2), (3) of \cref{para:cocycle}. The following equalities hold:
\begin{enumerate}
\item $2\pi i \alpha_{\phi \cdot \sigma} (M) =  \alpha_\phi (N)$ and $\pi i   \Sgn\nolimits_{\phi \cdot  \sigma}(M) = \Sgn\nolimits_\phi(N)$.
\item $2\pi i  \varrho_{\phi \cdot \sigma} (g_M) =  \varrho_\phi (g_N)$ and $\pi i  \varrho_{\phi \cdot \sigma}^\sgn (f_M) =  \varrho_\phi^\sgn (f_N)$.
\end{enumerate}
\end{thm}
The proof is given in \cref{thm:pair_tranfer}. There are two ingredients of the proof. 
The first is to identify the codimension 2 co-transfer map of cyclic cohomology groups with the cup product with $\sigma$. This is performed on the basis of the identification of $\sigma $ with the Dixmier--Douady class of a twist (central extension) of the action groupoid $\cB \rtimes \Gamma$, where $\cB$ is a bouquet of circles studied in \cref{section:Cstar}. 
The second is to define the codimension $2$ transfer map in terms of unconditional Banach algebras. 
This enables us to make the codimension $2$ transfer map compatible with the `mapping analytic surgery to homology' formalism of the higher index pairing developed by Piazza--Schick--Zenobi~\cite{piazzaMappingAnalyticSurgery2019}.
We remark that this theorem produces a class of non-trivial computations of the higher $\rho$-numbers.

This paper is organized as follows. 
In \cref{section:cocycle}, we review the construction of the codimension $2$ transfer map in general homology theory and discuss its relation with an extension of groups. 
In \cref{section:Cstar}, we revisit the construction of the C*-algebraic transfer map in \cite{kubotaGromovLawsonCodimensionObstruction2020} from the viewpoint of groupoid cocycles. 
Moreover, we also give a coarse geometric view of this construction.
In \cref{section:HigsonRoe}, we prove our first main theorem, \cref{thm:main1}. We first construct a map between Higson--Roe analytic surgery sequences, and next show that these maps relate the higher index invariants of $M$ to those of $N$.
In \cref{section:cyclic}, we prove our second main theorem, \cref{thm:main2}. We first study the codimension $2$ co-transfer of cyclic cohomology groups of $\bC[\Gamma]$, and then extend it to the unconditional Banach algebra $\cA \Gamma$.

\section{Codimension 2 transfer map via a second cohomology class}\label{section:cocycle}
In this section, we review the (codimension 2) submanifold transfer map of general homology theory introduced by Nitsche, Schick, and Zeidler~\cites{nitscheNewTopologicalIndextheoretical2018,nitscheTransferMapsGeneralized2019}. 
We slightly rearrange the exposition on the basis of a second cohomology class.
This provides us a systematic construction of codimension 2 inclusions of manifolds to which the theory is applied. 
We also discuss a realization of this second cohomology class by an extension of fundamental groups. 

\subsection{Setting}\label{subsection:setting}
Let $\Gamma$ be a finitely presented discrete group and let $\pi$ be its finitely presented subgroup. 
\begin{defn}
We define the `compactly supported' cohomology group of the topological space $E \Gamma /\pi $ as the reduced cohomology group of the pointed space
\[ (E\Gamma /\pi)^\dagger := (E\Gamma \times_\Gamma (\Gamma /\pi)^+) / (E\Gamma \times_\Gamma \{ \ast \} ),  \]
where $(\Gamma /\pi )^+$ denotes the $1$-point compactification $E\Gamma /\pi \sqcup \{ \ast \} $. 
In the same way, we also define the group $h^n_c(E\Gamma /\pi):=h^n((E\Gamma /\pi)^\dagger )$ for a general cohomology theory $h^*$.
\end{defn}
This coincides with the compactly supported cohomology in the usual sense if and only if $B\Gamma$ is compact. 
To be more precise, this group should be called the cohomology with fiberwisely compact support over $B\Gamma$, but in this paper we just call it, in short, the compactly supported cohomology.

We represent a second cohomology class $\sigma \in  H^2_c(E\Gamma/\pi; \bZ )$ in several ways. First, it is represented by a continuous map
\[ F_\sigma \colon (E\Gamma /\pi )^\dagger  \to K(\bZ, 2). \]
Second, it is represented by a compactly supported complex line bundle $L$ over $E\Gamma /\pi$, i.e., a line bundle over $(E\Gamma /\pi)^\dagger$. 
This bundle is related with $F_\sigma$ as $L:=F_\sigma^*(\widetilde{L})$, where $\widetilde{L} \to BU(1) =K(\bZ,2)$ is the universal line bundle. 
If $B\Gamma$ is modeled by a closed aspherical manifold, then the zero locus of a generic section $s \colon E\Gamma/\pi \to L$ is a closed codimension 2 submanifold $N \subset E\Gamma /\pi$, which represents $\sigma $ through the Poincar\'e duality.
Third, this cohomology class is represented by an extension of groups, which is discussed later in \cref{subsection:cocycle_ext}. 

We impose several assumptions on this cohomology class $\sigma$. 
Let $q \colon E\Gamma /\pi \to B\Gamma$ denote the projection and let $j \colon E\Gamma /\pi  \to (E\Gamma /\pi)^\dagger $ denote the inclusion, which induces the homomorphism
\[ j^* \colon H^2_c(E\Gamma /\pi ; \bZ) \to H^2(E\Gamma /\pi ; \bZ) \cong H^2(B\pi ; \bZ).\]
\begin{assmp}\label{assump:cocycle}
We consider the following assumptions:
\begin{enumerate}
\item[(A1)] There is an open subset $U_\sigma \subset E\Gamma /\pi$ such that $q |_{U_\sigma}$ is injective and $\sigma $ is contained in the image of the map $ H^2_c(U_\sigma;\bZ) \to H^2_c(E\Gamma /\pi; \bZ)$ induced from the inclusion. 
\item[(A2)] The equality $j^*(\sigma) =0$ holds. 
\end{enumerate}
By abuse of notation, when $\sigma$ satisfies (A1), we use the same letter $\sigma$ for a choice of its preimage in $H_c^2(U_\sigma;\bZ)$.
\end{assmp}

There is an immediate implication of (A2). 
Let $\mu \colon BU(1) \to BSU(2) $ denote the continuous map induced from the group homomorphism $U(1) \to SU(2)$ given by $z \mapsto \mathrm{diag}(z,\bar{z})$.

\begin{lem}\label{lem:SU2}
Let $\sigma$ satisfy (A2) of \cref{assump:cocycle}. Then the following holds;
\begin{enumerate}
    \item[(A2')] the composition $\mu \circ F_\sigma \colon (E\Gamma /\pi)^\dagger  \to BSU(2)$ is null-homotopic. 
\end{enumerate}
Moreover, if $\sigma$ also satisfy (A1) of \cref{assump:cocycle}, then $\mu \circ F_\sigma$ is null-homotopic as a map from the $1$-point compactification $U_\sigma^+$. 
\end{lem}
\begin{proof}
We may assume that $B\Gamma$ is modeled by a locally compact Hausdorff space. 
Let $L$ be the line bundle over $(E\Gamma/\pi)^\dagger$ representing $\sigma$ as above. 
Let $s \colon E\Gamma /\pi \to L$ be a continuous section such that $V:=\{ x \in (E\Gamma /\pi)^\dagger \mid  |s(x)|<1 \}$ is relatively compact. 
If we additionally assume (A1), we choose $s$ in the way that $|s(x)| \geq  1$ for any $x\in U_\sigma^c$.  

Since $\sigma|_V = j^*(\sigma)|_V$ is trivial, there is a non-vanishing section $t \colon V \to L^*|_V$. Set $\widetilde{s}:= \chi \cdot (s \oplus 0) + (1-\chi ) \cdot (0 \oplus t)$, where $\chi \colon E\Gamma /\pi \to [0,1]$ is a bump function satisfying  $\mathop{\mathrm{supp}}( \chi ) \subset V$ and $\chi \equiv 1$ on $\{ x \in B\Gamma \mid  |s(x)| < 1/2 \}$. This is a non-vanishing section of the $SU(2)$-bundle $L \oplus L^*$.  Therefore, the map $\mu \circ F_\sigma $ is null-homotopic. 
Moreover, if we additionally assume (A1), the section $(s,t)$ is an extension of $(s,0)$ on $U^c_\sigma $, and hence a null-homotopy of $\mu \circ F_\sigma$ is taken in $\mathop{\mathrm{Map}}(U_\sigma^+, BSU(2))$.
\end{proof}

\begin{exmp}\label{exmp:bundle}
Let $\Sigma_{g,1}$ denote the closed oriented surface of genus $g \geq 1$ and a single boundary. 
Let $\mathop{\mathrm{Diff}}(\Sigma_{g,1}, \partial)$ denote the group of diffeomorphisms on $\Sigma_{g,1}$ fixing the boundary, which is regarded as a subgroup of $\mathop{\mathrm{Diff}} (\Sigma_g)$.  
Let $\pi$ be a finitely presented group with a homomorphism $\pi \to \pi_0(\mathop{\mathrm{Diff}} (\Sigma_{g,1}, \partial))$.  
This induces an action $\pi \curvearrowright \pi_1(\Sigma_{g})$. Set $\Gamma := \pi_1(\Sigma_{g}) \rtimes \pi$. 
The homomorphism $\pi \to \pi_0(\mathop{\mathrm{Diff}} (\Sigma_{g,1}, \partial))$ also induces a fiber bundle $\cE \to B\pi $ with the fiber $\Sigma_g$ since the identity component of $\mathop{\mathrm{Diff}} (\Sigma_{g,1}, \partial)$ is contractible. 
Note that the total space $\cE$ is a model of $B\Gamma$. 
Moreover, by the construction, this bundle has a section $B\pi \to \cE$ such that the neighborhood of its image is isomorphic to $B\pi \times \bD^2$. 
Now, the space $E\Gamma /\pi$ is a $\widetilde{\Sigma}_g$-bundle over $B\pi$ including a copy of $B\pi \times \bD^2$. The Poincar\'{e} dual of $B\pi \subset B\pi \times \bD^2 \subset E\Gamma /\pi$ determines a second cohomology class $\sigma \in H^2_c(E\Gamma /\pi;\bZ)$. This $\sigma$ satisfies \cref{assump:cocycle}. 
\end{exmp}

\subsection{Codimension 2 transfer map}
The codimension 2 transfer map of general homology theories, constructed by Nitsche, Schick and Zeidler \cites{nitscheNewTopologicalIndextheoretical2018,nitscheTransferMapsGeneralized2019}, is a homomorphism 
\[ \tau _\sigma \colon h_*(B\Gamma) \to h_{*-2}(B\pi). \]
Here $h_*$ is a general homology theory satisfying the following nice property. 
We say that a generalized homology theory $h_*$ has \emph{lf-lifting} if, for any covering $\overline{Y} \to Y$ and any open subspace $U \subset \overline{Y}$ such that the restriction of the projection $U \to Y$ is a proper map, there is a homomorphism
\[ \nu_{\overline{Y},U} \colon  h_*(Y) \to h_*(\overline{Y} , \overline{Y} \setminus U) \]
such that, if there is another covering $\widetilde{Y} \to Y$, open subspace $V \subset \widetilde{Y}$ and a covering map $F \colon \widetilde{Y} \to \bar{Y}$ such that $F|_U$ is injective and $F(U) \supset V$, then $F_* \circ \nu_{\overline{Y},U} = \nu_{\widetilde{Y},V}$ holds.
This is a slight modification of the notion of locally-finite restriction introduced by Nitsche (see e.g.~\cite{nitscheTransferMapsGeneralized2019}*{Definition 4.2}), which works even if $h_*$ does not have equivariant theory but the target is restricted to Galois covering spaces (i.e., free and proper actions).
For example, ordinary homology theory, cobordism theories, real and complex K-theories have lf-lifting.

Let $h^* $ be a multiplicative generalized cohomology theory which is complex oriented.
Let $c_1^h \in h^2(K(\bZ,2))$ denote the complex orientation. 
By abuse of notation, we use the same letter $\sigma $ for the second $h$-cohomology class $F_\sigma^*(\xi) \in h^2_c(E\Gamma /\pi)$. 
We also define $\sigma \in h^2_c(E\Gamma /\pi)$ for a generalized homology theory $h^*$ which is not necessarily complex oriented, like spin cobordism theory or $\KO$-theory. 
For this sake, additionally we assume that $\sigma$ satisfies (A2') in \cref{lem:SU2}. 
\begin{lem}
Assume $\sigma$ satisfies (A2') in \cref{lem:SU2}. Then there is a lifting of $F_\sigma$ as
\[
\xymatrix{
& S^2 \ar[d] \\
(E\Gamma /\pi)^\dagger \ar[r]^{F_\sigma} \ar@{.>}[ru]^{\widetilde{F}_\sigma} & BU(1).
}
\]
\end{lem}
\begin{proof}
This is obvious since $S^2 \cong SU(2)/U(1) \to BU(1) \to BSU(2)$ is a fibration. 
\end{proof}
Let $1_h \in h^2(S^2 , \ast) \cong h^0$ denote the ring unit. 
In the same way as above, we use the same letter $\sigma$ for the second $h$-cohomology class $\widetilde{F}_\sigma^*(1_h) \in h^2_c(E\Gamma /\pi)$. 

Now we are ready to state the definition of the codimension 2 transfer map. 
We mention that, by replacing $F_{\sigma} $ if necessary, we may assume that the restriction of the projection $E\Gamma /\pi \to B\Gamma$ onto the support of $F_\sigma$, i.e., the open subset $V:= F_\sigma^{-1}(BU(1) \setminus \{ \ast \} ) \subset E\Gamma /\pi$, is a relatively proper map. 
Hence $\sigma \in h_c^2(E\Gamma /\pi)$ is the image of an $h$-cohomology class $\sigma_V \in h^2(E\Gamma / \pi , (E\Gamma /\pi) \setminus V )$.  
\begin{defn}\label{defn:tau}
Let $h_*$ be a multiplicative general homology theory having lf-lifting and let $\sigma \in H^2_c(E\Gamma /\pi ; \bZ)$. Assume that either $h_*$ is complex oriented, or $\sigma$ satisfies (A2') in \cref{lem:SU2}. 
Then the homomorphism $\tau_\sigma \colon h_*(B\Gamma) \to h_*(B\pi)$ is defined as the composition
\[\tau_\sigma \colon h_*(B\Gamma) \xrightarrow{\nu_{E\Gamma/\pi, V}} h_*(E\Gamma /\pi , (E\Gamma /\pi) \setminus V)  \xrightarrow{\sigma_V \cap {\cdot} } h_{*-2}( E\Gamma /\pi ) \cong h_{*-2}(B\pi), \]
where $V \subset E\Gamma /\pi$ is the open subset as above. 
\end{defn}
This definition is independent of the choice of $V$ and $\sigma_V$ because, if there is another choice $V'$ and $\sigma_{V'}$, there is $V'' \subset E\Gamma /\pi $ which includes $V \cup V'$ and the image of $\sigma_V$ and $\sigma_{V'}$ coincides in $H^2_c(V'';\bZ)$. 
\begin{para}\label{rmk:intersection}
The map $\tau_\sigma$ is understood in terms of intersection theory in the following way. Let $M$ be a closed manifold with $\pi_1(M) \cong \Gamma$ and let $\xi_M \colon M \to B\Gamma $ denote the classifying map of the universal covering $\widetilde{M}$. Then $\xi_M^*(\sigma) \in H^2_c(\widetilde{M}/\pi; \bZ)$ is represented by a closed submanifold $N \subset \widetilde{M}/\pi$ of codimension $2$.   
Let $U \subset \widetilde{M}$ denote the tubular neighborhood of $N$. In the same way as \cref{defn:tau}, we also define 
\[\tau_\sigma^{M,N} \colon h_*(M) \to h_*(\widetilde{M}/\pi , \widetilde{M}/\pi \setminus U) \xrightarrow{\xi_M^*(\sigma) \cap {\cdot} }  h_{*-2}(U) \cong  h_{*-2}(N). \]
By definition the diagram
\[
\xymatrix{
h_*(M) \ar[r]^{\tau_\sigma ^{M,N}} \ar[d]^{(\xi_{M})_*} & h_{*-2}(N) \ar[d]^{(\xi_{N})_*} \\
h_*(B\Gamma) \ar[r]^{\tau_\sigma } & h_{*-2}(B\pi )
}
\]
commutes, where $\xi_N$ denotes the classifying map of the universal covering of $N$. 

The map $\tau_\sigma^{M,N}$ is the intersection with $N$, taken in the covering space $\widetilde{M}/\pi$ instead of $M$. In particular, if $M$ is $h_*$-oriented, the fundamental class $[M]$ is sent to the fundamental class $[N]$. 
If we assume (A1) of \cref{assump:cocycle}, then we may choose $N$ as a submanifold of $\xi_M^{-1}(U_\sigma)$, and hence there is a submanifold $N \subset M$ which lifts to its copy in $\widetilde{M}/\pi$. 
In this case $\tau_\sigma^{M,N}$ is the intersection with $[N] \in H^2(M;\bZ)$. 

Let $\iota \colon N \to M$ denote the inclusion. Then $\iota^* (\sigma) = c_1(L|_N) = c_1(\nu N)$. 
Now, the additional assumption (A2) corresponds to the triviality of the normal bundle $\nu N$, since we have $\iota^*(\sigma) = \iota^*(j^* (\sigma )) =0$. 
\end{para}

\subsection{Transfer map via codimension 2 submanifold}
Here we relate the transfer map $\tau_\sigma$ with the topology of codimension 2 submanifolds  studied in \cites{gromovPositiveScalarCurvature1983,hankeCodimensionTwoIndex2015}. 
Let $M$ be a closed manifold and let $N$ be a codimension $2$ submanifold of $M$. 
Set $\Gamma=\pi_1(M)$ and $\pi :=\pi_1(N)$. 
Assume that the homomorphism $\pi \to \Gamma $ induced from the inclusion $N \subset M$ is injective. 
Let $p \colon \widetilde{M}/\pi \to M$ denote the covering map, where $\widetilde{M}$ denotes the universal covering space. 
The the connected component of the inverse image $p^{-1}(N)$ corresponds in one-to-one to the double coset space $\pi \setminus \Gamma /\pi$, and the unit $\pi e \pi$ corresnponds to a copy of $N$ in $\widetilde{M}/\pi$. 
Therefore, the Poincar\'e dual $\mathrm{PD}[N]$ of $N$ determines second cohomology classes in both $H^2(M;\bZ) $ and $H^2_c(\widetilde{M} ; \bZ)$.

.

The following proposition is a rephrasing of a part of \cite{nitscheNewTopologicalIndextheoretical2018}*{Theorem 5.3.6} and \cite{nitscheTransferMapsGeneralized2019}*{Theorem 5.1}. 
\begin{prp}\label{prp:NSZ}
Let $M$ be a closed manifold and let $N \subset M$ be a codimension 2 closed submanifold. Let $\Gamma :=\pi_1(M)$, $\pi:=\pi_1(N)$. Assume that 
\begin{enumerate}
    \item $\pi \to \Gamma$ is injective,
    \item $\pi_2(N) \to \pi_2(M)$ is surjective, and
    \item the normal bundle $\nu N$ is trivial. 
\end{enumerate}
Then there is a cohomology class $\sigma = \sigma_{M,N} \in H^2_c(E\Gamma/\pi ;\bZ )$ which satisfies \cref{assump:cocycle} and the pull--back $\xi_M^*(\sigma)$ by the classifying map $\xi_M \colon M \to B\Gamma $ coincides with the Poincar\'{e} dual $\PD [N] \in H^2_c(\widetilde{M}/\pi ;\bZ)$.
\end{prp}

The proof of this proposition relies on the following key observation by Hanke--Pape--Schick, which is proved in the proof of {\cite{hankeCodimensionTwoIndex2015}*{Theorem 4.3}}. This claim is an essential `geometric' ingredient of the codimension 2 obstruction theory. 

\begin{lem}[{\cite{hankeCodimensionTwoIndex2015}}]\label{lem:HPS}
Let $M$ and $N$ satisfy (1), (2), (3) of \cref{prp:NSZ}. Let $U$ denote the tubular neighborhood of $N$, which is homeomorphic to $N \times \bD^2$. Then the homomorphism 
$\pi_1(\partial U) \to \pi_1((\widetilde{M}/\pi) \setminus U)$
is split injective. 
\end{lem}

\begin{proof}[{Proof of \cref{prp:NSZ}}]
The cohomology class $\sigma_{M,N}$ is constructed in  \cites{nitscheNewTopologicalIndextheoretical2018,nitscheTransferMapsGeneralized2019}. 
In order to check that $\sigma_{M,N}$ satisfies \cref{assump:cocycle}, we recall their construction. 
By \cref{lem:HPS}, there is a splitting $\pi_1(\widetilde{M}/\pi \setminus U) \to \bZ$, which induces $F \colon (\widetilde{M}/\pi) \setminus U \to  S^1$. 
This $F $ extends to 
\[
F \colon (\widetilde{M}/\pi , (\widetilde{M}/\pi) \setminus \overline{U} ) \to (\bD^2 , \partial \bD^2).
\]
A model of $B\Gamma$ is obtained by attaching $k$-cells $\{ D_i\}_{i \in \bN}$, for $k\geq 3$, to $M$. 
Corresponding to this, we obtain a model of $E\Gamma /\pi$ by attaching $\Gamma /\pi$ copies of cells $\{ D_{i,g\pi}\}_{i \in \bN, g\pi \in \Gamma /\pi}$ to $\widetilde{M}/\pi$. 
Let $X_0:=\widetilde{M}/\pi$ and we inductively define the CW-complex $X_i$ to be the one obtained by attaching $\{D_{i,g\pi}\}_{g\pi \in \Gamma /\pi}$ to $X_{i-1}$. Let $K_0:=\overline{U}$ and let $K_i$ be the subcomplex obtained by attaching to $K_{i-1}$ the (finite number of) $D_{i,g\pi}$'s whose boundary intersects with $K_{i-1}$. 
We extends the maps $F_{i} \colon (X_{i}, X_{i} \setminus K_{i}) \to (\bD^2, \partial \bD^2)$ inductively, and define $F \colon (E\Gamma/\pi , E\Gamma/\pi \setminus K) \to (\bD^2, \partial \bD^2)$ as their limit. Now the pull-back of the generator of $H^2(\bD^2, \partial \bD^2;\bZ)$ is the desired cohomology class.

When extending $F_{i-1}$ to $F_i$, it is chosen such that $q|_{V_i}$ is injective, where  $V_i$ denotes the inverse image of the closure of $B_{\varepsilon}(0) \subset \bD^2$. 
By setting $U_\sigma:= F^{-1}(B_\varepsilon (0))$, this shows (A1).
Indeed, since $p(V_{i-1}) \cap \partial D_i$ is identified with the disjoint union of $\partial D_{i,g\pi} \cap p(V_{i-1})$, there is an open neighborhood $U_{i,g\pi}$ of each $\partial D_{i,g\pi} \cap p(V_{i-1})$ in $D_i$ which are mutually disjoint. 
We can choose an extension $F_i$ on $D_{i,g\pi}$ in the way that $F_i (U_{i,g\pi}^c) \subset B_\varepsilon(0)^c$, which implies the injectivity of  $p|_{V_i}$ as desired.

Let $\xi_N \colon N \to B\pi$ denote the classifying map. 
The assumption (3) implies $c_1(\nu N) =\xi_N^{*}(\sigma ) = 0$. 
Now (A2) follows from the injectivity of $\xi_N^{*}$ in the second cohomology groups, which follows from $2$-connectedness of $\xi_N$. 
\end{proof}

\begin{notn}\label{notn:mfds}
Let $M$ be a closed manifold and let $N \subset M$ be a codimension 2 submanifold of $M$ satisfying (1), (2), (3) of \cref{prp:NSZ}. Let $U$ be a tubular neighborhood of $N$. 
Let $\widetilde{M}$ denote the universal covering of $M$. By the assumption of fundamental groups, there is a copy of $U$ in the $\Gamma/\pi$-covering $\widetilde{M}/\pi$ over $M$. Let $\sM$ denote the complement of $U$. Let $\widetilde{\sM}$ denote its $\pi$-covering, which is a manifold with the boundary and is a closed subset of $\widetilde{M}$. We write $\overline{\sM}$ for the $\bZ$-covering of $\widetilde{\sM}$ associated to the splitting given in \cref{lem:HPS}: 
\[\xymatrix@R=1.5em{
& \widetilde{M} \ar[r]^\pi  \ar@{}[d]|{\cup} &\widetilde{M}/\pi \ar[r]^{\Gamma/\pi} \ar@{}[d]|{\cup} & M \\
\overline{\sM} \ar[r]^\bZ & \widetilde{\sM} \ar[r]^\pi & \sM. & 
}
\]
We write $\sN$, $\widetilde{\sN}$, $\overline{\sN}$ for the boundary of $\sM$, $\widetilde{\sM}$, $\overline{\sM}$ respectively. Note that $\sN \cong N \times S^1$, $\widetilde{\sN} \cong \widetilde{N} \times S^1$ and $\overline{\sN} \cong \widetilde{\sN} \times \bR$. 
We also write $M^\circ$ for the complement of $U$. 
\end{notn}

\subsection{Extension of fundamental groups}\label{subsection:cocycle_ext}
The cohomology group $H^*_c(E\Gamma /\pi ; \bZ)$ is identified with the group cohomology $H^*(\Gamma ; \bZ[\Gamma /\pi])$. 
Here, $\bZ[\Gamma /\pi]$ denote the free abelian group of finitely supported $\bZ$-valued functions on $\Gamma /\pi$, in other words, the direct sum of $\Gamma /\pi$ copies of $\bZ$, whose $\Gamma$-module structure is induced from the regular $\Gamma$-action on $\Gamma /\pi$. 
To be precise, this identification is given in the following way. 
\begin{lem}
The groups $H^*_c(E\Gamma /\pi ; \bZ )$ and $H^*(\Gamma ; \bZ[\Gamma /\pi])$ are isomorphic.  
\end{lem}
\begin{proof}
We consider the relative version of the Serre spectral sequence for the pair of Serre fibrations $(E\Gamma \times _\Gamma (\Gamma /\pi)^+ ,  E\Gamma \times_\Gamma \{ \ast \})$ over $B\Gamma $ (cf.~\cite{mcclearyUserGuideSpectral1985}*{Exercise 5.6}). 
It is a spectral sequence converging to $H^2_c(E\Gamma /\pi ; \bZ)$ whose $E_2$-page is 
\[E_2^{p,q} \cong \begin{cases} H^p(B\Gamma ; \bZ[\Gamma /\pi ] ) & \text{if $q=0$}, \\ 0 & \text{otherwise}. \end{cases}\]
This shows the lemma. 
\end{proof}
Through this isomorphism, the cohomology class $\sigma \in H^2(\Gamma ; \bZ[\Gamma /\pi])$ corresponds to a group extension 
\begin{align}
    1 \to \bZ[\Gamma /\pi] \to G \to \Gamma \to 1. \label{eq:G}
\end{align}
We relate the middle group $G$ with a fundamental group of manifolds. Let $M$ be a closed manifold with $\pi_1(M) \cong \Gamma$ and let $N \subset M$ be a codimension $2$ submanifold representing $f^*\sigma$ as in \cref{rmk:intersection}. 
In the same way as \cref{notn:mfds}, we write as $M^\circ := M \setminus U$ and $\sM:=(\widetilde{M} / \pi ) \setminus U$, where $U$ is a tubular neighborhood of $N$. 
\begin{lem}\label{lem:G}
We assume that $\sigma$ satisfies (A1) of \cref{assump:cocycle}.  
Then the homomorphism $\iota_* \colon \pi_1(M^\circ) \to \pi_1(M) \cong \Gamma$ factors through $G$, that is, there is a group homomorphism $\varphi \colon \pi_1(M^\circ) \to G$ such that the diagram
\[
\xymatrix{
\pi_1(M^\circ) \ar[r]^{ \ \ \ \ \ \ \ \varphi}  \ar[rd]_{\iota_*} & G \ar[d] \\ & \Gamma 
}
\]
commutes. 
\end{lem}
\begin{proof}
Set $G':=\pi_1(M^\circ )$. Let $\xi_M \colon M \to B\Gamma$ and $\xi_{M^\circ} \colon M^\circ \to BG'$ denote the classifying maps and let $\iota \colon M^\circ \to M$ denote the inclusion. 
Then $\iota_* \colon G' \to \Gamma$ induces a map from $BG'$ to $B\Gamma$, denoted by $B\iota_*$. 
The composition $B\iota_* \circ \xi_{M^\circ} $ is the classifying map induced from the homomorphism $G' \to \Gamma$ and hence is homotopic to $\xi_M \circ \iota$. 
Also, by definition of $U$, we have 
\[ (\xi_M \circ \iota )^* (\sigma)  = 0 \in H^2(M^\circ; \bZ[\Gamma / \pi]),\]
and hence $(B\iota_* \circ \xi_{M^\circ})^*(\sigma) =0$. 
Moreover, since $\xi_{M^\circ}$ is a $2$-connected map, it induces an injection of second cohomology groups. 
This show that $(B\iota_*)^* (\sigma) =0 \in H^2(BG';\bZ[\Gamma / \pi ])$. 
In terms of group extension, the pull-back extension 
\[1 \to \bZ[\Gamma /\pi] \to G' \times_\Gamma G \to G' \to 1\] 
splits, i.e., there is a homomorphism
\[
\xymatrix{
&&&G' \ar[d]^{\iota_*} \ar@{.>}[dl] & \\ 
1 \ar[r] & \bZ[\Gamma /\pi] \ar[r] & G  \ar[r] & \Gamma \ar[r] & 1
}
\]
as desired.
\end{proof}

We rephrase the condition (A2) of \cref{assump:cocycle} in terms of group extension. 
Let $\bZ[\Gamma/\pi]^{{\myhat}}$ denote the completion of $\bZ[\Gamma /\pi]$, i.e., the direct product $\prod_{\Gamma /\pi} \bZ$. 
Then $H^*(\Gamma; \bZ[\Gamma /\pi]^{\myhat})$ is isomorphic to the cohomology group $H^*(E\Gamma /\pi;\bZ)$. Moreover, through this isomorphism, 
 the map $ H^*(\Gamma ; \bZ[\Gamma /\pi]) \to H^*(\Gamma; \bZ[\Gamma /\pi]^{\myhat})$ induced from the inclusion of coefficients $\bZ[\Gamma /\pi] \to \bZ[\Gamma /\pi]^{{\myhat}}$ is identified with $j^*$. 
  Let $S^1 \subset \sM$ be a link of $N \subset \widetilde{M}/\pi$. Its fundamental group $\pi_1(S^1) \cong \bZ$ is sent to $\bZ \cdot \delta_{e\pi } \subset \bZ[\Gamma /\pi] \subset G$ by the map $\varphi$ in \cref{lem:G}.
\begin{lem}\label{lem:group_split}
Assume that $\sigma \in H^2_c(E\Gamma /\pi ; \bZ)$ satisfies \cref{assump:cocycle}. 
Then the homomorphism $\pi_1(S^1 ) \cong \bZ \to \pi_1(\widetilde{M}^\circ /\pi )$ is split injective.
\end{lem}
\begin{proof}
Let $H$ and $\hat{H}$ denote the preimage of $\pi$ by the quotient $G \to \Gamma $ and $\hat{G} \to \Gamma$ respectively. 
Since (A2) is equivalent to $j^*(\sigma) =0$, we have that the extension 
\[1 \to \bZ[\Gamma /\pi]^{\myhat} \to \hat{G} \to \Gamma \to 1\]
is trivial. Therefore the pull-back extension $1 \to \bZ[\Gamma /\pi]^{\myhat} \to \hat{H} \to \pi  \to 1$, and hence its quotient 
\[ 1 \to \bZ \to \hat{H}/ \hat{Z}^\perp \to \pi \to 1 \]
by the normal subgroup $ \hat{Z}^\perp := \prod_{(\Gamma / \pi) \setminus \{ e\pi \}} \bZ $ of $\hat{H}$, are also trivial. 
That is, $\hat{H}/\hat{Z}^\perp  \cong \pi \times \bZ$. 
Now the composition 
\[ \pi_1(\widetilde{M}{}^\circ /\pi ) \cong H \to \hat{H} \to \hat{H}/\hat{Z}^\perp \cong \pi \times \bZ \to \bZ \]
is the desired splitting of $\bZ \to H$. 
\end{proof}

\begin{rmk}\label{rmk:cocycle_transfer}
Let $s \colon \Gamma \to G$ be a set-theoretic section. Then the function $(g,h) \mapsto s(g)s(h)s(gh)^{-1} \in \bZ[\Gamma /\pi]$ satisfies the $2$-cocycle relation and is a representative of $\sigma \in H^2(\Gamma; \bZ[\Gamma /\pi])$. 
Hereafter we also use the letter $\sigma(g,h)$ for this $2$-cocycle. 
The codimension $2$ transfer map in group homology is given by the cap product as
\[H_*(\Gamma ; \bZ) \to H_*(\Gamma ; \bZ[\Gamma /\pi]^{\myhat}) \xrightarrow{\sigma \cap \cdot } H_{*-2}(\Gamma ; \bZ[\Gamma /\pi]) \to H_{*-2}(\pi; \bZ). \]
Conversely, the codimension $2$ co-transfer map in group cohomologies is defined by the cup product as
\[H^{*-2}(\pi ; \bZ) \cong H^{*-2}(\Gamma ; \bZ[\Gamma /\pi]^{\myhat}) \xrightarrow{\sigma \cup \cdot } H^{*}(\Gamma ; \bZ[\Gamma /\pi] ) \to H^*(\Gamma ; \bZ), \]
where the last morphism is induced from the map of coefficients $\bZ[\Gamma /\pi] \to \bZ$ sending $\sum n_{g\pi} \delta_{g\pi} \in \bZ[\Gamma /\pi]$ to $ \sum n_{g\pi} \in \bZ $.
\end{rmk}

We discuss a converse of \cref{prp:NSZ}, getting a codimension 2 submanifold $N \subset M$ from the cohomology class $\sigma$.
\begin{para}\label{para:cocycle}
Let $\Gamma$ be a finitely presented group, let $\pi$ be a subgroup and let $M$ be a closed manifold with $\pi_1(M) \cong \Gamma$. Let the cohomology class $\sigma \in H^2_c(E\Gamma / \pi ;\bZ)$ satisfy \cref{assump:cocycle}, and let $N \subset \xi_M^{-1}(U) \subset M$ be a codimension $2$ submanifold of $M$ as in  \cref{rmk:intersection}. 
Then the pair $(M,N)$ satisfies the following.
\begin{enumerate}
    \item There is a homomorphism $\pi_1(N) \to \pi$.
    \item The normal bundle $\nu N$ is trivial.
    \item The homomorphism $\pi_1(S^1) \to \pi_1(\sM) $ splits. 
\end{enumerate}
Indeed, (1) follows from (A1) since $N$ is identified with a submanifold of $\widetilde{M}/\pi$, (2) is checked in the same way as the last paragraph of \cref{rmk:intersection}, and (3) follows from \cref{lem:group_split}.
These three conditions are enough for the construction of codimension $2$ transfer discussed in the coming sections. 
\end{para}

\section{C*-algebraic codimension 2 transfer revisited}\label{section:Cstar}
A C*-algebraic codimension 2 transfer map \eqref{eq:KS} is constructed in \cite{kubotaGromovLawsonCodimensionObstruction2020} by using a representation of $\Gamma$ onto the Calkin algebra of a Hilbert C*-module. 
In this section, we revisit this construction from the viewpoint of the second cohomology class $\sigma \in H^2_c(E\Gamma /\pi ; \bZ)$ introduced in the previous section. 
We associate to $\sigma$ a twist, i.e.~a $\bT$-valued $2$-cocycle, of the action groupoid $\cB \rtimes \Gamma$, where $\cB$ is a bouquet of circles. 
The associated twisted crossed product C*-algebra $C(\cB) \rtimes_\sigma \Gamma $ is related to the Calkin algebra used in \cite{kubotaGromovLawsonCodimensionObstruction2020}. 

Throughout the paper, we consider the maximal completions for group C*-algebras or crossed products unless otherwise noted. 

\begin{rmk}
This and the next sections are based on complex $\K$-theory, but written in such a way that all statements and discussions are immediately extended to Real $\K$-theory (for example, we treat the degree 2 shift of K-theory faithfully). 
For this sake, we mix some notations specific to Real $\K$-theory into our complex-based discussion. 
Let $\bR^{0,1}$ denote the Real space $\bR$ equipped with the involution $t \mapsto -t$ and let $S^{0,1}:=C_0(\bR^{0,1})$, the Real C*-algebra of continuous functions on $\bR^{0,1}$ vanishing at infinity. Let $\mathbb{T}^{0,1}$ denote the $1$-point compactification of $\bR^{0,1}$, which is a circle with $2$ fixed points of the involution.  
\end{rmk}

\subsection{C*-algebraic codimension 2 transfer and twisted crossed product}\label{section:groupoid}
Let $\Pi$ denote the direct product group $\pi \times \bZ$ and 
let $\bB_{C^*\Pi }$ and $\bK_{C^*\Pi}$ denote the C*-algebra of bounded adjointable operators and compact operators on the Hilbert $C^*\Pi$-module $\ell^2(\Gamma /\pi ) \otimes C^*\Pi$ respectively. 
In \cite{kubotaGromovLawsonCodimensionObstruction2020}, a $\ast$-homomorphism
\[\phi \colon C^*\Gamma \to \cQ_{C^*\Pi}\]
is constructed, where $\cQ_{C^*\Pi}$ is the Calkin algebra $\bB_{C^*\Pi} / \bK_{C^*\Pi }$.
The C*-algebraic codimension 2 transfer map \eqref{eq:KS} is defined by the composition 
\begin{align}
    \tau_\sigma \colon \K_*(C^*\Gamma) \xrightarrow{\phi_*} \K_*(\cQ_{C^*\Pi}) \xrightarrow{\partial} \K_{*-1}(\bK_{C^*\Pi})  \xrightarrow{\beta} \K_{*-2}(C^*\pi ), \label{eq:transfer}
\end{align} 
where $\partial$ denotes the boundary map of associated to the extension $0 \to \bK_{C^*\Pi} \to \bB_{C^*\Pi} \to \cQ_{C^*\Pi} \to 0$ and $\beta$ denotes the projection onto the second direct summand of 
\begin{align*}
    \K_{*-1}(\bK_{C^*\Pi}) \cong \K_{*-1}(C^*\Pi) &\cong \K_{*-1}(C^*\pi \otimes C(\mathbb{T}^{0,1})) \cong  \K_{*-1}(C^*\pi) \oplus \K_{*-2}(C^*\pi),
\end{align*}
which is actually given by the Kasparov product with an element $\beta \in \K_{-1}(C(\bR^{0,1}))$.

We shortly recall the construction of $\phi$. 
Let $\cV$ denote the Mishchenko bundle, i.e., the $C^*\Pi$-module bundle 
\[\cV :=\overline{\sM}{}^\circ \times_\Pi C^*\Pi \to \sM^\circ \] 
over $\sM^\circ := \widetilde{M}{}^\circ /\pi \subset \sM$.
Let $\bar{p}_!\cV$ denote its push-forward onto $M^\circ$ with respect to the projection $\bar{p}$. 
This is a Hilbert $C^*\Pi$-module bundle whose fiber is 
\[ (\bar{p}_! \cV) _x \cong \ell^2(\Gamma /\pi) \otimes C^*\Pi.\] 
By \cref{lem:G}, the group $G$ defined as in \eqref{eq:G} acts on $(\overline{p}_!\cV)_x$ by the monodromy representation, which gives rise to a $\ast$-homomorphism  $ \tilde{\phi } \colon G \to \cU(\bB_{C^*\Pi })$. 
Moreover, the generator $t := \delta_{e\pi } \in \bZ[\Gamma /\pi] \subset G$ acts on the fiber $(\pi_!\cV)_x$ by a compact operator. 
Therefore, the $G$-action reduces to a homomorphism from $\Gamma = G/\langle t \rangle $ to the unitary group of $\cQ_{C^*\Pi}$, which induces the desired $\ast$-homomorphism $\phi$.

This $\phi$ associates a C*-algebra extension 
\[0 \to \bK_{C^*\Pi} \to \bB_{C^*\Pi} \oplus _{\cQ_{C^*\Pi}} C^*\Gamma  \to C^*\Gamma \to 0, \]
where the middle C*-algebra is the fiber product, i.e., the subalgebra of the direct sum $\bB_{C^*\Pi} \oplus  C^*\Gamma $ consisting of pairs $(x,y)$ satisfying $q (x) = \phi (y) \in \cQ_{C^*\Pi}$ (here $q \colon \bB_{C^*\Pi} \to \cQ_{C^*\Pi} $ denotes the quotient). 
By the maximality of the norm on the group C*-algebra $C^*G$, the representation
\[ (\tilde{\phi}, \psi) \colon G \to \cU(\bB_{C^*\Pi} \oplus _{\cQ_{C^*\Pi}} C^*\Gamma ),\]
where $\psi \colon G \to \Gamma$ denote the quotient, gives rise to a $\ast$-homomorphism $C^*G \to \bB_{C^*\Pi} \oplus _{\cQ_{C^*\Pi}} C^*\Gamma$. 

Now we determine the image of this $\ast$-homomorphism. 
Recall that the group $G$ is obtained as a twisted semi-direct product $ \bZ[\Gamma /\pi] \rtimes_\sigma \Gamma $. Hence we have
\[ C^*G \cong C^*(\bZ[\Gamma /\pi]) \rtimes_\sigma \Gamma = C(T) \rtimes_\sigma \Gamma, \]
where $T:=\prod _{\Gamma /\pi } \bT^{0,1}$ denotes the Pontrjagin dual of $\bZ[\Gamma /\pi]$. 
Here, by the abuse of notation, we use the same letter $\sigma $ for the associated $\cU(C^*(\bZ[\Gamma /\pi]))$-valued $2$-cocycle of $\Gamma$. Through the Pontrjagin duality $C^*(\bZ[\Gamma /\pi]) \cong C(T)$, this $\sigma$ is identified with the $2$-cocycle
\begin{align*}
 &&   \hat{\sigma}_{T} \colon \Gamma \times \Gamma \times T \to \bT, && \hat{\sigma}_{T}(g,h,\chi) := \chi(\sigma(g,h)), &&
\end{align*} 
of the action groupoid $T \rtimes \Gamma$. 
\begin{rmk}
The $2$-cocycle $\hat{\sigma}_T  \in Z_\Gamma ^2(T, \bT)$ is complex-conjugation invariant in the following sense. The group $\bZ_2$ acts on the Real space $T$ by the involution and on the sheaf $\underline{\bT}$ by complex conjugation. Then $\hat{\sigma}_T$ is $\bZ_2$-invariant in $Z_\Gamma^2(T, \underline{\bT})$. This enables us to impose the canonical Real structure on the twisted crossed product $C(T) \rtimes _\sigma \Gamma \cong C^*G$ determined by 
\[\overline{f \cdot u_g} = \bar{f} \cdot u_g, \]
where $\bar{f}(\chi)=\overline{f(\bar{\chi})}$ for $\chi \in T$. Note that this Real structure is the same with the standard Real structure imposed to the group C*-algebra $C^*G$. 
\end{rmk}

For $g \pi \in \Gamma /\pi$, let $\cX_{g\pi}$ denote the submodule $\ell^2(\{ g\pi \}) \otimes C^*\Pi$ of $\ell^2(\Gamma /\pi) \otimes C^*\Pi$.  
Since the $\bZ$-covering $\overline{\sM}{}^\circ \to \sM^\circ $ extends to $\overline{\sM } \to \sM$, the monodromy of the element $\delta_{g\pi } \in \bZ[\Gamma /\pi] \subset G$ is the diagonal unitary given by
\begin{align}
    \begin{split}
    u_{\delta_{g\pi}} |_{\cX_{h\pi}} = \begin{cases} u_t & \text{ if $g=h$, }\\ 1 & \text{otherwise,} \end{cases} 
    \end{split} \label{eq:Bzero}
\end{align} 
 where $u_t \in C^*\bZ \subset C^*\Pi$ denotes the generator. In particular, each $u_{\delta_{g\pi}}$ is a compact operator on $\ell^2(\Gamma /\pi) \otimes C^*\Pi$. 
This shows that the $\ast$-homomorphism
\[\phi \colon C^*(\bZ[\Gamma /\pi]) \cong C( T ) \to \bB_{C^*\Pi}\]
factors through $C(\cB)$, where 
\[\cB:=(\bR^{0,1} \times \Gamma /\pi)^+ \subset T\]
denotes the bouquet of $\Gamma/ \pi$ copies of circles.

Let $\hat{\sigma}_\cB $ denote the restriction of $\hat{\sigma}_T$ to the subgroupoid $\cB \rtimes \Gamma$. By abuse of notation, we simply write $\sigma$ for this cocycle. 
Then the above discussion is summarized to be the existence of a factorization indicated as the dotted arrow;
\begin{align*}
    \xymatrix{
    C^*G \ar[r] \ar[d] & C(\cB) \rtimes _{\sigma} \Gamma  \ar[r] \ar@{.>}[ld] & C^*\Gamma \ar[d]  \\
    \bB_{C^*\Gamma } \ar[rr] && \cQ_{C^*\Gamma }.
    }
\end{align*}

\begin{lem}\label{lem:untwist}
Assume that $\sigma$ satisfies (A2) of \cref{assump:cocycle}. Then the twisted crossed product $C_0(\cB_0) \rtimes_\sigma \Gamma $ is isomorphic to $S^{0,1}C^*\pi \otimes \bK(\ell^2(\Gamma /\pi))$.
\end{lem}
\begin{proof}
The lemma follows from the fact that the restriction of $\sigma$ onto $C_0(\cB_0)$ is a coboundary in the multiplier C*-algebra $C_b(\cB_0)$, which follows from (A2). 
In more detail, by (A2) we have a $1$-cocycle $b \colon \Gamma \to \bZ[\Gamma / \pi]^{\myhat }$ such that 
\[\sigma(g,h) = b(g)\gamma_g(b(h))b(gh)^{-1} \]
holds. Since $b(g) \in \bZ[\Gamma /\pi]^{\myhat}$ determines a bounded continuous function $b(g)|_\cB \in C_b(\cB_0)$, we get a $\ast$-homomorphism
\[\varphi \colon C_0(\cB_0) \rtimes_{\sigma , \alg} \Gamma \to C_0(\cB_0) \rtimes_{\alg} \Gamma \]
determined by $\varphi (f \cdot  u_g) = (f \cdot b(g)) \cdot u_g$ for any $f \in C_0(\cB_0)$ and $g \in \Gamma$. This extends to a $\ast$-isomorphism of maximal C*-completions.  
\end{proof}

Consequently, we obtain the following commutative diagram of exact sequences; 
\begin{align}
\begin{split}
\xymatrix{
0 \ar[r] & S^{0,1}C^*\pi \otimes \bK \ar[r] \ar[d] & C(\cB) \rtimes_\sigma \Gamma \ar[r] \ar[d] & C^*\Gamma \ar[r] \ar@{=}[d] & 0\\
0 \ar[r] & \bK_{C^*\Pi} \ar[r] & \bB_{C^*\Pi} \oplus _{\cQ_{C^*\Pi}} C^*\Gamma \ar[r] & C^*\Gamma \ar[r] & 0.
}
\end{split}\label{eq:exact_groupoid}
\end{align}
Therefore, the C*-algebraic codimension 2 transfer map \eqref{eq:transfer} is rewritten as 
\begin{align} \tau_\sigma \colon \K_*(C^*\Gamma)  \xrightarrow{\partial} \K_{*-1}(S^{0,1}C^*\pi \otimes \bK ) \cong \K_{*-2}(C^*\pi), \label{eq:transfer2}
\end{align}
where $\partial$ denotes the boundary map of the upper exact sequence of \eqref{eq:exact_groupoid}.

\begin{rmk}\label{rmk:intermediate}
This construction does not produce the reduced version of codimension $2$ transfer map, i.e., the map $\K_*(C^*_{\mathrm{red}} \Gamma ) \to \K_{*-2}(C^*_{\mathrm{red}} \pi)$.
This is due to the fact that the reduced crossed product functor is not exact at the middle if $\Gamma$ is not an exact group (for this topic, we refer the reader to \cite{brownAlgebrasFinitedimensionalApproximations2008}*{Subsection 5.1}). 
\end{rmk}

\subsection{Dixmier--Douady description of the second cohomology class $\sigma$}
The $2$-cocycle $\sigma $ lies in the second cohomology group $H^2_c(E\Gamma /\pi ; \bZ)$, which is the same thing as the compactly supported equivariant cohomology group $H^2_{\Gamma, c} (\Gamma /\pi ; \bZ)$. 
On the other hand, the groupoid twist $\sigma $ determines an element of the second reduced equivariant cohomology group $H^2_\Gamma(\cB , \ast ; \underline{\bT})$. 
Since $\cB$ is the suspension of the based space $(\Gamma /\pi)^+$, there is an isomorphism 
\begin{align}
H^2(\Gamma ; \bZ[\Gamma /\pi]) \cong H^2_\Gamma(\cB , \ast ; \underline{\bT}) \cong H^3_\Gamma (\cB , \ast ; \bZ). \label{eq:H-isom}
\end{align}
Indeed, the cohomology classes determined by $\sigma$ on both sides of the above isomorphism are identified in the following way. 
\begin{prp}\label{prp:Dixmier}
The group $2$-cocycle $\sigma \in Z^2(\Gamma ; \bZ[\Gamma /\pi])$ and the groupoid $2$-cocycle $\hat{\sigma}_\cB \in Z^2_\Gamma (\cB, \ast; \bT )$ are in the same cohomology class through the isomorphism \eqref{eq:H-isom}. 
\end{prp}
This justifies our simplified notation using $\sigma $ instead of $\hat{\sigma}_\cB$.
\begin{proof}
The identity map $z \colon S^1 \to \bT$ is a \v{C}ech $0$-cocycle of the sheaf $\underline{\bT}$ of $\bT$-valued functions on $S^1$.
The isomorphism 
\[ H_\Gamma ^2 ((\Gamma /\pi)^+, \ast ; \bZ ) \cong H^2_\Gamma (\cB , \ast ; \underline{\bT})\]
is induced from the cup product
\[\check{C}^2_\Gamma  ((\Gamma /\pi)^+,\ast  ; \bZ ) \otimes \check{C}^0(S^1, \ast ; \underline{\bT}) \to \check{C} ^2_\Gamma ((\Gamma /\pi)^+ \wedge S^1 , \ast ; \underline{\bT}) \]
 with $z \in \check{C}^0(S^1, \ast ; \underline{\bT})$. By definition we have $\sigma \cup z = \hat{\sigma} |_\cB$. 
\end{proof}

\subsection{From groupoid C*-algebra to Roe algebra}
In this subsection, we relate the twisted crossed product C*-algebra $C(\cB) \rtimes _{\sigma}\Gamma$ with invariant Roe algebras.  
Let $\Gamma$, $\pi$, $\sigma$, $M$ and $N$ be as in \cref{para:cocycle} and let $\widetilde{M}$, $\sM$ be the covering spaces as in \cref{rmk:intersection}.

\begin{para}
We fix a Riemannian metric on $M$. Let $C^*(\widetilde{M})^\Gamma$ denote the $\Gamma$-invariant (maximal) Roe algebra of $\widetilde{M}$, i.e., the completion of the $\ast$-algebra $\bC[\widetilde{M}]^\Gamma$ of bounded operators on $L^2(\widetilde{M})$ such that
\begin{enumerate}
    \item the propagation of $T$ is finite, 
    \item $T$ is locally compact, i.e., $fT, Tf \in \bK(L^2(\widetilde{M}))$, and
    \item $T$ is $\Gamma$-invariant.
\end{enumerate}
Here, the propagation of $T$ is defined to be the supremum of the distance $d(x,y)$ of $x,y \in \widetilde{M}$ such that $f_2Tf_1 \neq 0$ for any $f_1,f_2 \in C_c(\widetilde{M})$ such that $f_1(x) \neq 0$, $f_2(y) \neq 0$. The norm on $\bC[\widetilde{M}]^\Gamma$ is chosen as the maximal one among all possible C*-norms on it (the well-definedness is proved in \cite{gongGeometrizationStrongNovikov2008}*{3.5}).  
We also define the $\Pi$-invariant Roe algebra $C^*(\overline{\sM})^\Pi $ in the same way.
\end{para}

As is constructed in \cite{changPositiveScalarCurvature2020}*{Lemma 2.12}, there is a $\ast$-homomorphism
\begin{align}
     \epsilon_\bZ \colon C^*(\overline{\sM})^\Pi \to C^*(\widetilde{\sM} )^\pi, \label{eq:epsilon}
\end{align} 
which sends the operator $T \in \bC[\overline{\sM}]^\Pi$ represented by a kernel function $t(x,y) \in \widetilde{M} \times \widetilde{M} \to \bC $ to the operator on $\widetilde{\sM} $ which is represented by the kernel function 
\[ \epsilon_\bZ (t)(x,y) = \sum_{\pi (\bar{y}) = y} t(\tilde{x}, \tilde{y}), \]
which is well-defined independent of the choice of $\bar{x} \in \overline{\sM}$ such that $\pi(\bar{x})=x$. The $\ast$-homomorphism of this kind does exist for any Galois covering space and a surjection of the group. 
More generally, for a free proper $G$-space $X$ and a normal subgroup $N \triangleleft G$, a $\ast$-homomorphism $\epsilon_N \colon C^*(X)^G \to C^*(X/N)^{G/N}$ is defined in the same way. We simply write this $\ast$-homomorphism as $\epsilon$ if $N$ is clear from the context. 
Let $C^*(\widetilde{\sM})^\Pi_0$ denote the kernel of $\epsilon_\bZ $.

Set $M^\bullet :=M \setminus N$ and let $\widetilde{M}{}^\bullet$, $\overline{M}{}^\bullet$ and $\widehat{M}{}^\bullet$ denote its covering space whose fiber is $\Gamma$, $G / Z^\perp$ and $G$ respectively, where $Z^\perp:=\bZ[(\Gamma /\pi) \setminus \{ e\pi\} ]$. 
Since $\widetilde{M} {}^\bullet$ is a dense open subspace of $\widetilde{M}$, the identification $L^2(\widetilde{M}) \cong L^2(\widetilde{M}{}^\bullet)$ induces  $\ast$-isomorphisms $C^*(\widetilde{M})^\Gamma \cong C^*(\widetilde{M}{}^\bullet )^\Gamma$ and $C^*(\widetilde{M})^\pi \cong C^*(\widetilde{M}{}^\bullet )^\pi$. 
We choose a tubular neighborhood $V$ of $N$ including the closure of $U$ and a diffeomorphism between $V \setminus N$ and $V \setminus U$ which is the identity at the boundary. 
This gives rise to a $\ast$-isomorphism  of Roe algebras $\chi \colon C^*(\widetilde{M}{}^\bullet )^\pi \to C^*(\widetilde{\sM})^\pi $. 

We write the restriction map as the inclusion of invariant Roe algebras
\begin{align*}
    \mathrm{res}_{\Gamma}^\pi \colon C^*(\widetilde{M})^\Gamma \to  C^*(\widetilde{M} )^\pi .  \label{eq:group_Roe}
\end{align*}
Also, we identify the group C*-algebra $C^*\Gamma \otimes \bK$ with the invariant Roe algebra. Let $U$ be a fundamental domain of the $\Gamma$-space $\widetilde{M} $, i.e., a $1$-connected dense open subset of $M$. 
We choose $U$ in the way that $U \cap N = \emptyset$. 
Then the isomorphism $L^2(\widetilde{M}) \cong \ell^2(\Gamma ) \otimes L^2(U)$ gives rise to $C^*(\widetilde{M})^\Gamma \cong C^*\Gamma \otimes \bK(L^2(U))$. 
In summary, we get a $\ast$-homomorphism
\begin{align}
    \phi \colon C^*\Gamma \otimes \bK \cong C^*(\widetilde{M})^\Gamma \xrightarrow{\mathrm{res}_\Gamma ^\pi} C^*(\widetilde{M}{}^\bullet)^\pi \xrightarrow{\chi} C^*(\widetilde{\sM})^\pi.
\end{align}

\begin{lem}\label{lem:group_Roe}
There is a homomorphism of C*-algebra extensions
\[
\xymatrix{
0 \ar[r] & S^{0,1}C^*\pi \otimes \bK \ar[r] \ar[d]^\phi   & C(\cB) \rtimes _\sigma \Gamma \otimes \bK \ar[r] \ar[d]^\phi & C^*\Gamma \otimes \bK \ar[r] \ar[d]^{\phi } &  0 \\
0 \ar[r] & C^*(\overline{\sM})^{\Pi}_0 \ar[r] & C^*(\overline{\sM})^{\Pi} \ar[r] &  C^*(\widetilde{\sM})^\pi \ar[r] & 0 .
}
\]
such that the right vertical map $\phi$ is the same as \eqref{eq:group_Roe}. Moreover, the image of the left vertical map $\phi$ coincides with $C^*(\overline{\sN} \subset \overline{\sM})^\Pi \cap C^*(\overline{\sM})^\Pi $.
\end{lem}
\begin{proof}
Let $Z:=\bZ[\Gamma /\pi]$ and $Z^\perp:=\bZ[(\Gamma /\pi) \setminus \{ e\pi\} ]$ as above. By definition of the $\ast$-homomorphism $\epsilon_Z$ as in \eqref{eq:epsilon}, the diagram
\[
\xymatrix{
C^*G \otimes \bK \ar[r]^\cong \ar[d]^\psi  & C^*(\widehat{M}{}^\bullet )^G \ar[d]^{\epsilon_Z} \ar[r]^{\mathrm{res}_{G}^H} & C^*(\widehat{M}{}^\bullet)^{H} \ar[r]^{\epsilon_{Z^\perp}} \ar[d]^{\epsilon_Z}  & C^*(\overline{M}{}^\bullet)^\Pi \ar[d]^{\epsilon_\bZ}   \ar[r]^\chi  & C^*(\overline{\sM})^\Pi \ar[d]^{\epsilon_\bZ} \\ 
C^*\Gamma \otimes \bK \ar[r]^\cong & C^*(\widetilde{M}{}^\bullet)^\Gamma \ar[r]^{\mathrm{res}_{\Gamma}^\pi} & C^*(\widetilde{M}{}^\bullet)^\pi \ar[r]^\cong & C^*(\widetilde{M}{}^\bullet )^\pi \ar[r]^\chi  & C^*(\widetilde{\sM})^\pi 
}
\]
commutes. 

Through the unitary isomorphism
\begin{align}
L^2(\overline{M}{}^\bullet ) \cong L^2(U) \otimes \ell^2(G/Z^\perp ) \cong L^2(U) \otimes \ell^2 \Gamma  \otimes \ell^2\bZ,  \label{eq:L2_trivialization}
\end{align}
the representation $\phi$ in the above paragraph is identified with the monodromy representation $\tilde{\phi}$ introduced in \cref{section:groupoid}, and hence factors through $C(\cB ) \rtimes_\sigma \Gamma $. 
Moreover, by \eqref{eq:Bzero}, the image of $C_0(\cB_0 ) \rtimes_\sigma \Gamma$ acts on the Hilbert $C^*\Pi$-module $\ell^2(\Gamma /\pi) \otimes C^*\Pi $ by a compact operator. 
This shows that $\phi$ sends $(C_0(\cB_0) \rtimes \Gamma \otimes \bK)$ to $C^*(\overline{\sN} \subset \overline{\sM})^\Pi \cap C^*(\overline{\sM})_0^\Pi$.
\end{proof}

\begin{rmk}\label{rmk:algebraic_coarse}
\cref{lem:group_Roe} is also proved more algebraically with the language of twisted crossed products. Through the unitary \eqref{eq:L2_trivialization}, the Roe algebras are identified with crossed products as 
\begin{align*}
    C^*(\overline{\sM})^\Pi & \cong (c_b(\Gamma /\pi) \otimes C^*\bZ) \rtimes \Gamma, \\
    C^*(\overline{\sM})^\Pi_0 & \cong (c_b(\Gamma /\pi) \otimes S^{0,1}) \rtimes \Gamma.  
\end{align*} 
Hence the map of exact sequences in the statement of \cref{lem:group_Roe} is identified with 
\begin{align*}
\mathclap{
    \xymatrix@C=1.5em{
    0 \ar[r] & S^{0,1}c_0(\Gamma /\pi) \rtimes_\sigma \Gamma \ar[r] \ar[d]   & C(\cB) \rtimes _\sigma \Gamma  \ar[r] \ar[d]^\phi & C^*\Gamma  \ar[r] \ar[d] &  0 \\
0 \ar[r] & S^{0,1}c_b(\Gamma /\pi) \rtimes_\sigma \Gamma  \ar[r] & (c_b(\Gamma /\pi) \otimes C^* \bZ ) \rtimes _\sigma \Gamma  \ar[r] &  c_b(\Gamma /\pi) \rtimes \Gamma  \ar[r] & 0 ,
}}
\end{align*}
after taking tensor product with $\bK(L^2(U))$. 
\end{rmk}

Let $s_0 \colon C^*\Gamma \to C(\cB) \rtimes_\sigma \Gamma $ be a set-theoretical section. 
Then \cref{lem:group_Roe} shows that the composition 
\begin{align}
    s:=q \circ \phi \circ s_0 \colon C^*(\widetilde{M} )^\Gamma \to C^*(\overline{\sM})^\Pi / C^*( \overline{\sN}{} \subset \overline{\sM})^\Pi \label{eq:split}
\end{align} 
is a $\ast$-homomorphism such that the diagram 
\[ 
\xymatrix
{
C^*(\widetilde{M})^\Gamma \ar[r]^{s \hspace{3em}} \ar[d]^{\mathrm{res}_\Gamma ^\pi }  & C^*(\overline{\sM})^\Pi / C^*( \overline{\sN}{} \subset \overline{\sM})^\Pi \ar[d]^\epsilon  \\ 
C^*(\widetilde{\sM})^\pi \ar[r]^{q\hspace{3em}} & C^*(\widetilde{\sM})^\pi / C^*( \widetilde{\sN} \subset \widetilde{\sM})^\pi
}
\]
commutes.

For the latter use, we give an explicit and intuitive description of this $\ast$-homomorphism $s$. Let $R$ denote the injectivity radius of $M$ with respect to a fixed Riemannian metric. We use the neighborhood $B_{R/4}(N):=\{ x \in M \mid d(x,N) < R/4 \}$ of $N$ as $U$. 
Then, as is constructed in  \cite{changPositiveScalarCurvature2020}*{Proposition 2.8}, an operator $T \in \bB(L^2(\widetilde{\sM}))$ with $\Prop T < R/4$ lifts to $\overline{T} \in \bB(L^2(\overline{\sM}))$ which is uniquely characterized by the property $\Prop \overline{T} = \Prop T <R/4$ (we remark that any two points $x,y \in \widetilde{\sM}$ with $d(x,y) < R/4$ is connected by a unique geodesic in $\widetilde{M}{}^\bullet$ with the length less that $R/4$).  
Let 
$P$ denote the projection onto the subspace $L^2(\widetilde{\sM}) \subset L^2(\widetilde{M})$.
\begin{prp}\label{prp:lift}
Assume that $T \in C^*(\widetilde{M})^\Gamma$ satisfies $\Prop T < R/4$. Then the equality $s(T) =  q(\overline{PTP}) $ holds, where $q$ denotes the quotient map. 
\end{prp}
\begin{proof}
We use the neighborhood $B_{R/2}(N)$ for $V$ in the proof of \cref{lem:group_Roe}.  Let $\widetilde{U}$, $\overline{U}$ and $\widehat{U}$ denote its inverse image in the covering spaces $\widetilde{M}{}^\bullet$, $\overline{M}{}^\bullet$ and $\widehat{M}{}^\bullet$ respectively. Let $\chi $ be a smooth bump function on $M$ such that $\chi \equiv 0$ on $V$ and $\chi  \equiv 1$ on the complement of $B_{3R/4}(N)$. We use the same letter $\chi$ for its lift to $\widetilde{M}$. 

For $T \in C^*(\widetilde{M})^\Gamma$ with $\Prop T < R/4$, set $T_0:= T \chi \in C^*(\widetilde{M})^\Gamma $ and $T_1:=T(1-\chi)$. 
Then the decomposition $T=T_0 + T_1$ satisfies that $\mathop{\mathrm{supp}}T_0 \subset \widetilde{M}{}^\circ $ and $\mathop{\mathrm{supp}}T_1 \subset B_{R}(N)$. 
We apply \cite{changPositiveScalarCurvature2020}*{Proposition 2.8} to get a lift $\widehat{T}_0 \in C^*(\widehat{M}{}^\circ )^G$ with propagation less than $R/4$. By the proof of \cref{lem:group_Roe}, we obtain that 
\[ s(T_0) = \text{$\epsilon (\widehat{T}_0) = \overline{T}_0$  modulo $C^*(\overline{\sN} \subset \overline{\sM})^\Pi $. }\]
Here the equality $\epsilon (\widehat{T}_0) = \overline{T}_0$ follows from the fact $\Prop(\epsilon (\widehat{T}_0)) < R/4$, which follows from the definition of the map $\epsilon$. 

The remaining task is to show that $s(T_1) = \overline{PT_1P}$ modulo the boundary. Set $W:=B_R(N) \setminus N$ and we write its inverse images in $\widetilde{M}{}^\bullet$, $\overline{M}{}^\bullet$ and $\widehat{M}{}^\bullet$ as $\widetilde{W}$, $\overline{W}$ and $\widehat{W}$ respectively.  Then $\widetilde{W}$ is a disjoint union of connected spaces $\widetilde{W}_{g\pi}$ parametrized by $\Gamma /\pi$. Each $\widetilde{W}_{g\pi}$ has a $\bZ$-covering $\overline{W}_{g\pi}$ such that
\begin{align*}
    \widehat{W} =& \bigsqcup _{g\pi \in \Gamma /\pi} \overline{W}_{g\pi} \times \bZ[X_g ], \\
    \overline{W} =& \overline{W}_{e\pi} \sqcup \bigsqcup _{\substack{g\pi \in \Gamma /\pi \setminus \{e\pi\} }} \widetilde{W}_{g\pi} \times \bZ,
\end{align*}
where $X_g:=(\Gamma/\pi) \setminus \{ g\pi\}$.
Since $T_1$ is supported in $\widehat{W}$, it is decomposed into an infinite sum $\sum _{g \in \Gamma /\pi} T_{1,g\pi}$. Let $S_{g\pi}$ be an arbitrary lift of $T_{1,g\pi}$ to $C^*(\overline{W}_{g\pi})^\bZ$, i.e., $\epsilon_\bZ (S_{g\pi} ) = T_{1,g\pi} \in C^*(\widetilde{W}_{g\pi})$. Then
\[\sum_{g\pi \in \Gamma /\pi} S_{g\pi} \otimes 1 \in \prod _{g\pi \in \Gamma /\pi} C^*(\overline{W}_{g\pi})^\bZ \otimes C^*(\bZ[X_g])^{\bZ[X_g]} \]
is a lift of $T_1$ to $\widetilde{W}$, and hence $s (T_1)$ is 
\[\epsilon_{Z^\perp } \big( \sum S_{g\pi} \otimes 1 \big) = S_{e\pi} \oplus \sum_{g\pi \neq e\pi} T_{g\pi} \otimes 1 \in C^*(\overline{W}_{e\pi})^\bZ \oplus \prod_{g\pi \neq e\pi } C^*(\widetilde{W}_{g\pi} \times \bZ)^\bZ, \]
which coincides with 
\[ 0 \oplus \sum_{g\pi \neq e\pi} T_{g\pi} \otimes 1 = \overline{PT_1P}\]
modulo $C^*(\overline{\sN} \subset \overline{\sM})^\Pi $. This finishes the proof. 
\end{proof}
This proposition shows that the map $\tau_\sigma$ constructed in \eqref{eq:transfer2} is rephrased as the composition
\[\tau _\sigma = \partial \circ s_* \colon \K_*(C^*(\widetilde{M})^\Gamma) \to \K_* \Big( \frac{C^*(\overline{\sM})^\Pi }{C^*(\overline{\sN} \subset \overline{\sM})^\Pi} \Big) \to \K_{*-1}(C^*(\overline{\sN} \subset \overline{\sM})^\Pi).  \]
Indeed, the range of this map is isomorphic to $\K_{*-2}(C^*(\widetilde{N})^\pi) \cong \K_{*-2}(C^*\pi)$, as is discussed in next section (\cref{lem:pi_zero}).

\section{Codimension 2 transfer of the secondary index invariants}\label{section:HigsonRoe}
In this section we give a proof of \cref{thm:main1}. We extend the codimension $2$ transfer map to the Higson--Roe analytic surgery sequence, i.e., the long exact sequence associated to the short exact sequence of coarse C*-algebras
\[0 \to C^*(\widetilde{M})^\Gamma \to D^*(\widetilde{M})^\Gamma \to Q^*(\widetilde{M})^\Gamma \to 0. \]
This is obtained by extending the $\ast$-homomorphism \eqref{eq:split} to $D^*(\widetilde{M})^\Gamma$. 
Combined with the `boundary of Dirac is Dirac' and `boundary of signature is $2^\epsilon$ times signature' principles, we also show that the resulting homomorphism in K-theory relates the primary and the secondary higher index invariants of $M$ and those of $N$. 

\subsection{Codimension 2 transfer of the Higson--Roe analytic surgery sequence}
We start with a short review of the pseudo-local coarse C*-algebra, in particular the maximal C*-completion. For a more detail on the definitions, see for example \cite{roeIndexTheoryCoarse1996} and \cites{gongGeometrizationStrongNovikov2008,oyono-oyonoTheoryMaximalRoe2009}.
\begin{para}\label{para:coarse}
Let $D^*_\alg (\widetilde{M})^\Gamma $ denote the set of operators on $T \in \bB(L^2(\widetilde{M}))$ such that 
\begin{enumerate}
    \item $T$ is $\Gamma$-invariant, 
    \item $T$ is pseudo-local, i.e., $[T, f] \in \bK(L^2(\widetilde{M}))$ for any $f \in C_c(\widetilde{M})$, 
    \item the propagation of $T$ is finite. 
\end{enumerate}
An operator $T \in D^*_\alg(\widetilde{M})^\Gamma$ satisfies $TS, ST \in \bC[\widetilde{M}]^\Gamma$ for any $S \in \bC[\widetilde{M}]^\Gamma$, and moreover, the multiplier norm 
\[\| T\|_{\mathrm{max}}:=\sup_{S \in \bC[\widetilde{M}]^\Gamma \setminus \{0\}} \| TS\|_{\mathrm{max}} / \| S \|_{\mathrm{max}} \]
is finite  (\cite{oyono-oyonoTheoryMaximalRoe2009}*{Lemma 2.16}). 
Hence $D^*_\alg(\widetilde{M})^\Gamma$ is regarded as a $\ast$-subalgebra of the multiplier C*-algebra $\cM(C^*(\widetilde{M})^\Gamma)$. Let $D^*(\widetilde{M})^\Gamma$ denote its closure.

In the same way, we also define the C*-subalgebra $D^*(\overline{\sM})^\Pi$ of $\cM(C^*(\overline{\sM})^\Pi)$. 
The set $D^*_\alg (\overline{\sN} \subset \overline{\sM})^\Pi $ 
of bounded operators on $L^2(\overline{\sM})$ satisfying the conditions (1), (2), (3) above and 
\begin{enumerate}
    \item[(4)] $T$ is supported on a $R$-neighborhood of $\overline{\sN}$, and
    \item[(5)] $Tf, fT \in \bK(L^2(\overline{\sM}))$ for any $f \in C_c(\overline{\sM} \setminus \overline{\sN})$,
\end{enumerate}
forms a $\ast$-ideal of $D^*_\alg(\overline{\sM})^\Pi$. Hence its closure, denoted by $D^*(\overline{\sN} \subset \overline{\sM})^\Pi $, is a $\ast$-ideal of $D^*(\overline{\sM})^\Pi$. 
\end{para}
The quotient C*-algebras are written as 
\begin{align*}
    Q(\overline{\sM})^\Pi &:=D^*(\overline{\sM})^\Pi / C^*(\overline{\sM})^\Pi ,\\
    Q^*(\overline{\sN} \subset \overline{\sM})^\Pi  &:= D^*(\overline{\sN} \subset \overline{\sM})^\Pi / C^*(\overline{\sN} \subset \overline{\sM})^\Pi . 
\end{align*}
Then $Q^*(\overline{\sN} \subset \overline{\sM})^\Pi $ is a $\ast$-ideal of $Q^*(\overline{\sM})^\Pi$. Moreover, there are isomorphisms
\begin{align*}
    \K_{1-*}^\Pi (\overline{\sM}) &\cong \K_*(Q^*(\overline{\sM})^\Pi ) , \\
    \K_{1-*}^\Pi (\overline{\sM}, \overline{\sN}) &\cong  \K_*(Q^*(\overline{\sM})^\Pi / Q^*(\overline{\sN} \subset \overline{\sM})^\Pi ) , 
\end{align*}
where the groups in the left hand side are equivariant K-homology groups.

\begin{rmk}\label{lem:pi_zero}
The $\K$-group $\K_*(C^*(\overline{\sN} \subset \overline{\sM})^\Pi)$ is isomorphic to $\K_*(C^*(\overline{\sN})^\Pi)$, and the same is true for $D^*$ and $Q^*$ cases  \cite{siegelHomologicalCalculationsAnalytic2012}*{Proposition 4.3.34}. These $\K$-groups are isomorphic to the coarse $\K$-groups of $\widetilde{N}$ in the following way.  

First, in $C^*$-case, the map 
\[\epsilon_* \oplus (\partial_{\mathrm{MV}} \circ \mathrm{res}_\Pi^\pi) \colon \K_*(C^*(\overline{\sN})^\Pi)) \to \K_{*}(C^*(\widetilde{N})^\pi) \oplus \K_{*-1}(C^*(\widetilde{N})^\pi) \]
is an isomorphism. Here, $\epsilon$ is the map \eqref{eq:epsilon}, $\mathrm{res}_\Pi ^\pi$ is the inclusion $C^*(\overline{\sN})^\Pi \to C^*(\overline{\sN})^\pi$, and $\partial_{\mathrm{MV}}$ denotes the coarse Mayer--Vietoris boundary map \cite{higsonCoarseMayerVietorisPrinciple1993}. 
This is due to the K\"unneth theorem \cite{rosenbergKunnethTheoremUniversal1986} since $C^*(\overline{\sN})^\Pi)\cong C^*(\widetilde{N})^\pi \otimes C^*\bZ$. 

The homomorphisms $\epsilon_* \oplus (\partial _{\mathrm{MV}} \circ \mathrm{res}_\Pi^\pi)$ are also defined for $D^*$ and $Q^*$ coarse C*-algebras, and are isomorphic. 
Indeed, the $\ast$-homomorphism \eqref{eq:epsilon} extends to $D^*$ and $Q^*$ algebras by the same construction. 
Also, the coarse Mayer--Vietoris boundary map is defined for \cite{siegelHomologicalCalculationsAnalytic2012}. 
In $Q^*$-case, this follows from the K\"unneth theorem of K-homology group $\K_*^\Pi(\overline{\sN})$. 
The $D^*$-case is proved by the five lemma. 
\end{rmk}

\begin{prp}\label{prp:splitD}
The $\ast$-homomorphism 
\[s \colon C^*(\widetilde{M})^\Gamma \to C^*(\overline{\sM})^\Pi / C^*(\overline{\sN} \subset \overline{\sM})^\Pi \]
constructed in \eqref{eq:split} lifts to 
\[s \colon D^*(\widetilde{M} )^\Gamma \to D^*(\overline{\sM})^\Pi / D^*( \overline{\sN} \subset \overline{\sM})^\Pi. \]
\end{prp}
\begin{proof}
Let $\varepsilon >0$ be the positive number in \cref{prp:lift}.
For $T \in D^*(\widetilde{M})^\Gamma$, there is a decomposition $T=T_0 + T_1$, where $T_0 \in D^*(\widetilde{M})^\Gamma $ satisfies $\Prop (T_0) < R/4 $ and $T_1 \in C^*(\widetilde{M})^\Gamma$ (cf.\ \cite{roeIndexTheoryCoarse1996}*{Lemma 5.8}).
We define the $\ast$-homomorphism $s$ as
\[s(T):= \overline{PT_0P} + s(T_1).\]
Here $\overline{T}_0$ and $s(T_1)$ are as in \cref{prp:lift}. 
This is well-defined independent of the choice of the decomposition $T = T_0 + T_1$. 
Indeed, for another decomposition $T=T_0' + T_1'$
we have
\[(\overline{T}_0 + s(T_1)) - (\overline{T}{}_0' + s(T_1')) = \overline{S} - s(S),  \]
where $S:=T_0-T_0' = T_1'-T_1$ is an operator with $\Prop (S) < R/4 $ and $S \in C^*(\widetilde{M})^\Gamma$. By \cref{prp:lift}, the difference $\overline{S} - s(S)$ lies in $C^*(\overline{\sN} \subset \overline{\sM})^\Pi \subset D^*(\overline{\sN} \subset \overline{\sM})^\Pi$.  

We show that this $s$ is multiplicative. For $T,S \in D^*(\widetilde{M})^\Gamma$, we choose decompositions $T=T_0+T_1$ and $S = S_0+S_1$ such that the propagation of $T_0$ and $S_0$ are less than $R/8$. Then, since $\mathop{\mathrm{Prop}} (T_0S_0 ) < R/4 $ and $ST - T_0S_0 \in C^*(\widetilde{M})^\Gamma$, we have $s(TS) = \overline{PT_0S_0P} + s(TS -T_0S_0)$, and hence
\begin{align*}
    s(T)s(S) - s(TS) =& (\overline{P T_0 P } \cdot \overline{PS_0P} - \overline{PT_0S_0P} )  + (\overline{PT_0P} \cdot s(S_1) - s(T_0S_1)) \\
    & + (s(T_1) \cdot \overline{PS_0P}  - s(T_0S_1))  + (s(T_1)s(S_1) - s(T_1S_1)).
\end{align*}   
It is straightforward to see that the first term $\overline{P T_0 P } \cdot \overline{PS_0P} - \overline{PT_0S_0P} = \overline{P[T_0,P]S_0P}$ is contained in $D^*(\overline{\sN} \subset \overline{\sM})^\Pi $ and the second, third and forth terms are in $C^*(\overline{\sN} \subset \overline{\sM})^\Pi $. 
\end{proof}

We write $\partial _D$ and $\partial_Q$ for the boundary homomorphism in K-theory associated to the exact sequences
\begin{align*}
    0 \to D^*( \overline{\sN} \subset \overline{\sM})^\Pi \to & D^*(\overline{\sM})^\Pi \to D^*(\overline{\sM})^\Pi / D^*( \overline{\sN} \subset \overline{\sM})^\Pi \to 0, \\
    0 \to Q^*( \overline{\sN} \subset \overline{\sM})^\Pi \to & Q^*(\overline{\sM})^\Pi \to Q^*(\overline{\sM})^\Pi / Q^*( \overline{\sN} \subset \overline{\sM})^\Pi \to 0,
\end{align*}
respectively. 
\begin{defn}\label{defn:transfer}
The codimension 2 transfer maps are defined as
\begin{align*}
    \tau_\sigma:= \partial_{\mathrm{MV}} \circ \mathrm{res}_\Pi^\pi \circ \partial_D \circ s_* &\colon \K_*(D^*(\widetilde{M})^\Gamma ) \to \K_{*-2}(D^*(\widetilde{N})^\pi ), \\
    \tau_\sigma:= \partial_{\mathrm{MV}} \circ \mathrm{res}_\Pi^\pi \circ \partial_Q \circ s_* &\colon \K_*(Q^*(\widetilde{M})^\Gamma ) \to \K_{*-2}(Q^*(\widetilde{N})^\pi ).
\end{align*}
\end{defn}

\begin{thm}\label{thm:HigsonRoe}
The following diagram of long exact sequences commutes;
\[\mathclap{
\xymatrix@C=0.5cm{
\cdots  \ar[r] & \K_*(C^*(\widetilde{M})^\Gamma) \ar[r] \ar[d]^{\tau_\sigma} & \K_*(D^*(\widetilde{M})^\Gamma) \ar[r] \ar[d]^{\tau_\sigma} & \K_*(Q^*(\widetilde{M})^\Gamma) \ar[r] \ar[d]^{\tau_\sigma} & \K_{*-1}(C^*(\widetilde{M})^\Gamma ) \ar[d]^{\tau_\sigma} \ar[r] & \cdots \\
\cdots \ar[r] & \K_{*-2}(C^*(\widetilde{N})^\pi ) \ar[r] & \K_{*-2}(D^*(\widetilde{N})^\pi ) \ar[r] &  \K_{*-2}(Q^*(\widetilde{N})^\pi ) \ar[r] & \K_{*-3}(C^*(\widetilde{N})^\pi) \ar[r] & \cdots .
}}
\]
\end{thm}
\begin{proof}
This follows from the commutativity of the diagrams
\[
\xymatrix{
0 \ar[r] & C^*(\widetilde{M})^\Gamma \ar[d]^{s_*} \ar[r] & D^*(\widetilde{M}) ^\Gamma \ar[r] \ar[d]^{s_*} & Q^*(\widetilde{M})^\Gamma \ar[r] \ar[d]^{s_*}  & 0\\    
0 \ar[r] & \frac{C^*(\overline{\sM})^\Pi}{C^*( \overline{\sN} \subset \overline{\sM})^\Pi} \ar[r] & \frac{D^*(\overline{\sM})^\Pi}{D^*( \overline{\sN} \subset \overline{\sM})^\Pi} \ar[r] & \frac{Q^*(\overline{\sM})^\Pi}{Q^*( \overline{\sN} \subset \overline{\sM})^\Pi} \ar[r]  & 0
}
\]
and
\[
\xymatrix{
 0 \ar[r] & C^*( \overline{\sN} \subset \overline{\sM})^\Pi \ar[d] \ar[r] &  C^*(\overline{\sM})^\Pi \ar[r] \ar[d] & C^*(\overline{\sM})^\Pi / C^*( \overline{\sN} \subset \overline{\sM})^\Pi \ar[r] \ar[d] & 0 \\
 0 \ar[r] & D^*( \overline{\sN} \subset \overline{\sM})^\Pi \ar[d] \ar[r] &  D^*(\overline{\sM})^\Pi \ar[r] \ar[d] & D^*(\overline{\sM})^\Pi / D^*( \overline{\sN} \subset \overline{\sM})^\Pi \ar[r] \ar[d] & 0 \\
 0 \ar[r] & Q^*( \overline{\sN} \subset \overline{\sM})^\Pi  \ar[r] &  Q^*(\overline{\sM})^\Pi \ar[r]  & Q^*(\overline{\sM})^\Pi / Q^*( \overline{\sN} \subset \overline{\sM})^\Pi \ar[r]  & 0, 
}
\]
both of which are obvious from the definitions. 
\end{proof}

\subsection{Secondary higher index invariants}
Next, we study the behavior of K-theory classes with the geometric origin, the primary and secondary higher index invariants of manifolds, under the codimension $2$ transfer map.

\begin{para}\label{lem:Roe}
Let $R>0$ be injectivity radius of $M$. Let $D$ be a $\Gamma$-invariant formally self-adjoint elliptic differential operator of first order on a $\Gamma$-vector bundle on $\widetilde{M}$. Assume that the norm of the principal symbol $\sigma(D)$ restricted to the unit cosphere bundle $\bS (T^*\widetilde{M})$ is bounded by $1$. 
Let $\overline{D}_\infty$ be an elliptic operator on $\overline{\sM}_\infty := \overline{\sM} \sqcup_{\overline{\sN}} \overline{\sN} \times \bR_{\geq 0}$, equipped with a complete metric extending that of $\overline{\sM}$, whose restriction to $\overline{\sM}$ coincides with the lift of $D$. 

Let $\chi $ be a real-valued odd function on $\bR$ whose Fourier transform $\hat{\chi}$ is smooth on $\bR \setminus \{ 0\}$ and satisfies $\mathop{\mathrm{supp}} (\hat{\chi}) \subset [-R/4,R/4]$ and $\hat{\chi}|_{[-R/8,R/8]} \equiv x^{-1}$.  
The both functional calculi $\chi(D)$ and $\chi (\overline{D}_\infty)$ are pseudo-local bounded operators with finite propagation and moreover, $\chi(D)^2-1 $ and $\chi(\overline{D}_\infty)^2 -1 $ lie in the Roe algebra. 
By \cite{roePartitioningNoncompactManifolds1988}*{Proposition 2.3}, the propagation of $\chi(D)$ is less than $R/4$ and $\chi(D)\xi$ depends only on geometric data of the $R/2$-neighborhood of the support of $\xi \in L^2(\widetilde{M})$ (the same is also true for $\chi(\overline{D}_\infty)$). This and \cref{prp:lift} show that
\[s ( \chi(D))= P\chi (\overline{D}_\infty )P \text{ modulo $D^*(\overline{\sN} \subset \overline{\sM})^\Pi$}.\]
\end{para}

We review the construction of K-theory classes associated with elliptic operators with the geometric origin.

\begin{para}\label{para:Dirac}
Let $M$ be an $n$-dimensional closed spin manifold. 
Then the Dirac operator $\slashed{D}_M \colon L^2(\widetilde{M}, S) \to L^2(\widetilde{M}, S)$ is odd and self-adjoint.
The functional calculus $\chi (\slashed{D}_M)$ is bounded, odd, self-adjoint, pseudo-local and finite propagation, and its image in $ Q^*(\widetilde{M})^\Gamma $ is a unitary. 
Hence it determines the Dirac fundamental class
\[[M]:= [\chi (\slashed{D}_M)] \in \K_{n+1}(Q(\widetilde{M})^\Gamma). \]
More precisely, we consider the $\Cl_{0,q}$-Dirac operator when $n \equiv -q$ modulo $8$, which determines an element of the $\KR_{n+1}$-group, as is discussed in \cref{rmk:odd_RealK}. 

In the same way the Dirac operator $\slashed{D}_{\sM}$ on $\overline{\sM}_\infty$ determines 
\[ [\overline{\sM} , \overline{\sN}]:=[ P\chi (\slashed{D}_\sM )P ] \in \K_{n+1}(Q^*(\overline{\sM})^\Pi /Q^*(\overline{\sN} \subset \overline{\sM})^\Pi),\]
where $P$ is the projection onto $L^2(\overline{\sM})$. By \cref{lem:Roe}, we have 
\begin{align*}
    s(\chi (\slashed{D}_M)) =  P\chi (\slashed{D}_{\sM})P \in Q^*(\overline{\sM}) ^\Pi  / Q^*(\overline{\sN} \subset \overline{\sM})^\Pi, 
\end{align*}
and hence 
\[s_*[\widetilde{M}] = [\overline{\sM}, \overline{\sN}] \in \K_{n+1}(Q^*(\overline{\sM}) ^\Pi  / Q^*(\overline{\sN} \subset \overline{\sM})^\Pi) \cong \K_n^\Pi (\overline{\sM},\overline{\sN}). \]
\end{para}

\begin{para}\label{para:psc}
Let us assume that an $n$-dimensional closed spin manifold $M$ is imposed a positive scalar curvature metric $g_M$. Then, as is proved in \cite{guoLichnerowiczVanishingTheorem2020}*{Theorem 1.1}, the operator $ \chi(\slashed{D}_M)$ is invertible in $D^*(\widetilde{M})^\Gamma$. 
Hence it determines the $\K$-theory class
\[\rho (g_M) := [\chi(\slashed{D}_M)] \in \K_{n+1}(D^*(\widetilde{M})^\Gamma) \]
called the higher $\rho$-invariant. 

We additionally assume that the psc metric $g$ is of the form $g_{\bD^2} + g_N$ on a tubular neighborhood $U = N \times \bD^2$ of $N$, where $g_{\bD^2}$ denotes the standard flat metric on $\bD^2$. 
We extend the metric on $\sM$ to a Riemannian metric $g_\sM$ on $\overline{\sM}_\infty$ which has uniformly positive scalar curvature and of the form $ds^2 + g_{\sN}$ on $\sN \times \bR_{>0}$, where $ds^2$ is the metric on $\bR_{>0}$. 
The Dirac operator $\slashed{D}_\sM $ on $\overline{\sM}_\infty$ with respect to the pull-back Riemannian metric on $\overline{\sM}_\infty$ determines an invertible element 
$\chi (\slashed{D}_\sM) \in D^*(\overline{\sM}_\infty )^\Pi$. 
Through the isomorphism 
\[D^*(\overline{\sM}_\infty)^\Pi / D^*(\overline{U}_\infty \subset \overline{\sM}_\infty) \cong D^*(\overline{\sM})^\Pi/ D^*(\overline{\sN} \subset \overline{\sM})^\Pi, \]
it determines the higher $\rho$-invariant 
\[\rho(g_\sM ) := [P\chi (\slashed{D}_\sM )P] \in \K_{n+1}(D^*(\overline{\sM})^\Pi/ D^*(\overline{\sN} \subset \overline{\sM})^\Pi). \]
By \cref{lem:Roe} we have $s(\chi (\slashed{D}_M)) = P\chi (\slashed{D}_{\sM})P$, which implies
\[ s_* (\rho (g_M)) = \rho (g_\sM). \]
\end{para}

\begin{para}\label{para:sign}
Let $M$ be an $n$-dimensional closed oriented manifold. If $n$ is even, the signature operator $D_M^\sgn := d + d^*$ acting on $L^2(\widetilde{M}, \lwedge ^*_\bC TM)$ is an odd self-adjoint elliptic operator, and hence determines a K-theory class
\[[M]_\sgn := [ \chi(D^\sgn_M ) ] \in \K_{n+1}(Q^*(\widetilde{M})^\Gamma). \]
A similar construction also works when $n$ is odd (see \cref{rmk:odd_RealK} for a more detail).
In the same way, the signature operator $D^\sgn_\sM$ on $\overline{\sM}_\infty  $ also determines
\[ [\bar{\sM} , \bar{\sN}]_\sgn :=[P\chi (D^\sgn_\sM  )P] \in \K_{n+1}(Q(\overline{\sM})^\Pi / Q(\overline{\sN} \subset \overline{\sM})^\Pi ). \]
By \cref{lem:Roe}, we have
\begin{align*}
    s(\chi (D_M^\sgn )) = P\chi (D_\sM^\sgn )P \in Q^*(\overline{\sM}) ^\Pi  / Q^*(\overline{\sN} \subset \overline{\sM})^\Pi,
\end{align*}
and hence
\[ s_* [M]_\sgn = [\sM, \sN ]_\sgn \in \K_{n+1}(Q^*(\overline{\sM}) ^\Pi  / Q^*(\overline{\sN} \subset \overline{\sM})^\Pi ) \cong \K_n^\Pi (\overline{\sM}, \overline{\sN}) .\]
\end{para}

\begin{para}\label{para:homotopy}
Let $f \colon M' \to M$ be an oriented homotopy equivalence of $n$-dimensional manifolds. We construct an invertible perturbation of the signature operator 
following the construction of \cite{hilsumInvarianceParHomotopie1992}*{Section 3.2} (we also refer the reader to \cite{piazzaBordismRhoinvariantsBaumConnes2007}).  
We choose a submersion $F \colon \bD^k \times M' \to M$ such that $F|_{\{ 0\} \times M'} =f$ and a differential form $v \in \Omega_0^k(\bD^k)$ with $\int_{\bD^k} v =1$. We write $p \colon M' \times \bD^k \to M'$ for the projection. 
We define the linear map $T:= p_! \circ e(v) \circ F^*$, where $F^*$ denotes the pull-back, $e(v)$ denotes the exterior product with $v$, and $p_!$ denotes the fiberwise integral along $p$. Namely
\[T (\omega ):= \int_{\bD^k} v \wedge F^*(\omega). \]
As is shown in \cite{hilsumInvarianceParHomotopie1992}*{Lemma 3.2}, there are operators $y \colon \Omega^*(\widetilde{M}) \to \Omega^*(\widetilde{M})$ and $z \in \Omega^*(\widetilde{M}') \to \Omega^*(\widetilde{M}')$ such that $1-T^\dagger T = yd_{M}+ d_M y$ and $1-TT^\dagger = zd_{M'} + d_{M'}z$.

Let $\gamma$ denote the involution acting on $\Omega^p$ by $(-1)^p$, and let $\tau$ denote the Hodge $\ast$-operator (we use the same letter for the operators on both $M$ and $M'$). Let $a^\dagger$ stand for $\tau a^* \tau $. 
Following to \cite{hilsumInvarianceParHomotopie1992}*{Lemma 2.1}, we define the operators
\[d_{f_M } = \begin{pmatrix} d_M & \alpha T^\dagger  \\ 0 & -d_{M'} \end{pmatrix} _{\textstyle ,} \ \ \ L := \begin{pmatrix} 1 - T^\dagger T & (i\gamma + \alpha y) T^\dagger \\ T(i\gamma + \alpha y) & 1 \end{pmatrix}_{\textstyle ,}\]
where $\alpha >0$ is chosen in the way that $L$ is invertible (the existence of such $\alpha >0$ is shown in the proof of \cite{hilsumInvarianceParHomotopie1992}*{Lemma 2.1}). 
Set $S:= \tau L \cdot | \tau L|^{-1}$. 

Following the idea of \cite{guoLichnerowiczVanishingTheorem2020}*{Section 3}, we regard the operators as regular unbounded multipliers on 
the maximal uniform Roe algebra $C^*_{u,\mathrm{max}}(\widetilde{M}, \sH_f )^\Gamma$ with respect to the $\Gamma$-equivariant $C_0(\widetilde{M})$-module
\[ \sH_f:= L^2(\tilde{M}, \lwedge ^*_\bC TM) \oplus L^2(\widetilde{M}', \lwedge ^*_\bC TM').\] 
Then the relations $1-T^\dagger T = yd_{M}+ d_M y$ and $1-TT^\dagger = zd_{M'} + d_{M'}z$ show that the differential $d_{f_M}$ is acyclic as operators on this Hilbert module, and hence, the operator
\[ 
D_{f_M}^\sgn := 
\begin{cases}
(d_{f_M} - S d_{f_M} S ) \circ |\tau L|^{-1} & \text{ if $n$ is even, }\\
-i(d_{f_M} S  +S d_{f_M} ) \circ |\tau L|^{-1} & \text{ if $n$ is odd,}
\end{cases}
\]
is self-adjoint and invertible. Moreover it anticommutes with the $\bZ_2$-grading $S$ if $n$ is even. 
We also remark that this operator satisfies the assumption of \cref{lem:Roe} since both $T$ and $y$ has zero propagation. We define the signature higher $\rho$-invariant as
\[\rho_\sgn (f_M):= [\chi (D_{f_M}^\sgn)  ] \in \K_{n+1}(D^*(\widetilde{M})^\Gamma ). \]

We additionally assume that $f_M$ restricts to a homotopy equivalence $f_N \colon N' \to N$ of codimension 2 submanifolds. We choose the submersion $F$ as $F|_{N'}$ is also a submersion. Then $f_M$ and $F$ extends to a homotopy equivalence $f_\sM \colon \sM'_\infty \to \sM_\infty$ and a submersion $F_{\sM} \colon \sM'_\infty  \times \bD^k \to \sM_\infty $ respectively. By the same construction for these maps, we obtain the operator $D_{f_\sM}^\sgn $, which is locally a lift of $D_{f_M}^\sgn$. It determines the higher $\rho$-invariant 
\[\rho_\sgn (f_{\sM}):= [P\chi (D_{f_\sM}^\sgn)P ] \in \K_{n+1}(D^*(\overline{\sM})^\Pi  / D^*(\overline{\sN} \subset \overline{\sM})^\Pi  ). \]
By \cref{lem:Roe}, we have
\begin{align*}
    s(\chi (D_{f_M}^\sgn )) = P\chi (D_{f_\sM}^\sgn )P \in D^*(\overline{\sM}) ^\Pi  / D^*(\overline{\sN} \subset \overline{\sM})^\Pi,
\end{align*}
and hence
\[s_* \rho_\sgn (f_M) = \rho_\sgn (f_\sM ) \in \K_{n+1}(D^*(\overline{\sM}) ^\Pi  / D^*(\overline{\sN} \subset \overline{\sM})^\Pi). \]
\end{para}

The `boundary of Dirac is Dirac' and `boundary of signature is $2^\epsilon$ times signature' principles claims the following equalities.  
\begin{lem}\label{lem:boundary}
The following equalities hold;
\begin{enumerate}
    \item $\partial_Q ([\overline{\sM} , \overline{\sN}]) = [\overline{\sN}] = [N]$,
    \item $\partial_Q ([\overline{\sM} , \overline{\sN}]_\sgn  ) =2^\epsilon [\overline{\sN}]_\sgn  = 2[N]_\sgn $,
    \item $\partial_D ( \rho (g_{\sM , \sN})) = \rho (g_{\sN}) = \rho (g_N)$,
    \item $\partial_D ( \rho_\sgn  (f_{\sM,\sN})) = 2^\epsilon \rho_\sgn (f_\sN) = 2 \rho_\sgn (f_N)$,
\end{enumerate}
where $\epsilon \in \{ 0,1\}$ is determined as $\epsilon = 0$ if $\dim M$ is odd and $\epsilon =1 $ if $\dim M$ is even. Here, each of the second equalities are considered under the isomorphisms of K-groups given in \cref{lem:pi_zero}.
\end{lem}
\begin{proof}
Each of the first equalities are the `boundary of Dirac is Dirac' or the `boundary of signature is $2^\epsilon$ times signature' principles, which is proved in \cite{higsonAnalyticHomology2000}*{Proposition 11.2.15}, \cite{higsonCalgebraicHigherSignatures2018}*{2.13}, \cite{piazzaRhoclassesIndexTheory2014}*{Theorem 1.22} (see also \cite{zeidlerPositiveScalarCurvature2016}*{Theorem 5.15}) and  \cite{weinbergerAdditivityHigherRho2020}*{Appendix D} respectively. 
Each of second equalities again follows from the `boundary of Dirac is Dirac' or the `boundary of signature is $2^\epsilon$ times signature' principles, applied for the manifold $\sN \cong N \times \bR$.
We remark that the proof of (4) requires a further discussion because our construction of the signature higher $\rho$-invariant is not the same as the one dealt with in \cite{weinbergerAdditivityHigherRho2020}. 
This is detailed in \cref{section:product}.
\end{proof}

\begin{thm}\label{thm:second_tr}
Let $M$ be an $n$-dimensional closed manifold and let $N$ be a codimension $2$ submanifold of $M$ satisfying (1), (2) and (3) of \cref{prp:NSZ}.
\begin{enumerate}
\item Let $M$ be a spin manifold. Then we have
\[\tau _\sigma ([M]) = [N] \in \K_{n-1}(Q^*(\widetilde{N})^\pi) \cong \K_{n-2}(N).\]
\item Let $M$ be oriented. Then we have
\[\tau _\sigma ([M]_\sgn ) = [N]_\sgn \in \K_{*-1}(Q^*(\widetilde{N})^\pi) \cong \K_{n-2}(N).\]
\item Let $M$ be spin and equipped with a psc metric $g_M$ which is of product-like on $U \cong N \times \bD^2$. Then we have
\[\tau_\sigma (\rho(g_M)) = \rho(g_N) \in \K_{n-1}(D^*(\widetilde{N})^\pi). \]
\item Let $f \colon M' \to M$ be a oriented homotopy equivalence such that the restriction $f|_N \colon f^{-1}(N) \to N$ is also a homotopy equivalence. Then we have
\[\tau_\sigma (\rho_\sgn (f_M)) = \rho_\sgn (f_N) \in \K_{n-1}(D^*(\widetilde{N})^\pi). \]
\end{enumerate}
\end{thm}

\begin{proof}
They are already proved in \cref{para:Dirac}, \cref{para:psc}, \cref{para:sign}, \cref{para:homotopy} and \cref{lem:boundary}.
\end{proof}

This, in combination with \cref{thm:HigsonRoe}, reproves the main theorem of \cite{kubotaGromovLawsonCodimensionObstruction2020}.
\begin{cor}[{\cite{kubotaGromovLawsonCodimensionObstruction2020}*{Theorem 1.2, Theorem 1.3}}]
The codimension 2 transfer map $\tau_\sigma $ sends the Rosenberg index $\alpha_\Gamma (M)$ to $\alpha_\pi (N)$ and the higher signature $\Sgn _\Gamma (M)$ to $2\Sgn_\pi (N)$.
\end{cor}

\begin{exmp}\label{exmp:bundle2}
Let $\pi$ be a finite cyclic group equipped with $\pi \to \pi_0(\mathop{\mathrm{Diff}}(\Sigma_{g,1}, \partial ))$ as in \cref{exmp:bundle}. 
Let $N$ be a closed manifold with $\pi_1(N) \cong \pi$ and let $M \to N$ be the surface bundle as in \cref{exmp:bundle}. 
\begin{itemize}
\item Assume that $N$ is spin and is equipped with a psc metric $g_N$ such that $\rho(g_N) \neq 0$. 
This assumption is satisfied if there is a unitary representation $v \in R(\pi)$ such that the $\rho$-invariant $\rho_v(g_N) $ is non-zero, or alternatively, there is a conjugacy class $\langle h \rangle \in \langle \pi \rangle $ such that the delocalized $\eta$-invariant \cite{lottDelocalizedInvariants1999}*{Definition 7} is non-zero (this assumption implies $\rho(g_N) \neq 0$ \cite{xieDelocalizedEtaInvariants2019}*{Theorem 1.1}).
Moreover, we assume that the psc metric $ds^2 + g_N $ on $N \times S^1$ extends to a psc metric $g_M$ on $M$. 
Then we have $\tau_\sigma (\rho(g_M)) = \rho (g_N) \neq 0$, and hence $\rho(g_M) \neq 0$. This shows that $(M,g_M)$ is not psc null-cobordant. 
\item Let $f_N \colon N' \to N$ be an oriented homotopy equivalence with non-trivial higher $\rho$-invariant $\rho_\sgn (f_N)$. In the same way as above, the non-vanishing of the $\rho$-invariants or delocalized $\eta$-invariants is sufficient for this. 
Then $M':=f_N^*N'$ is also a surface bundle and $f_N$ extends to an oriented homotopy equivalence $f_M \colon M' \to M$. Then $\tau _\sigma (\rho_\sgn (f_M)) = \rho_\sgn (f_N) \neq 0$, and hence $\rho _\sgn (f_M) \neq 0$. This shows that $f_M \colon M' \to M$ is not $h$-cobordant to the identity.  
\end{itemize}
\end{exmp}

\subsection{Geometric model of the codimension 2 transfer}
We have given two different approaches for defining the codimension $2$ transfer map of K-homology groups $\tau _\sigma \colon \K_*(M) \to \K_{*-2}(N)$; the cohomological construction in \cref{defn:tau}, \cref{rmk:intersection} and the analytic (coarse geometric) one in  \cref{defn:transfer}. 
Here we show that they are actually the same.
We deal with the Baum--Douglas geometric model of the K-homology group  \cites{baumHomologyIndexTheory1982,baumEquivalenceGeometricAnalytic2007} as the domain and the range of the transfer map. 
We recall that a K-homology cycle is represented by a triple $(W,f,E)$, where $W$ is a $\mathrm{spin}^c$ manifold, $f \colon W \to M$ and $E$ is a complex vector bundle over $W$. 

\begin{lem}\label{lem:transfer_geom}
Let $(W,f,E)$ be a geometric K-homology cycle of $M$. Assume that $f$ is transverse to $N$ and set $X:=f^{-1}(N)$. 
Then the map $\tau_\sigma \colon \K_*(M) \to \K_*(N)$ sends $[W,f,E]$ to $[X,f|_X,E|_X]$. 
\end{lem}
\begin{proof}
The definition of $\tau_\sigma$ given in \cref{defn:tau} and \cref{rmk:intersection} is given by the cap product with $\sigma  \in \K^2(\widetilde{M}/\pi , \widetilde{M}/\pi \setminus U) \cong \K^0(U) \cong \K^2(M,M \setminus U)$. This $\sigma$ is by definition represented by $[L, \underline{\bC}, s]$, where $L$ is the line bundle on $M$ representing $\sigma \in H^2(U;\bZ)$ and $s$ is a non-vanishing section of $L$ on $M \setminus U$. We have
\begin{align*}
    [L,\underline{\bC}, s] \cdot [W,f,E] =& [W, f, E \otimes L] - [W,f,E]
    = [W \sqcup W^{\mathrm{op}}, f \sqcup f, E\otimes L \sqcup E] \\
    =& [\bS ((L \oplus \underline{\bR})|_X) , f|_X \circ q , q^*(( E \otimes L)|_X) ], 
\end{align*} 
where $W^{\mathrm{op}}$ is $W$ with the opposite orientation and $q \colon \bS((L \oplus \underline{\bR})_X) \to X$ denotes the projection. Here the last equality comes from a cobordism of K-homology cycles.  
Now the right hand side is the vector bundle modification of $[X,f|_X, E|_X]$ by the trivial bundle $L \to X$.
\end{proof}
The identification of geometric and analytic K-homology groups is given by the map $\Phi_M \colon \K_*(M) \to \K_{*+1}(Q^*(\widetilde{M})^\Gamma)$ defined as 
\[\Phi_M ([W,f,E]) := (\Ad V_f)_* ([\chi (D_{W}^E)]) \in \K_{*+1}(Q^*(\widetilde{M})^\Gamma).  \]
Here, $\chi$ is a function as in \cref{para:Dirac}, $D_W^E$ is the $\mathrm{spin}^c$ Dirac operator on $\widetilde{W}$ twisted by the vector bundle $E$, and $V_f \colon L^2(\widetilde{W}) \to L^2(\widetilde{M})$ is a $\Gamma$-equivariant covering isometry of $f$. That is, $V_f$ is a $\Gamma$-equivariant isometry such that $V_f^*\varphi  V_f - f^*\varphi \in \bK(L^2(\widetilde{W}))$ for any $\varphi \in C_c(\widetilde{M})$ and there is $R>0$ such that $\mathrm{supp}(V_f)$ is included to the $R$-neighborhood of the graph of $f$ in $M \times W$. 
Note that such an isometry induces $\ast$-homomorphisms between $C^*$, $D^*$ and $Q^*$ coarse C*-algebras.  
\begin{lem}\label{lem:wrongway}
Let $M$ and $W$ be manifolds with fundamental group $\Gamma $ and let $f \colon W \to M$ be a smooth map inducing an isomorphism of fundamental groups. 
Let $N$ be a codimension $2$ submanifold of $M$. 
Assume that $f$ is transverse to $N$ and set $X:=f^{-1}(N)$. 
Let $V_f \colon L^2(\widetilde{W}) \to L^2(\widetilde{M})$ and $V_f' \colon L^2(\widetilde{X}) \to L^2(\widetilde{N})$ be equivariant covering isometries. Then the diagram
\[
\xymatrix@C=8ex{
\K_*(C^*(\widetilde{W})^\Gamma ) \ar[r]^{(\Ad V_f)_*} \ar[d]^{\tau_\sigma} & \K_*(C^*(\widetilde{M})^\Gamma ) \ar[d]^{\tau_\sigma} \\
\K_*(C^*(\widetilde{X})^\pi) \ar[r]^{(\Ad V_f')_*} & \K_*(C^*(\widetilde{N})^\pi)
}
\]
commutes. The same commutativity also holds for both $D^*$ and $Q^*$ coarse C*-algebras. 
\end{lem}
\begin{proof}
Let $\sW := (\widetilde{W}/\pi) \setminus X$, let $\overline{\sW} $ be its $\Pi$-Galois covering, and let $\overline{\sX}:= \partial \overline{\sW} \cong \widetilde{X} \times \bR$.  
Let $p \colon \overline{\sM} \to \widetilde{\sM}$ denote the projection and let $\bar{f} \colon \overline{\sW} \to \overline{\sM}$ be the lift of $f$. 
Then there is a covering isometry $V_{\bar{f}}$ such that, for any open subset $U \subset \overline{\sM}$ such that $p|_U$ is injective, 
$V_{\bar{f}}|_{L^2(\bar{f}^{-1}(U))}$ is identified with $V_f|_{L^2(f^{-1}(p(U)))}$ by $p$.  
Then the diagram
\[
\xymatrix{
C^*(\widetilde{W})^\Gamma \ar[r]^{s \hspace{6ex}} \ar[d]^{\Ad V_f} & C^*(\overline{\sW})^\Pi / C^*(\overline{\sX} \subset \overline{\sW})^\Pi \ar[d]^{\Ad V_{\bar{f}}} \\
C^*(\widetilde{M})^\Gamma \ar[r]^{s \hspace{6ex}}  & C^*(\overline{\sM})^\Pi / C^*(\overline{\sN} \subset \overline{\sM})^\Pi 
}
\]
commutes. 
Since $(\Ad V_{\bar{f}})_*$ is also compatible with the K-theory boundary map, this finishes the proof. The same proof also works for $D^*$ and $Q^*$ coarse C*-algebras. 
\end{proof}

\begin{prp}
The following diagram commutes; 
\[
\xymatrix{
\K_*(M) \ar[r]^{\Phi_{M} \hspace{3ex}} \ar[d]^{\tau_\sigma} & \K_{*+1}(Q^*(\widetilde{M})^\Gamma ) \ar[d] ^{\tau_\sigma} \\
\K_{*-2}(N) \ar[r] ^{\Phi_{N}\hspace{3ex}} & \K_{*-1}(Q(\widetilde{N})^\pi).
}
\]
\end{prp}
\begin{proof}
An element of $\K_*(M)$ is represented by a triple $(W,f,E)$. Without loss of generality we may assume that $f$ is transverse of $N$. Then we have
\begin{align*} 
(\tau_\sigma \circ \Phi_M)( [W,f,E]) =& (\tau_\sigma \circ (\Ad V_f)_*)([\chi(D_W^E)]) = ((\Ad V_f')_* \circ \tau_\sigma) ([\chi (D_W^E)]) \\
=& (\Ad V_f')_* ([\chi (D_X^{E|_X})]) = \Phi_N ([X,f|_X, E|_X])\\
=& (\Phi_N \circ \tau_\sigma )([W,f,E]).  
\end{align*}
Here, the second equality is due to \cref{lem:wrongway}. 
For the third equality, we use the fact that a twisted $\mathrm{spin}^c$ Dirac operator is sent by $\tau_\sigma$  to its restriction to the codimension $2$ submanifold, which is shown by the same argument as \cref{para:Dirac}.
\end{proof}

We also study a geometric description of the codimension $2$ transfer map between analytic structure groups. 
Inspired from the work of Deeley and Goffeng  \cites{deeleyRealizingAnalyticSurgery2017}, in which a geometric model of the analytic structure group $\K_*(D^*(\widetilde{M})^\Gamma )$ is established, we study the map
\begin{align*}
    \Phi_{B\Gamma , M} \colon \K_*(B\Gamma , M) \to \K_*(D^*(\widetilde{M})^\Gamma ) 
\end{align*}
defined by 
\begin{align}
    \Phi_{B\Gamma, M} ([W,f,E]) :=   (\Ad V_f)_*([\chi (D_{\partial W}^E + A)])- j_* \ind _{\mathrm{APS}} (D_W^E, A), \label{eq:geom_rel}
\end{align} 
where $A$ is an arbitrary choice of smoothing invertible perturbation of $D_{\partial W}^E$ and $\ind _{\mathrm{APS}} (D_W^E, A)$ is the higher Atiyah--Patodi--Singer index with respect to the boundary condition $A$, in the sense of \cite{piazzaSurgeryExactSequence2016}*{Definition 2.27}. 
Alternatively, this element is also defined by using coarse Mayer--Vietoris boundary map as
\begin{align}
\Phi_{B\Gamma, M}([W,f,E]) := ((\Ad V_f )_* \circ  \partial _{MV}) ([\chi (D_{\bD W}^E)]), \label{eq:geom_rel_MV}
\end{align}
where $\bD W$ denotes the invertible double of $W$. Note that the operator $D_{\bD W}^E$ is invertible by \cite{xieRelativeHigherIndex2014}*{Theorem 5.1}. 
\begin{prp}\label{lem:double}
The maps \eqref{eq:geom_rel} and \eqref{eq:geom_rel_MV} coincides. 
\end{prp}
\begin{proof}
Let $A_W$ be a smoothing perturbation such that $\ind _{\mathrm{APS}} (D_W^E, A_W) =0$ (the existence of such $A_W$ is proved in e.g.~\cite{kubotaRelativeMishchenkoFomenkoHigher2020}*{Lemma 4.6}). Let $D_{W, \infty}^E + A_{W, \infty}$ be the perturbed Dirac operator on $\widetilde{W}_\infty:= \widetilde{W} \sqcup_{\partial \widetilde{W}} \widetilde{W} \times [0,\infty)$ as in \cite{piazzaSurgeryExactSequence2016}*{Definition 2.20}. By the triviality of the APS index, there is a smoothing operator $B$ supported on $\widetilde{W}$ such that $D_{W, \infty}^E + A_{W, \infty} + B$ is invertible. Then we get the conclusion since
\[ P \chi(D_{W, \infty}^E + A_{W, \infty} + B) P =  P \chi (D_{\bD W}^E) P \in D^*(\widetilde{W})^\Gamma / D^*(\partial \widetilde{W} \subset \widetilde{W})^\Gamma, \]
where $P$ denotes the projection onto the $L^2$-section space on $\widetilde{W}$, and
\[ \partial_{MV}([P\chi(D_{W, \infty}^E + A_{W, \infty} + B)P]) = [\chi (D_{\partial W}^E + A_W)] . \qedhere \]
\end{proof}

The definition \eqref{eq:geom_rel} of $\Phi_{B\Gamma, M}$ shows that the diagram 
\[
\mathclap{
\xymatrix{
\ar[r] & \K_*(M) \ar[r] \ar[d]^{\Phi_M} & \K_* (B\Gamma ) \ar[r] \ar[d]^{\ind } & \K_*(B\Gamma , M) \ar[d]^{\Phi_{B\Gamma, M}} \ar[r] & \K_{*-1}(M) \ar[r] \ar[d]^{\Phi_M} & \cdots \\
\ar[r] & \K_{*+1}(Q^*(\widetilde{M})^\Gamma ) \ar[r] & \K_*(C^*(\widetilde{M})^\Gamma ) \ar[r] & \K_*(D^*(\widetilde{M})^\Gamma) \ar[r] & \K_*(Q^*(\widetilde{M})^\Gamma ) \ar[r] & \cdots 
}}
\]
commutes. 

We compare the analytic codimension $2$ transfer map given in \cref{defn:transfer} with the relative cohomological transfer defined in the same way as \cref{defn:tau} by the composition
\[ \tau _\sigma \colon \K_*(B\Gamma, M) \to \K_*(E\Gamma /\pi , \widetilde{M}/\pi \cup U) \xrightarrow{\sigma \cap \cdot } \K_*(U, U \cap \widetilde{M}/\pi) \to \K_*(B\pi, N). \]
Recall that a relative K-homology cycle of the pair $(B\Gamma , M)$ is represented by a triple $(W,f,E)$, where $W$ is a compact $\mathrm{spin}^c$-manifold with boundary, $f \colon (W, \partial W) \to (B\Gamma , M)$ and $E$ is a complex vector bundle on $W$.
\begin{lem}\label{lem:transfer_geom_rel}
Let $(W,f,E)$ be a relative geometric K-homology cycle of the pair $(B\Gamma, M)$. Assume that $f|_{\partial W} $ is transverse to $N$ and set $\partial X:= f^{-1}(N)$. 
Then there is a submanifold-with-boundary $(X, \partial X) \subset (W, \partial W)$ such that the map $\tau_\sigma \colon \K_*(M) \to \K_*(N)$ sends $[W,f,E]$ to $[X,f|_{X},E|_X]$. 
\end{lem}
\begin{proof}
The existence of such $X$ follows from $[\partial X] = f^* \sigma \in H^2(\partial W ; \bZ)$, which shows that there is a line bundle $L \to W$ such that $\partial X $ is the zero locus of a section $s \colon \partial W \to L|_{\partial W}$. 
The lemma is proved completely in the same way as \cref{lem:transfer_geom}. 
\end{proof}

\begin{thm}\label{thm:geom_transfer_rel}
The following diagram commutes; 
\[
\xymatrix{
\K_*(B\Gamma, M) \ar[r]^{\Phi_{B\Gamma, M}} \ar[d]^{\tau_\sigma} & \K_*(D^*(\widetilde{M})^\Gamma ) \ar[d] ^{\tau_\sigma} \\
\K_*(B\pi , N) \ar[r] ^{\Phi_{B\pi, N}} & \K_*(D(\widetilde{N})^\pi).
}
\]
\end{thm}
\begin{proof}
Let $\bD W$ and $\bD X$ denote the invertible double of $W$ and $X$ respectively. Let $\widetilde{\bD W}$ denote the $\Gamma$-covering of $\bD W$, let $\overline{\bD \sW}$ denote the $\Pi$-covering of $(\widetilde{\bD W}/\pi) \setminus \bD X$, and let $\overline{\bD \sX}$ denote the boundary of $\overline{\bD \sW}$, which is $\Pi$-equivariantly diffeomorphic to $\widetilde{\bD X} \times \bR$.  
Then the diagram
\[
\mathclap{
\xymatrix{
\K_*(D^*(\widetilde{\bD W})^\Gamma ) \ar[r] \ar[d]^{\partial _{MV}} & \K_*\Big( \frac{D^*( \overline{\bD \sW})^\Pi }{D^*(\overline{\bD \sX} \subset \overline{\bD \sW})^\Pi} \Big) \ar[r] \ar[d]^{\partial _{MV}} & \K_*(D^*(\overline{\bD\sX})^\Pi ) \ar[d]^{\partial _{MV}} \ar[r] & \K_*(D^*(\widetilde{\bD X})^\pi ) \ar[d]^{\partial_{MV}}  \\
\K_*(D^*(\widetilde{\partial W})^\Gamma ) \ar[r] & \K_*\Big(\frac{D^*(\overline{\partial \sW})^\Pi }{D^*(\overline{\partial \sX} \subset \overline{\partial \sW})^\Pi }  \Big) \ar[r] & \K_*(D^*(\overline{\partial \sX})^\Pi ) \ar[r] & \K_*(D^*(\widetilde{\partial X})^\pi )
}
}
\]
commutes. This shows the commutativity of $\tau_\sigma$ and $\partial_{MV}$, and hence we get
\begin{align*}
    (\tau_\sigma \circ \Phi_{B\Gamma, M})([W,f,E]) & = (\tau_\sigma \circ  (\Ad V_f)_* \circ \partial_{MV})(\chi (D_{\bD W}^E))) \\
    &= ((\Ad V_f')_* \circ \tau _\sigma \circ \partial_{MV}) ([\chi (D_{\bD W}^E)] ) \\
    &= ((\Ad V_f')_* \circ \partial _{MV} \circ \tau _\sigma ) ([\chi (D_{\bD W}^E)] ) \\
    &=  ((\Ad V_f')_* \circ \partial _{MV}) ([\chi (D_{\bD X}^E)]) \\
    &= \Phi_{B\pi , N  }([X, f|_X, E|_X])\\
    &=  (\Phi_{B\pi, N} \circ \tau_\sigma) ([W,f,E]).\qedhere 
\end{align*}   
\end{proof}

We remark that the same discussion also work in Real geometric K-homology only by replacing a triple $(W,f,E)$ with a representative of Real K-homology cycles, i.e., a spin manifold $W$, $f \colon W \to M$, and a real vector bundle $E$ on $W$ (cf.~\cite{reisKOhomologyTypeString2009}*{Section 2}). 

\section{Codimension 2 transfer and higher index pairing}\label{section:cyclic}
In this section, we prove the compatibility of the codimension 2 transfer maps 
studied in previous sections and the co-transfer map in cyclic cohomology groups. 
This reduces computations of the Connes--Moscovici higher index pairing and Lott's higher $\rho$-numbers of $M$ to those of $N$.

\subsection{Codimension 2 co-transfer in cyclic cohomology}
For a general theory of cyclic homology and cohomology groups, we refer the readers to standard references \cites{connesNoncommutativeGeometry1994,cuntzCyclicHomologyNoncommutative2004,lodayCyclicHomology1998}. 
We define the $\ast$-subalgebra of smooth functions $C^\infty(\cB)$ on $\cB$, i.e., 
\begin{align*} 
C^\infty(\cB)  := \big ( \bigoplus \nolimits_{\Gamma /\pi} C^\infty_0(0,1) \big)^+  \big( c_c(\Gamma /\pi) \otimes_{\alg} C^\infty_0(0,1) \big)^+.
\end{align*}
Here $C_0^\infty(0,1)$ denotes the space of smooth functions on $\bT$ which vanishes at $1 \in \bT$. 
Then $C^\infty(\cB)$ contains the functions $\sigma(g,h)$ for any $g,h \in \Gamma$, and hence the algebraic twisted crossed product $C^\infty(\cB) \rtimes _{\sigma, \alg} \Gamma$ makes sense. 
In the same way as \ref{lem:untwist}, the cocycle $\sigma$ is untwisted when it is restricted to the ideal $C^\infty_0(\cB_0):= \bigoplus \nolimits_{\Gamma /\pi} C^\infty_0(0,1) $, we obtain a $\ast$-isomorphism
\[ \mathclap{C^\infty_0(\cB_0) \rtimes_{\alg, \sigma} \Gamma \cong C^\infty_0(\cB_0) \rtimes_{\alg} \Gamma  \cong \sS\bC[\pi] \otimes   \bK_\alg,} \]
where $\sS\bC[\pi]:=\bC[\pi] \otimes C^\infty_0(0,1)$, and hence an exact sequence
\begin{align} 0 \to \sS\bC[\pi] \otimes \bK_\alg \to C^\infty(\cB) \rtimes_{\sigma, \alg} \Gamma \to \bC[\Gamma ] \to 0. \label{eq:alg_ext} 
\end{align}

The periodic cyclic homology $\HP_\ast (A)$ of a unital $\bC$-algebra is $A$ the homology group of the periodic cyclic complex $(\CC_{\bullet }(A)[\![v^{\pm 1}]\!], b+vB )$, where $\CC_{\bullet }(A):=A^+ \otimes  A ^{\otimes p}$, $v$ is a degree $-2$ formal variable and 
\begin{align*}
    b(\widetilde{a}_0 \otimes \cdots \otimes a_n) =& \sum_{j=0}^{n-1} (\widetilde{a}_0 \otimes \cdots \otimes a_ja_{j+1} \otimes \cdots a_n) + a_n\widetilde{a}_0 \otimes a_1 \otimes \cdots \otimes a_{n-1}, \\
    B(\widetilde{a}_0 \otimes \cdots \otimes a_p) =& \sum_{j=0}^p (-1)^{pj} 1 \otimes a_j \otimes \cdots \otimes a_p \otimes a_0 \otimes \cdots \otimes a_{j-1} .
\end{align*}
The cyclic homology $\HC_*(A)$ is the homology group of the truncated complex $\CC_\bullet(A)[\![ v^{\pm 1} ]\!]/CC_\bullet (A)[\![v]\!]$. Note that $\HP_*(A)$ is obtained as the projective limit of $\HC_*(A)$ by Connes' $S$-map. 

Dually, the periodic cyclic cohomology $\HP^*(A)$ is the cohomology of the complex $(\CC^\bullet (A)[u^{\pm 1}], b + u B)$, where $\CC^n (A):=\Hom(\CC_n(A), \bC)$, $b$ and $B$ denotes the adjoint of the corresponding maps defined above, and $u$ is a degree $2$ formal variable. 

The periodic cyclic homology have the excision property \cite{lodayCyclicHomology1998}*{Theorem 2.2.17} implying the long exact sequence
\begin{align*}
\cdots 
\to \HP_{n}(\sS \bC[\pi] \otimes \bK_\alg ) 
&\to \HP_n(C^\infty(\cB) \rtimes_{\sigma, \alg} \Gamma) 
\to \HP_n(\bC[\Gamma]) 
\to \cdots . 
\end{align*}
Since $\sS\bC[\pi] \otimes \bK_\alg$ is Morita equivalent to $\sS\bC[\pi]$, there is an isomorphism $\HP_*(\sS\bC[\pi] \otimes \bK_\alg) \cong  \HP_{*-1}(\bC[\pi])$ (\cite{lodayCyclicHomology1998}*{Theorem 2.2.9}). 
\begin{defn}\label{def:codim2_cyclic}
We write $\tau_\sigma$ for the boundary map 
\[\mathclap{\tau_\sigma := \partial  \colon \HP_*(\bC [\Gamma] ) \to \HP_{*-1}(\sS \bC [\pi] \otimes \bK_\alg ) \cong \HP_{*-2}(\bC [\pi] ) } \]
of the cyclic homology associated to the exact sequence \eqref{eq:alg_ext}. Similarly, we define the co-transfer map $\tau_\sigma^*$ as the boundary map
\[\tau _\sigma^* :=\delta \colon \HP^{*-2}(\bC [\pi]) \cong  \HP^{*-1}(\sS\bC [\pi] \otimes \bK_\alg ) \to \HP^*(\bC [\Gamma]).  \]
\end{defn}
Due to the compatibility of the boundary map and the pairing of cyclic homology and cohomology groups, the equality
\[ \langle \tau_\sigma (\xi) , \phi \rangle_{\bC [\pi]} = \langle \xi, \tau_\sigma^* ( \phi) \rangle_{\bC [\Gamma] }  \]
holds for any $\xi \in \HP_*(\bC [\Gamma] )$ and $\phi \in \HP^{*-2}(\bC [\pi] )$.

Next, we observe that this codimension 2 transfer map respects the decomposition of the cyclic (co)homology group of the group algebra given by Burghelea \cite{burgheleaCyclicHomologyGroup1985}, which is generalized to the general crossed product algebra by Nistor \cite{nistorGroupCohomologyCyclic1990}.  
For $g \in \Gamma$, let $\langle g\rangle $ denote the conjugacy class containing $g$. Let $\langle \Gamma \rangle$ denote the set of conjugacy classes of $\Gamma$. 
For a complex $\Gamma$-algebra $A$, the cyclic complex is decomposed as the direct sum of subcomplexes
\begin{align*}
(\CC_{\bullet }(A \rtimes _\alg \Gamma )[\![v^{\pm 1}]\!] , b+vB) \cong \bigoplus _{\langle g \rangle \in \langle \Gamma \rangle }    (\CC_{\bullet }(A \rtimes _\alg \Gamma ; \langle g \rangle )[\![v^{\pm 1}]\!] , b+vB),
\end{align*}
where
\begin{align*}
     &\CC_{p}(A \rtimes _\alg \Gamma  ; \langle g\rangle):= \overline{\mathrm{span}}\{ ( (a_0 u_{g_0} \otimes a_1 u_{g_1} \otimes \cdots \otimes a_ pu_{g_p} \mid g_0g_1 \cdots g_q \in  \langle g\rangle \} .
\end{align*}
The homology group of each subcomplex $\CC_{\bullet }( A \rtimes_\alg \Gamma  ; \langle g\rangle )[\![v^{\pm 1}]\!]$ is denoted by $\HP_\bullet (A \rtimes_\alg  \Gamma ; \langle g \rangle  )$. Then we have a decomposition
\begin{align}
    \HP_* (A \rtimes _\alg \Gamma ) \cong \bigoplus _{\langle g \rangle \in \langle \Gamma \rangle }    \HP_{*}(A \rtimes _\alg \Gamma ; \langle g \rangle ). \label{eq:decomposition}
\end{align}
Let $N_\Gamma (g)$ denote the normalizer subgroup of $g \in \Gamma$. 
As is shown in \cite{nistorGroupCohomologyCyclic1990}*{Lemma 2.7}, the inclusion $ A \rtimes_\alg N_\Gamma (g) \subset A \rtimes_\alg \Gamma$ induces an isomorphism
\[   \HP_*(A \rtimes_\alg N_\Gamma (g); \langle g \rangle )  \xrightarrow{\cong} \HP_*(A \rtimes_\alg \Gamma ; \langle g \rangle).\]

We discuss the decomposition \eqref{eq:decomposition} in the case that $A = c_c(\Gamma /\pi)$. 
The $N_\Gamma(g)$-action on $\Gamma/\pi$ decomposes into $X \sqcup Y$, where $X$ denotes the fixed point set of the left multiplication of $g$ and $Y$ denotes its complement. 
If $x\pi \in X$, then $g x\pi = x\pi$ implies that $x^{-1}gx \in \pi$. 
Moreover, if $x\pi , y \pi \in \Gamma /\pi$ are in the same orbit of the $N_\Gamma(g)$-action, then $\langle x^{-1} gx \rangle = \langle y^{-1}gy \rangle \in \langle \pi \rangle$ holds. Hence $X$ decomposes into the disjoint union $\bigsqcup_{\langle h \rangle \in \langle \pi ; g\rangle } X_h$, where $\langle \pi; g \rangle$ is the image of $N_\Gamma(g) \cap \pi$ to $\langle \pi \rangle$. 
The stabilizer subgroup of $N_\Gamma (g)$ at $x\pi \subset X_h$ is $N_\Gamma (g) \cap x \pi x^{-1}$. Therefore, $\Ad(x^{-1})$ gives an isomorphism of group actions
\[ N_\Gamma(g) \curvearrowright X_h \cong  N_{\Gamma } (h ) \curvearrowright N_\Gamma (h) / N_\pi (h).\] 
By abuse of notation, we use the same letter $X_h$ for $N_\Gamma(h)/N_\pi(h)$. 

For $h = x^{-1}gx \in \langle \pi ; g\rangle $, the following diagram
\[
\xymatrix@C=12ex{
\bC[ N_\pi(h)]  \ar[r]^{\Ad (x^{-1}) \circ j_1 \ \ \ } \ar[d]^{j_2 } &  
c_c(X_h) \rtimes_\alg N_\Gamma (g) \ar[d]^{j_3 } \\
\bC[ \pi]   \ar[r]^{\Ad (x^{-1}) \circ j_4 \ \ \ } &  c_c(\Gamma /\pi) \rtimes _\alg \Gamma 
}
\]
commutes, where $j_1$, $j_2$, $j_3$, $j_4$ are inclusions, considered by identifying $\bC[N_\pi(h)] $ with $c(\{ e \pi \}) \rtimes_\alg N_\pi(h)$ and $\bC[\pi]$ with $c(\{ e \pi \}) \rtimes_\alg N_\Gamma(h)$. 
By taking the sum over all $\langle h \rangle \in \langle \pi ; g \rangle$, we get
\begin{align}
\begin{split}
\xymatrix@C=7ex{
\bigoplus_{\langle h \rangle} 
\HP_*(\bC[N_\pi(h)]; \langle h \rangle)  \ar[r]^{\bigoplus_{\langle h \rangle} (j_1)_* \ \ \ \ \ \ \ } \ar[d]^{\bigoplus_{\langle h \rangle}(j_2)_* } &  
\bigoplus_{\langle h \rangle} 
\HP_*(c_c(X_h) \rtimes_\alg N_\Gamma (g); \langle h \rangle) \ar[d]^{ \sum_{\langle h \rangle} (j_3)_* } \\
\bigoplus_{\langle h \rangle} \HP_*(\bC[\pi]; \langle h \rangle)  \ar[r]^{ \sum_{\langle h \rangle} (j_4)_* \ \ \ } &  \HP_*(c_c(\Gamma /\pi) \rtimes _\alg \Gamma ; \langle g \rangle).
}
\end{split}\label{eq:cyclic_diagram}
\end{align}
Here we remove the homomorphism $\Ad (x^{-1})$ since it induces the identity on the cyclic homology groups.

By \cite{nistorGroupCohomologyCyclic1990}*{Lemma 2.7}, each $(j_2)_*$ is an isomorphism for any $\langle h \rangle \in \langle \pi;g \rangle$. 
Moreover, each $(j_1)_*$ is also an isomorphism since the inclusion $j_1$ implements the Morita equivalence. By the same reason, $\sum_{\langle h \rangle } (j_4)_*$ is also isomorphic. Indeed, the direct sum 
\[\bigoplus_{\langle g \rangle \in \langle \Gamma \rangle} \sum_{\langle h \rangle \in \langle \pi ; g \rangle } (j_4)_* \colon \bigoplus _{\langle h \rangle \in \langle \pi \rangle } \HP_*(\bC[\pi]; \langle h \rangle) \to \bigoplus_{\langle g \rangle \in  \langle \Gamma \rangle }\HP_*(c_c(\Gamma /\pi) \rtimes_\alg \Gamma ; \langle g \rangle ) \]
coincides with the map $\HP_*(\bC [\pi]) \to \HP_*(c_c(\Gamma /\pi) \rtimes_\alg \Gamma)$, and hence is isomorphic. 
This shows that $\sum_{\langle h \rangle }(j_3)_*$ is also isomorphic, and in particular, $\HP_*(c_c(Y) \rtimes_\alg \Gamma ; \langle g \rangle )=0$.

\begin{lem}
The map $\tau_\sigma$ respects the decomposition \eqref{eq:decomposition}, i.e., $\tau_\sigma$ is the direct sum of homomorphisms
\[ \mathclap{\tau_\sigma^{\langle g \rangle}  \colon \HP_*(\bC[\Gamma] ; \langle g \rangle ) \to \HP_{*-2}(c_c(\Gamma /\pi) \rtimes_{\alg}\Gamma ; \langle g \rangle ) \cong \bigoplus_{\langle h \rangle \in \langle \pi ;g \rangle } \HP_{*-2}(\bC [\pi] ; \langle h \rangle ).} \]
\end{lem}
\begin{proof}
In general, if one has an extension $0 \to I \to A \to A/I \to 0$ of $\Gamma$-algebras, then the boundary map $\partial \colon \HP_n(A/I \rtimes \Gamma) \to \HP_{n-1}(I \rtimes \Gamma)$ respects the decomposition \eqref{eq:decomposition} since the decomposition is given in chain level. 
This proves the lemma since the twisted action $C^\infty(\cB) \curvearrowleft  \Gamma $ is Morita equivalent to a genuine action onto $C^\infty (\cB , \cK_\sigma)$ where $\cK_\sigma$ is a $\Gamma$-equivariant bundle of compact operator algebras (see \cref{rmk:PR}). 
\end{proof}

In \cite{burgheleaCyclicHomologyGroup1985}, each group $\HP_*(\bC[\Gamma ] ; \langle g \rangle)$ is related to the group homology. Indeed, there is an isomorphism 
\[ \Phi _\Gamma \colon  H_{[*]}(\Gamma ; \bC ) \to \HP_*(\bC[\Gamma] ; e), \]
where $H_{[*]}({\cdot })$ denotes the product of homology group $\prod_{k \in \bZ} H_{* + 2k} ({\cdot}) $. 
We compare the codimension 2 transfer maps in group homology, given in \cref{rmk:cocycle_transfer}, and cyclic homology groups. 
\begin{lem}\label{lem:cyc_unit}
The diagram
\[
\xymatrix{
H_{[*]}(\Gamma  ; \bC) \ar[r]^{2 \pi i \sigma \cdot } \ar[d]^{\Phi_\Gamma } &  H_{[*]}(\pi  ; \bC) \ar[d]^{\Phi_\pi } \\ 
\HP_*(\bC [\Gamma ] ; e ) \ar[r]^{\tau_\sigma } &  \HP_{*}(\bC[\pi] ;  e ).
}
\]
commutes.
\end{lem}
\begin{proof}
First, we study the map $\tau_\sigma$ in twisted equivariant (co)homology, instead of cyclic homology.  
Let $(\Omega^\bullet (\cB), d)$ denote the de Rham complex on $\cB$, defined precisely as $\Omega^\bullet  (\cB) = \bC \cdot 1 \oplus \Omega_0(\cB) $, where $(\Omega_0(\cB_0), d)$ is the de Rham complex on the manifold $\cB_0$ vanishing at infinity and $d (1)=0$. 
We regard $\Omega^\bullet (\cB)$ as a complex of $\Gamma$-modules and define the chian and cochain double complexes $C_\bullet (\Gamma , \Omega^\bullet (\cB)) $ and  $C^\bullet (\Gamma, \Omega^\bullet (\cB))$ with coefficient in $\Omega^\bullet (\cB)$. 
The cohomology of the total complex of $C^\bullet (\Gamma, \Omega^\bullet (\cB))$ is isomorphic to the equivariant cohomology $H^*_\Gamma (\cB;\bC) $.

Let us define the $(2,1)$-cocycle $\Theta \in C^2(\Gamma ; \Omega^1(\cB))$ as
\[ \Theta (g,h) :=  \sigma(g,h)^{-1} d \sigma(g,h)  \in\Omega ^1(\cB). \]
Since $\int z^{-1}dz = 2\pi i$, this $\Theta$ represents the cohomology class $2 \pi i \sigma \in H^3_\Gamma (\cB; \bC)$ (cf. \cref{prp:Dixmier}). 

The twisted de Rham differential $\delta_\Theta := \delta_\Gamma + vd_{\mathrm{dR}} + \Theta$ on the double complex $C_\bullet (\Gamma, \Omega^\bullet (\cB))[\![v^{\pm 1}]\!]$, where $v$ is a formal symbol of degree $-2$, gives rise to an exact sequence
\begin{align}
    0 \to C_\bullet (\Gamma , \Omega^\bullet _0(\cB_0 ) )[\![v^{\pm 1}]\!] &\to C_\bullet (\Gamma , \Omega^\bullet (\cB ) )[\![v^{\pm 1}]\!] 
    \to C_\bullet (\Gamma , \bC)[\![v^{\pm 1}]\!]  \to 0. \label{eq:exact_diff}
\end{align}
Since the differential $\delta_\Theta$ is of the form $ \big( \begin{smallmatrix} \delta_\Gamma  & 0 \\ \Theta  & \delta _\Gamma + d_{\mathrm{dR}} \end{smallmatrix}\big) $ under the direct sum decomposition $C_\bullet (\Gamma  , \Omega ^\bullet ( \cB)) \cong C_\bullet (\Gamma , \bC) \oplus C_\bullet (\Gamma, \Omega_0^\bullet (\cB_0))$, 
the boundary map of the associated long exact sequence is the cap product with $\Theta$.
By \cref{lem:untwist}, the group $H_*(C_\bullet (\Gamma , \Omega^\bullet _0(\cB_0 ) )[\![v^{\pm 1}]\!] , \delta_\Theta)$ is isomorphic to the equivariant homology $H_{[*-1]}^\Gamma (\Gamma /\pi ; \bC)$ with degree shift by $1$. Under this isomorphism, the cap product with $[\Theta]$ is identified with the cap product in equivariant homology with $2\pi i \sigma \in H^2_\Gamma( \Gamma /\pi ; \bC)$.

The remaining task is to relate the exact sequence of group homologies in \eqref{eq:exact_diff} with that of cyclic chain complexes.  
This is essentially due to the work of Angel~\cite{angelCyclicCocyclesTwisted2013}, but a we need slightly modified version of it. 
The detailed discussion of this part is given in \cref{section:Chern}. 
\end{proof}
 
We also describe the delocalized part of the cyclic homology in terms of group homology. 
For $h \in \langle \pi \rangle$, let 
\[ \sigma^{\langle h \rangle }:=\psi_{\langle h\rangle} (\iota^* (\sigma)) \in H^2(N_\Gamma (h); \bZ[X_h]), \]
where $\iota \colon N_\Gamma (h) \to \Gamma $ denote the inclusion and $\psi_{\langle h\rangle} \colon \bZ[\Gamma /\pi] \to \bZ[X_h]$ denote the quotient. 
It corresponds to an extension
\[0 \to \bZ[X_h] \to \bZ[X_h] \rtimes_\sigma N_\Gamma (h) \to N_\Gamma (h) \to 0.  \]
Set $N_{\Gamma, g}:=N_\Gamma (g) /\langle g \rangle$ and $N_{\pi,h}:=N_\pi(h)/\langle h \rangle $. If $g \in \Gamma$ is a torsion element, then the homomorphism
\begin{align*} 
 \HC_*(c_c(X_h) \rtimes_\alg N_\Gamma (h) ; \langle h \rangle ) \to \HC_*(c_c(X_h) \rtimes N_{\Gamma , h} ; e)  
\end{align*}
is isomorphic since $h$ acts on $X_h$ trivially \cite{nistorGroupCohomologyCyclic1990}*{2.7}.

\begin{prp}\label{prp:cocycle}
If $g \in \Gamma$ is a torsion element, the diagram
\[
\xymatrix{
H_{[*]}(N_{\Gamma , g} ; \bC) \ar[r]^{\sum \sigma^{\langle h\rangle} _* \hspace{5ex}} \ar[d] & \bigoplus _{\langle h \rangle  \in \langle \pi ; g \rangle }  H_{[*]}(N_{\pi, h} ; \bC) \ar[d] \\ 
\HP_*(\bC [\Gamma ] ; \langle g \rangle ) \ar[r]^{\tau_\sigma^{\langle h \rangle } \hspace{4ex}} & \bigoplus_{\langle h \rangle  \in \langle \pi ; g \rangle} \HP_*(\bC[\pi] ; \langle h \rangle)
}
\]
commutes. 
\end{prp}
\begin{proof}
By the above arguments, the diagram
\[\mathclap{
\xymatrix@C=1em{
\HC_*(\bC[N_{\Gamma,g}] ; e) \ar[r]^{\partial \hspace{4ex}} & 
\bigoplus_{\langle h \rangle } \HC_*(\bC[X_h] \rtimes N_{\Gamma , h}; e)   & 
\bigoplus_{\langle h \rangle } \HC_*(\bC[N_{\pi, h}] ; e) \ar[l]_\cong  \\
\HC_*(\bC[N_{\Gamma}(g)] ; \langle g \rangle) \ar[r]^{\partial \hspace{3.5ex}} \ar[u]_\cong \ar[d]^\cong &
\bigoplus_{\langle h \rangle } \HC_*(c_c(X_h) \rtimes N_\Gamma(h))  \ar[u]_\cong \ar[d]^\cong & 
\bigoplus_{\langle h \rangle } \HC_*(\bC[N_\pi(h)] ; \langle h \rangle ) \ar[d]^\cong \ar[u]_\cong \ar[l]_\cong \\
\HC_*(\bC[\Gamma] ; \langle g \rangle ) \ar[r]^{\partial \hspace{3ex}} & 
\HC_*(c_c(X) \rtimes \Gamma ; \langle g \rangle )& 
\bigoplus_{\langle h \rangle } \HC_*(\bC[\pi] ; \langle h \rangle ) \ar[l]_\cong 
}}
\]
commutes. After taking the projective limit by Connes' $S$-map, the upper horizontal maps is identified with $2 \pi i \sigma^{\langle h \rangle}$ by \cref{lem:cyc_unit}. This finishes the proof. 
\end{proof}
\begin{rmk}
As is shown in \cite{burgheleaCyclicHomologyGroup1985}, if $g $ is an infinite order element, then 
\[ \HP_*(\bC[\Gamma]; \langle g\rangle) \cong T_*^\Gamma (g; \bC ):= \varprojlim (H_{[*]}(N_{\Gamma , g} ; \bC), S),\]
where $S $ is the Gysin map of the fibration $S^1 \to BN_{\Gamma }(g) \to BN_{\Gamma , g}$. 
The same proof as \cref{prp:cocycle} identifies $\tau_\sigma^{\langle g \rangle}$ with the sum of the cap product with $\sigma^{\langle h \rangle}$ defined on $T_*^\Gamma (g;\bC)$.
\end{rmk}

The proof of \cref{prp:cocycle} also shows the dual statement in cohomology theories. 
Note that the cyclic cohomology group $\HC^*(\bC[\Gamma] )$ decomposes into the direct product of groups $\HC^*(\bC[\Gamma]; \langle g \rangle)$.  
Let $H^{[\ast]}(\cdot )$ stands for the direct sum $\bigoplus _{k\in \bZ}H^{* + 2k}({\cdot})$.
\begin{cor}\label{cor:cocycle_dual}
If $g \in \Gamma$ is a torsion element, the diagram
\[
\xymatrix@C=4em{
H^{[*]}(N_{\Gamma , g} ; \bC) \ar[r]^{\sum 2\pi i \sigma^{\langle h\rangle}  \hspace{5ex} } \ar[d] & \bigoplus _{\langle h \rangle  \in \langle \pi ; g \rangle }  H^{[*]}(N_{\pi, h} ; \bC) \ar[d] \\ 
\HP^*(\bC [\Gamma ] ; \langle g \rangle ) \ar[r]^{\tau_\sigma^{\langle h \rangle } \hspace{4ex}} & \bigoplus_{\langle h \rangle  \in \langle \pi ; g \rangle} \HP^*(\bC[\pi] ; \langle h \rangle)
}
\]
commutes. 
\end{cor}

\subsection{Mapping codimension $2$ transfer of analytic surgery to that of homology}
Next we relate the codimension $2$ transfer map with the `mapping analytic surgery to homology' formalism developed by \cite{piazzaMappingAnalyticSurgery2019}. For this sake, we firstly rebuild the codimension $2$ transfer maps in the unconditional crossed product Banach algebras. 
In contrast to \cref{section:Cstar,section:HigsonRoe}, hereafter $C^*(\widetilde{M})^\Gamma$ stands for the C*-algebra completion of $\bC[\widetilde{M}]^\Gamma$ with respect to the reduced norm, i.e., the operator norm on $L^2(\widetilde{M})$. Other coarse C*-algebras $D^*(\widetilde{M})^\Gamma$ and $Q^*(\widetilde{M})^\Gamma $ are also defined in a consistent way. 

Let $\| \cdot \|_{A\Gamma} $ be a submultiplicative norm on the group algebra $\bC[\Gamma]$. 
Assume this is unconditional, i.e., if $\sum a_g u_g, \sum b_g u_g \in \bC[\Gamma]$ satisfies $|a_g| \leq |b_g|$ for any $g \in \Gamma$, then $\| \sum a_g u_g \|_{A\Gamma } \leq \| \sum b_g u_g\|_{A\Gamma}$ holds. 
We write $A\Gamma $ for the completion of $\bC[\Gamma] $ with respect to this norm. We also assume that $A\Gamma$ is closed under holomorphic functional calculus. 
Let $\| \cdot \|_{A\pi}$  the restriction of $\| \cdot \| _{A\Gamma}$ to $\bC[\pi]$, which is also a submultiplicative unconditional norm on $\bC[\pi]$ and the completion $A\pi := \overline{\bC[\pi]}$ is closed under holomorphic functional calculus. 

For a $\Gamma$-C*-algebra $B$, we define the norm $\| \cdot \|_{A(B,\Gamma )}$ on the algebraic crossed product $B \rtimes_{\alg} \Gamma $ as
\begin{align}
    \Big\| \sum_g f_g u_g \Big\|_{A(B, \Gamma )} := \Big\| \sum \| f_g \| u_g \Big\|_{A\Gamma }. \label{eq:Banach}
\end{align}  
We write $A(B,\Gamma)$ for the completion of $B \rtimes _\alg \Gamma $ with respect to this norm. 
\begin{lem}\label{lem:ext_Ban}
Let $0 \to I \to B \xrightarrow{q} B/I \to 0$ be an extension of $\Gamma$-C*-algebras such that there is a (possibly non-equivariant) contractive section $\nu \colon B/I \to B$. Then
\[ 0 \to A(I,\Gamma) \to A(B, \Gamma ) \to A(B/I, \Gamma) \to 0 \]
is an exact sequence of Banach algebras. 
\end{lem}
\begin{proof}
The exactness at the left and the right are obvious. Hence it suffices to show that the map $\tilde{q} \colon A(B,\Gamma) / A(I,\Gamma)  \to A(B/I, \Gamma)$ induced from $q$ is an isomorphism. For any $x \in A(B, \Gamma )$ and an arbitrary $\varepsilon >0$, we take a finite sum $x':=\sum_{g } b_g u_g \in B \rtimes_{\mathrm{alg}} \Gamma $ such that $\| x - x' \|_{A(B,\Gamma)} < \varepsilon$. 
Let $y':= \sum_g (\nu \circ q)(b_g) \cdot u_g \in B \rtimes_\alg \Gamma$. 
Then 
\[ x' - y' = \sum_g (b_g - (\nu \circ q)(b_g)) \cdot u_g \in I \rtimes_\alg \Gamma \]
and hence $\tilde{q} (x') = \tilde{q} (y')$. Therefore, we obtain 
\begin{align*} 
 \| \tilde{q} (x) \| \leq  \| \tilde{q} (x')\| + \varepsilon =& \inf_{\tilde{q} (x') = \phi (y)} \| y \|_{A(B,\Gamma)} + \varepsilon \leq  \| y' \|_{A(B, \Gamma )} + \varepsilon  \\
=& \Big\| \sum_g \| (\nu \circ q)(b_g) \|_{B} u_g \Big\|_{A\Gamma } + \varepsilon \\
 \leq & \Big\| \sum_g \| q(b_g) \|_{B/I} u_g \Big\|_{A \Gamma } + \varepsilon = \| q(x') \|_{A(B/I, \Gamma)} + \varepsilon . 
\end{align*}
Here, for the last inequality we use the unconditionality condition of the norm $\| \cdot \|_{A\Gamma}$. This shows the lemma. 
\end{proof}

Let $\cK_\sigma$ be a $\Gamma$-equivariant compact operator algebra bundle over $\cB$ whose Dixmier--Douady class is $\sigma \in H^2_\Gamma (\cB; \bT)$ (a construction is given in \cref{rmk:PR}). 
Let $C(\cB , \cK_\sigma )$ denote the C*-algebra of continuous section of $\cK_\sigma$. Then the crossed product $C(\cB, \cK_\sigma) \rtimes \Gamma $ is Morita equivalent to $C(\cB) \rtimes_\sigma \Gamma$. 
Apply \cref{lem:ext_Ban} to the extension $0 \to C_0 (\cB_0,\cK_\sigma ) \to C(\cB,\cK_\sigma ) \to \bC \to 0$, we obtain an exact sequence
\[ 0 \to A (C_0 (\cB_0 , \cK_\sigma ), \Gamma) \to A (C(\cB, \cK_\sigma ), \Gamma) \to A (\bK, \Gamma)  \to 0  \]
of Banach algebras. This induces the boundary map
\[ \tau _\sigma := \partial  \colon \HP_*(A\Gamma ) \to \HP_{*-1}( A (C_0 (\cB_0 , \cK_\sigma ), \Gamma)) \cong \HP_{*-2}(A \pi ), \]
which is compatible with \cref{def:codim2_cyclic}. The last isomorphism follows from our assumption on the unconditional norm $\| \cdot \|_{A\pi}$, which implies that $S A\pi $ is identified with a full corner subalgebra of $A(C_0 (\cB_0, \cK) , \Gamma) = A(S c_c(\Gamma /\pi) \otimes \bK, \Gamma )$. 

The following lemma shows that the codimension $2$ transfer maps constructed in Banach algebraic and C*-algebraic settings are compatible.
\begin{lem}\label{lem:Banach_Cstar}
Assume that $\Gamma$ is exact (cf.~\cref{rmk:intermediate}). 
Then there is an inclusion of exact sequences 
\[\xymatrix{
0 \ar[r] & A(C_0(\cB_0 , \cK_\sigma ) , \Gamma  ) \ar[r] \ar[d]  & A(C(\cB , \cK_\sigma ) , \Gamma  ) \ar[r] \ar[d]  & A(\bK, \Gamma) \ar[r] \ar[d] & 0 \\
0 \ar[r] & C_0(\cB_0 , \cK_\sigma ) \rtimes_{\mathrm{red}} \Gamma   \ar[r]   & C(\cB, \cK_\sigma ) \rtimes_{\mathrm{red}} \Gamma   \ar[r]   & \bK \rtimes_{\mathrm{red}} \Gamma  \ar[r] & 0.
}\]
\end{lem}
\begin{proof}
By assumption of the norm $\| \cdot \|_{A\Gamma}$, the right vertical map is contractive. 
The left vertical map is also contractive since the algebras $A(C_0(\cB_0, \cK_\sigma) , \Gamma)$ and $C_0(\cB_0 , \cK_\sigma) \rtimes_{\mathrm{red}}\Gamma$ are Morita equivalent to $SA\pi$ and $SC^*_{\mathrm{red}}\pi$ respectively, and moreover $A \pi$ is a subalgebra of $C^*_{\mathrm{red}}\pi $ (note that the closure of $\bC[\pi]$ in $C^*_\mathrm{red}\Gamma $ coincides with $C^*_{\mathrm{red}}\pi$). 

To see that the middle map is bounded, let $x = \sum _g b_gu_g  \in C(\cB, \cK_\sigma) \rtimes_\alg \Gamma $. For simplicity, we write the unconditional norms as $\| \cdot \|_A$ and reduced crossed product norms as $\| \cdot \|_{C^*}$. 
In the same way as \cref{lem:ext_Ban}, let $y := \sum (\nu \circ q)(b_g) u_g$. 
Then the assumption of unconditionality on $A\Gamma$ shows that
\begin{itemize} 
    \item $\| y \|_{A} \leq \| x \| _{A}$, 
    \item $\| x-y\|_{C^*} \leq \| x - y\|_A \leq 2 \| x\|_A $, 
    \item $ \| y \|_A \leq \| q(y)\|_A = \| q(x)\|_A \leq \| x \|_A$,
\end{itemize}
which implies that $\| x\|_{C^* } \leq 3\| x\|_A$. This shows that the middle vertical map is bounded on a dense subalgebra of $A(C(\cB, \cK_\sigma) , \Gamma)$, and hence extends to a bounded algebra homomorphism. 
\end{proof}

The decomposition of periodic cyclic homology group $\HP(\bC [\Gamma])$ into the localized and delocalized parts extends to a long exact sequence
\[ \cdots \to \HP_*^e(A\Gamma ) \to \HP_*(A\Gamma ) \to \HP_*^\del(A\Gamma ) \xrightarrow{\delta_\Gamma} \HP_*(A\Gamma) \to \cdots . \]
In \cite{piazzaMappingAnalyticSurgery2019}*{Definition 6.36}, Piazza--Schick--Zenobi has constructed the Chern character maps taking value in the non-commutative de Rham homology groups
\begin{align}
\begin{split}
    \mathrm{Ch}_\Gamma  & \colon \K_{*}(C^*(\widetilde{M})^\Gamma ) \to \overline{H}_{[*]}(A \Gamma),\\
    \mathrm{Ch}_\Gamma^{\del}  & \colon \K_*( D^*(\widetilde{M})^\Gamma )  \to \overline{H}{}_{[*]}^{\del}(A \Gamma), \\
    \mathrm{Ch}_\Gamma^{e}  & \colon \K_*( Q^*(\widetilde{M})^\Gamma ) \to \overline{H}{}_{[*]}^{e}(A \Gamma), 
    \end{split}\label{eq:Chern}
\end{align}
by using the pseudodifferential operator algebras (whose K-theory is identified with those of coarse C*-algebras). 
Here, $\overline{H}{}_{[*]}(A\Gamma )$ denotes the periodic reduced noncommutative de Rham homology group \cite{karoubiHomologieCycliqueTheorie1987}, which embeds into the reduced periodic cyclic homology  $\overline{\HP}_*(A\Gamma ) := \mathop{\mathrm{coker}} (\HP_*(\bC) \to \HP_*(A\Gamma ))$. 
\begin{rmk}\label{rmk:Chern}
In \cite{piazzaMappingAnalyticSurgery2019}, the map $\Ch_\Gamma$ is defined to be the composition of the Chern character of noncommutative de Rham homology and Lott's trace map \cite{lottSuperconnectionsHigherIndex1992}. 
This map coincides with the Chern--Connes character in cyclic homology theory. 
Indeed, this is checked by comparing \cite{karoubiHomologieCycliqueTheorie1987}*{1.17, Example 1.15} and \cite{lodayCyclicHomology1998}*{Theorem 8.3.2, Corollary 8.3.5}, noting that the inclusion $\overline{H}_*(A \Gamma ) \subset \overline{\HC}_*(A \Gamma )$ is induced from the isomorphism $\overline{\mathrm{C}}{}_n^\lambda (A \Gamma )/ \Im  b \cong \overline{\Omega}_n(A\Gamma )_{\mathrm{ab}} / \Im d$ (see \cite{lodayCyclicHomology1998}*{Theorem 2.6.8}). 
 
\end{rmk}

In this paper we deal with a minor modification of \eqref{eq:Chern}. 
Let $V$ be a vector bundle over $M$ and let $X$ denote the sphere bundle $\bS (V \oplus \underline{\bR})$, which includes $M$ as a submanifold. We use the letter $\cV$ for the tubular neighborhood of $M$ in $X$, on which a Riemannian metric is imposed in the way that the projection $\cV \to M$ is a coarse equivalence. 
Let $\widetilde{X}$ denote the $\Gamma$-Galois covering of $X$.   
In \cite{piazzaMappingAnalyticSurgery2019}*{Definition 5.14}, the Banach algebra completions $\Psi^0_{A\Gamma }(\widetilde{X})$ and $\Psi^{-\infty}_{A\Gamma }(\widetilde{X})$ of the algebras of $\Gamma$-invariant pseudodifferential operators of order $0$ and $-\infty$ are defined. 
We write $\Psi^{k}_{A\Gamma }(\widetilde{\cV})$ ($k=0,-\infty$) for the closed subalgebras of $\Psi^{k}_{A\Gamma }(\widetilde{X})$ consisting of operators supported on $\widetilde{\cV}$. Then $\Psi^{-\infty }_{A\Gamma }(\widetilde{\cV})$ is a dense subalgebra of $C^*(\widetilde{\cV})^\Gamma \cong C^*_r \Gamma \otimes \bK$ and the quotient $\Psi^0_{A\Gamma }(\widetilde{\cV}) / \Psi^{-\infty}_{A\Gamma }(\widetilde{\cV})$ is isomorphic to the algebra of principal symbols, i.e., the continuous function algebra $C_0(\bS \cV)$ on the sphere bundle of $T\cV$.
Moreover, $\Psi^0_{A\Gamma }(\widetilde{\cV}) $ is a subalgebra of $D^*(\widetilde{\cV})^\Gamma$. 
As is shown in \cite{zenobiAdiabaticGroupoidHigsonRoe2019}*{Theorem 3.5}, there is an isomorphism of K-groups of mapping cones
\[ \K_*\big( C( C_0(\cV) \to \Psi_{A\Gamma}^0(\widetilde{\cV}))\big) \cong \K_*\big( C(0 \to D^*(\widetilde{\cV})^\Gamma )\big)  \]
such that the composition 
\[ \K_*\big( C(0 \to \Psi^0_{A\Gamma}(\widetilde{\cV}) \big) \to  \K_* \big( C(C_0(\cV) \to \Psi_{A\Gamma}^0(\widetilde{\cV})) \big) \cong \K_* \big( C(0 \to D^*(\widetilde{\cV})^\Gamma ) \big)\]
is the same map as the one induced from the inclusion $\Psi^0_{A\Gamma}(\widetilde{\cV}) \subset D^*(\widetilde{\cV})^\Gamma $. 

Note that a five lemma argument and the Thom isomorphism $\K_*(\cV) \cong \K_*(M)$ shows that $D^*(\widetilde{M})^\Gamma$ and $D^*(\widetilde{\cV})^\Gamma$ have the same K-theory.  
We write $\Ch_{\Gamma, V}$, $\Ch_{\Gamma, V}^e$ and $\Ch_{\Gamma, V}^\del $ for the compositions
\begin{align*}
    \Ch_{\Gamma, V}\colon &\K_*(C^*(\widetilde{M})^\Gamma ) \cong 
    \K_{*-1}(C(0 \to C^*(\widetilde{\cV})^\Gamma ) ) \xrightarrow{\Ch_\Gamma }\overline{\HP}{}_*(A \Gamma),  \\
    \Ch_{\Gamma, V}^\del \colon &\K_*(D^*(\widetilde{M})^\Gamma ) \cong 
    \K_{*-1}(C(C_0({\cV}) \to \Psi^0(\widetilde{\cV})^\Gamma )) \xrightarrow{\Ch_\Gamma ^\del  }\overline{\HP}{}_*(A \Gamma),  \\
    \Ch_{\Gamma, V}^e \colon &\K_*(Q^*(\widetilde{M})^\Gamma ) \cong 
    \K_*(C(C_0({\cV}) \to C_0(\bS {\cV})) ) \xrightarrow{\Ch_\Gamma ^e }\overline{\HP}{}_*(A \Gamma),  
\end{align*} 
where the last maps $\Ch_\Gamma$, $\Ch_\Gamma^e$, $\Ch_\Gamma^\del $ are the Chern character maps \eqref{eq:Chern} of Piazza--Schick--Zenobi (more precisely, the composition of Chern character maps on $\widetilde{X}$ and the inclusions induced from $\Psi^k_{A\Gamma } (\widetilde{\cV})\to \Psi^k_{A\Gamma} (\widetilde{X})$). 
Here we remark that one gets the same map if they uses another open embedding $\cV \subset Y$ inducing isomorphism of $\pi_1$ instead of $\cV \subset X$. It directly follows from the definition of Chern character maps. 
By \cite{piazzaMappingAnalyticSurgery2019}*{Theorem 6.38}, the diagram
\[ \mathclap{
\xymatrix{
\cdots \ar[r] & \K(Q^*(\widetilde{M})^\Gamma ) \ar[r] \ar[d]^{\Ch_{\Gamma, V}^e} & \K_*(C^*(\widetilde{M})^\Gamma ) \ar[r] \ar[d]^{\Ch_{\Gamma, V}} & \K_*(D^*(\widetilde{M}) ^\Gamma ) \ar[d]^{\Ch_{\Gamma, V}^\del } \ar[r] & \K_*(Q^*(\widetilde{M})^\Gamma ) \ar[d]^{\Ch_{\Gamma, V}^e} \ar[r] & \cdots  \\
\cdots \ar[r] & \overline{\HP}{}_*^e(A \Gamma ) \ar[r] & \overline{\HP}{}_*(A \Gamma ) \ar[r] & \overline{\HP}{}_*^\del (A \Gamma) \ar[r]^{\delta_\Gamma } & \overline{\HP}{}^e_*(A\Gamma ) \ar[r] &  \cdots 
}}
\]
commutes.

\begin{lem}\label{lem:Khomology2}
Let $W$ be a compact $\mathrm{spin}^c$ manifold with boundary $\partial W$. Assume that the Euler number of $\partial W$ is zero. 
Let $f \colon (W, \partial W) \to (B\Gamma, \cV)$ be a continuous map of pairs such that $f|_{\partial W}$ is a smooth embedding with trivial normal bundle. 
Let $E$ is a vector bundle over $W$. 
Then there is 
$x_{W,f,E} \in \K_*(\Psi^0_{A\Gamma}(\widetilde{\cV}))$ with the following properties; 
\begin{enumerate}
    \item the class $x_{W,f,E}$ is sent to $\Phi_{B\Gamma, M}([W,f,E]) \in \K_*(D^*(\widetilde{M})^\Gamma )$, 
    \item $\Ch^\del_{\Gamma, V} (x_{W,f,E}) =0$. 
\end{enumerate}
\end{lem}
\begin{proof}
Firstly we remark that any pseudodifferential operator acting on sections of a vector bundle $\mathsf{E}$ over a manifold $\mathsf{M}$ is viewed as an element of $\bM_N (\Psi^0_{A\Gamma }(\widetilde{\mathsf{M}}))$ for some $N$ by identifying $L^2(\widetilde{\mathsf{M}} , \mathsf{E}) $ with $P L^2(\widetilde{\mathsf{M}}, \bC^N)$, where $P \in C(\mathsf{M}, \bM_N)$ is the support projection of $\mathsf{E}$ embedded into a trivial bundle $\underline{\bC}^N$.

Here we focus on the case that $\dim W $ is even, and just remark that the same proof also works for odd-dimensional case by considering Dirac operators with $\bC \ell_1$-symmetry. 
Let $D_{\partial W}$ and $D_W$ denote the $\mathrm{spin}^c$-Dirac operator on $\partial \widetilde{W} $ and $\widetilde{W}$ twisted by $E$ respectively. 
Let $A_W$ be a smoothing operator as \cref{lem:double}, i.e., we have $(D_{\partial W} + A_W)^2 \geq \varepsilon ^2$ and $\ind_{\mathrm{APS}} (D_W, A_W) =0$. 
By \eqref{eq:geom_rel}, we have 
\[ [\chi(D_{\partial W} +A_W)] = \Phi_{B\Gamma, M}([W,f,E]) \in \K_0(D^*(\partial \widetilde{W})^\Gamma ).\] 
Let $\bD_r^m$ denote the $m$-dimensional open disk of radius $r>0$. We define operators $C$ and $D$ on $L^2(\bD_r^m , \Delta_{m,m} )$, where $\Delta_{m,m}$ is the unique irreducible representation of the Clifford algebra $\Cl_{m,m}$, as
\[D:= \sum c(f_i) \partial_{x_i}, \ \ \ \  C := \frac{1}{r}\sum c(e_i)x_i, \]
where $c(e_i)$ and $c(f_i)$ are self-adjoint generators of $\Cl_{m,m}$ with $c(e_i)^2=1$ and $c(f_i)^2=-1$.  
Then $C$ is the Clifford multiplication with the radial vector field on $\bD_r^m$, and in particular $C^2-1 \in C_0(\bD_r^m)$. Set $S_{\partial W, m}^E:= E \otimes S_{\partial W} \hotimes \Delta_{m,m}$. By the assumption on the vanishing of the Euler number of $\partial W$, there is a non-vanishing vector field $\xi$ on $\partial W \times \bD_r^m$ which coincides with the radial vector field outside a compact subset, and hence its Clifford multiplication acting on the bundle $S_{\partial W , m}^E$ satisfies 
\[ C- c(\xi) \in C_0(W \times \bD_r^m , \End (S_{\partial W,m}^E)).\]
Now, the bounded self-adjoint operator 
\begin{align*}
     B:= & (1-C^2)^{1/4}\chi(D_{\partial W} +A_W + D )(1 - C^2)^{1/4} +  C
\end{align*}
on $\partial W \times \bD_r^m$ is a $0$-th order pseudo-differential operator whose principal symbol satisfies $\sigma (B) - c(\xi) \in C_0(\bS V , \End (S_{\partial W,m}^E))$. 
Moreover, it is invertible if $r$ is sufficiently large so that
\[ \| [(1-C^2)^{1/4}, \chi(D_{\partial W} + A_W +  D)] \| < \varepsilon /4 .\] 

Without loss of generality, we may assume that the Riemannian metric on $r$-neighborhood of $\partial W$ is the product metric $g_{\partial W} + g_{\bD_r^n}$. 
We define $x_{W,f,E}$ as the difference class 
\[ x_{W,f,E} := \iota_* [(B, c(\xi))] \in \K_0(\Psi^0_{A\Gamma} (\widetilde{\cV})), \]
where $\iota \colon \Psi^0_{A\Gamma } (\partial \widetilde{W} \times \bD_r^m) \to \Psi^0_{A\Gamma}(\widetilde{\cV})$ is induced from the isometric open embedding.  
 The inclusion $\Psi_{A\Gamma}^0({\cV}) \subset D^*({\cV})^\Gamma$ sends this $x_{W,f,E}$ to 
\[ [\chi (D_{\partial W} +A_W)] \boxtimes \beta \in  \K_*(D^*(\widetilde{\cV})^\Gamma ), \]
where $\boxtimes $ denotes the secondary exterior product in the sense of \cref{section:product}. 
Hence it is identified with $\Phi_{B\Gamma, M}([W,f,E])$ under the isomorphism $\K_*(D^*(\widetilde{M})^\Gamma) \cong \K_*(D^*(\widetilde{\cV})^\Gamma) $. 

Finally we show that $\Ch_{\Gamma, V}^\del (x_{W,f,E}) =0$. 
This follows from \cite{piazzaMappingAnalyticSurgery2019}*{Proposition 6.71}, in which the delocalized $\rho$-invariant is identified with the higher APS index. 
Set $\cW:= W \times S^m$. 
Note that both $B$ and $c(\xi)$ extends to a pseudo-differential operator on $\partial \cW$ twisted by $S_{\partial W} \otimes E \hotimes \cE$, where $\cE$ is the $\bZ_2$-graded vector bundle on $S^n$ obtained by clutching trivial bundles $\bD_r^m \times \Delta_{m,m}$ and $\bD_r^m \times (\Delta_{m,m}^{0})^{\oplus 2}$ by $\id \oplus C|_{\partial \bD_r^m}$. 
Therefore, the image of $x_{W,f,E}$ in the K-group of $\Psi_{A\Gamma}^0(\widetilde{\cW})$ is $[B] - [c(\xi)]$. 

Since $\Ch^\del_\Gamma$ factors through the mapping cone $C(C(\cW ) \to \Psi_{A\Gamma}^0(\widetilde{\cW}))$, we have $\Ch_{\Gamma}^\del ([c(\xi)]) =0$. 
Moreover, since $B$ is homotopic to $\chi(D_{\partial W} + A_W + D_{S^m}^\cE ) $ in $\Psi_{A\Gamma}^0(\widetilde{\cW}) $, we have
\begin{align*}
    \Ch_{\Gamma, V}^\del (x_{W,f,E}) =& \Ch_{\Gamma}^\del ([B])  = \Ch_{\Gamma}^\del ([\chi(D_{\partial W} + A_W + D_{S^m}^\cE)]  )\\
    =& \ind_{\mathrm{APS}} (D_W + D_{S^m}^\cE, A_W \hotimes 1 ) \\
    =& \ind_{\mathrm{APS}} (D_W, A_W) \cdot \ind (D_{S^m}^\cE) =0. 
\end{align*}  
Here the third equality is \cite{piazzaMappingAnalyticSurgery2019}*{Proposition 6.71} and the forth equality is the multiplicativity of the higher APS index. 
\end{proof}

\begin{thm}\label{lem:Chern}
Assume that $\Gamma$ is an exact group. Then, for any pair $(M, N)$ satisfying (1), (2), (3) of \cref{para:cocycle}, there is a vector bundle $V$ on $M$ such that the following diagrams commute;  
\[\mathclap{
\xymatrix@C=2em{
\K(C^* (\widetilde{M})^\Gamma ) \ar[r]^{\Ch_{\Gamma, V} } \ar[d]^{\tau_\sigma}  & \overline{\HP}{}_*(A \Gamma ) \ar[d]^{\tau_\sigma } \\
\K_{*}(C^* (\widetilde{N})^\pi) \ar[r]^{\Ch_{\pi, V} } & \overline{\HP}{}_{*}(A\pi ) ,
}
\xymatrix@C=2em{
\K(D^* (\widetilde{M})^\Gamma  ) \ar[r]^{\Ch ^\del_{\Gamma, V}  } \ar[d]^{\tau_\sigma }  & \overline{\HP}{}_*^\del (A \Gamma ) \ar[d]^{\tau_\sigma } \\
\K_{*}(D^* (\widetilde{N})^\pi ) \ar[r]^{\Ch ^\del_{\pi , V}  } & \overline{\HP}{}_{*}^\del (A \pi), 
}
\xymatrix@C=2em{
\K(Q^* (\widetilde{M})^\Gamma  ) \ar[r]^{\Ch ^e_{\Gamma , V}  } \ar[d]^{\tau_\sigma }  & \overline{\HP}{}_*^e (A \Gamma ) \ar[d]^{\tau_\sigma } \\
\K_{*}(Q^* (\widetilde{N})^\pi ) \ar[r]^{\Ch ^e_{\pi, V}  } & \overline{\HP}{}_{*}^e (A \pi).
}}
\]
\end{thm}

\begin{proof}
The commutativity of the left diagram is due to the functoriality, more precisely the compatibility with the boundary map, of the Chern-Connes character (cf. \cref{rmk:Chern}). 

We show the commutativity of the right diagram. Recall that $\Ch_\Gamma$ and $\Ch_\Gamma ^e$ are defined in algebraic  (\cite{piazzaMappingAnalyticSurgery2019}*{Section 6.4}). Since $\tau_\sigma $ is also defined in algebraic level (cf. \cref{def:codim2_cyclic}), we get a commutative diagram
\[\mathclap{
\xymatrix{
\K_*(C_0(T^* M)) \ar[r] \ar@/^2em/[rrr]^{ \Ch_{\Gamma,V}^e} \ar[d]^{\tau_\sigma} & 
\K_*(\Psi^{-\infty}_{\bC \Gamma }(\widetilde{M})) \ar[r]^{\Ch_{\Gamma,V} } \ar[d]^{\tau_\sigma} & 
\overline{\HP}_*(\bC [\Gamma])  \ar[d]^{\tau_\sigma} & 
\overline{\HP}{}_*^e(\bC [\Gamma]) \ar[l] \ar[d]^{\tau_\sigma} \\
\K_*(C_0(T^* N) ) \ar[r] \ar@/_2em/[rrr]^{\Ch_{\pi,V}^e}  & 
\K_*(\Psi^{-\infty}_{\bC \pi}(\widetilde{N}) ) \ar[r]^{\Ch_{\pi,V} }  & 
\overline{\HP}_*(\bC[\pi] )  & 
\overline{\HP}{}_*^e(\bC[\pi] ). \ar[l]  
}}
\]
Here we write the mapping cone $C(C(M) \to C(\bS M))$ as $C_0(T^*M)$ in short. This finishes the proof since $\overline{\HP}{}^e_*(\bC[\Gamma ]) \to \overline{\HP}{}_*(\bC[\Gamma])$ is injective. 

In the same way, a diagram chasing argument shows the commutativity of the middle diagram on the subgroup $\cX := \mathrm{Im}(\K_*(C^*(\widetilde{M})^\Gamma )_\bC \to \K_*(D^*(\widetilde{M})^\Gamma )_\bC)$ as 
\[\mathclap{
\xymatrix{
\K_*(C^*(\widetilde{M})^\Gamma) \ar[r] \ar@/^2em/[rrr]^{ \Ch_{\Gamma, V}} \ar[d]^{\tau_\sigma} & 
\K_*(D^*(\widetilde{M})^\Gamma) \ar[r]^{\Ch_{\Gamma,V}^\del  } \ar[d]^{\tau_\sigma} & 
\overline{\HP}{}_*^\del(A\Gamma)  \ar[d]^{\tau_\sigma} & 
\overline{\HP}{}_*(A\Gamma) \ar[l] \ar[d]^{\tau_\sigma} \\
\K_*(C^*(\widetilde{N})^\pi ) \ar[r] \ar@/_2em/[rrr]^{\Ch_{\pi , V}}  & 
\K_*(D^*(\widetilde{N})^\pi ) \ar[r]^{\Ch_{\pi,V}^\del  }  & 
\overline{\HP}{}_*^\del(A\pi )  & 
\overline{\HP}{}_*(A\pi ). \ar[l]  
}}
\]

The remaining task is to show the commutativity of the diagram on the `complement' of $\cX$.
By the assumption of the exactness, $\Gamma$ has the strong Novikov property \cites{ozawaAmenableActionsExactness2000,higsonAmenableGroupActions2000,yuCoarseBaumConnesConjecture2000}, and hence we have 
\begin{align*}
    \K_*(C^*(\widetilde{M})^\Gamma )_\bC  &\cong \K_*^\Gamma (\underline{E}\Gamma, E\Gamma)_\bC \oplus \K_*^\Gamma(E\Gamma )_\bC, \\
    \K_*(D^*(\widetilde{M})^\Gamma )_\bC  &\cong \K_*^\Gamma (\underline{E}\Gamma, E\Gamma)_\bC \oplus \K_*^\Gamma(E\Gamma, \widetilde{M})_\bC. 
\end{align*}
Let $x_1, \cdots, x_k \in \K_*(B\Gamma, M)$ be a finite number of elements such that their images generate the subgroup $\mathrm{Im} (\partial \colon \K_*(B\Gamma, M) \to \K_{*+1}(M))$. The, by the above isomorphism, $x_1, \cdots, x_k$ and $\cX$ generates $\K_*(D^*(\widetilde{M})^\Gamma) $.
Let $[W_i,f_i,E_i]$ be geometric cycles representing $x_i$. We may assume that the tangent bundles $T(\partial W_i)$ are trivial. 
We choose a vector bundle $V \to M$ and a Riemannian metric on it in the way that there are mutually disjoint isometric open embeddings $\partial W_i \times \bD_r^m \subset V$. 
Then, by \cref{thm:geom_transfer_rel} we have $\Ch_{\Gamma, V}^\del(x_i)=\Ch_{\Gamma, V}(x_{W_i,f_i,E_i}) =0$. 
Moreover, by \cref{lem:Khomology2}, we can apply the same argument to $\tau_\sigma (x_i) = [W_i \cap N, f_i |_{W_i \cap N}, E_i|_{W_i \cap N}]$, and get $\Ch_{\pi,V}^\del (\tau_\sigma(x_i)) =0$. 
This finishes the proof. 
\end{proof}

\begin{rmk}
It is natural to ask whether $\Ch_{\Gamma, V} ^\del $ coincides with $\Ch_{\Gamma }^\del $ in general, but here we do not deal with this problem. 
We only remark that $\Ch_\Gamma = \Ch_{\Gamma , V}$ is easily verified from the definition, and hence a diagram chasing argument similar to the above proof of \cref{lem:Chern} shows that $\Ch_\Gamma^\del = \Ch_{\Gamma ,V}^\del$ on the subgroup $\cX $.
\end{rmk}

\subsection{The higher index pairing}
Finally, we relate the primary and secondary higher index pairings with the codimension $2$ transfer map.

Let $\Gamma $ be a finitely presented hyperbolic group and let $\pi$ be its hyperbolic subgroup (note that any hyperbolic group is exact \cite{adamsBoundaryAmenabilityWord1994}). 
Let $S$ be a generating set of $\Gamma$ such that $S \cap \pi$ generates $\pi$. Let $\ell$ and $\ell_\pi$ denote the associated word-length function on $\Gamma$ and $\pi$ respectively. Note that $\ell(h) \leq \ell_\pi(h)$ holds for any $h \in \pi$.
Jolissaint \cite{jolissaintKtheoryReducedAlgebras1989} and Lafforgue \cites{lafforgueProofPropertyRD2000,lafforgueKtheorieBivariantePour2002}, there is an increasing sequence of unconditional norms 
\[ 
\big\| \sum_g a_g u_g \big\|_{A_n\Gamma }^2 := \sum_{g} |a_g|^2 (1 + \ell (g))^{2n}
\]
of the group algebra $\bC[\Gamma]$. 
It is shown in \cite{lafforgueProofPropertyRD2000}*{Proposition 1.2} that this norm is submultiplicative for sufficiently large $n$.
In \cite{puschniggNewHolomorphicallyClosed2010}, Puschnigg has introduced a 
new series of unconditional norms $\| \cdot \| _{A_n^k \Gamma }$ constructed from $\| \cdot \|_{A_n\Gamma}$ for $k \in \bN$. We write $A_n^k \Gamma $ for the Banach $\ast$-algebra completion of $\bC[\Gamma ] $ with respect to $\|\cdot  \|_{A_n^k \Gamma} $.

\begin{para}\label{para:cyclic}
The family $\{ \| \cdot \|_{A_n^k \Gamma }\}_{n,k \in \bN}$ is an increasing sequence of submultiplicative norms on $B \rtimes_\alg \Gamma$.
We write the projective limit Fr\'{e}chet algebra as $\cA \Gamma$. 
It is identified with the intersection $\bigcap_{n,k \in \bN} A_n^k \Gamma  $ as $\ast$-subalgebras of $C^*_{\mathrm{red}}\Gamma$. 
They have the following properties:
\begin{enumerate}
    \item Each $A_n^k \Gamma $ is closed under holomorphic functional calculus.
    \item There are inclusions $\bC[ \Gamma ] \subset A_n^k \Gamma \subset A_m^l \Gamma \subset C^*_{\mathrm{red}} \Gamma $ for $n \geq m$ and $k \geq l$. 
    \item There is a splitting $\kappa_\Gamma \colon \HP^*(\bC[\Gamma]) \to \HP^*(\cA \Gamma )$ of the map $\HP^*(\cA \Gamma) \to \HP^*(\bC[\Gamma])$ induced from the inclusion. 
    \item Let $\HP^*(\cA \Gamma ; \langle g \rangle )$ denote the cohomology group of the closure of the subcomplex $\CC^\bullet (\bC[\Gamma];\langle g \rangle )$ in the cyclic cocomplex $\CC^\bullet (\cA \Gamma ) $. There is a splitting surjection $ c_{\langle g \rangle} \colon \HP^*(\cA \Gamma)  \to \HP^*(\cA \Gamma ; \langle g \rangle )$ of the map induced from the inclusion.  
\end{enumerate}
\end{para}
In the same way, for a $\Gamma$-C*-algebra $B$, we write the projective limit of $\{ A_n^k(B, \Gamma) \}_{n,k}$ as $\cA(B,\Gamma)$. By \cref{lem:Banach_Cstar}, we have $\cA(B,\Gamma) \subset B \rtimes_{\mathrm{red}}\Gamma$ when $B$ is one of the C*-algebras $C_0(\cB_0, \cK_\sigma)$, $C(\cB, \cK_\sigma)$ or $\bK$.

By \cref{para:coarse} (3), the pairing of $\HP_*(\cA\Gamma)$ and $\HP^*(\bC[\Gamma ] ; \langle g \rangle)$ makes sense. 
The invariants treated in this section are defined by using this pairing, as are introduced in \cites{connesCyclicCohomologyNovikov1990,lottHigherEtainvariants1992}. We also refer the reader to \cite{piazzaMappingAnalyticSurgery2019}*{Section 6}. 

\begin{defn}\label{defn:CM_Lott}
Let $\Gamma$ be a (subgroup of a) hyperbolic group, let $M$ be a closed manifold with $\pi_1(M) \cong \Gamma $ and let $N$ be its codimension $2$ submanifold satisfying (1), (2), (3) of \cref{para:cocycle}. 
Let $\phi \in \HC^*_e(\bC[\Gamma])$ and let $\psi \in \HC^*_{del}(\bC[\Gamma])$. 
We fix a vector bundle $V$ on $M$ as in \cref{lem:Chern}. 
\begin{enumerate}
\item If $M$ is spin, the higher index is defined as $\alpha_\phi(M):= \langle \Ch_{\Gamma, V} (\alpha_\Gamma (M)) , \phi \rangle$.
\item If $M$ is spin and equips a psc metric $g_M$, the higher $\rho$-number is defined as $\varrho_\psi(g_M):= \langle \Ch_{\Gamma, V}^{del}(\rho(g_M)) , \psi \rangle$. 
\item The higher signature is defined as $\Sgn_\phi (M):=\langle \Ch_{\Gamma , V} (\Sgn _\Gamma (M)) , \phi \rangle$. 
\item If $M$ equips an oriented homotopy equivalence $f_M \colon M' \to M$, the higher $\rho$-number is defined as $\varrho_{\psi }^\sgn (f_M) := \langle \Ch_{\Gamma, V}^\del (\rho_\sgn (f_M)) , \psi \rangle$. 
\end{enumerate}
\end{defn}
The Connes--Moscovici higher index theorem \cite{connesCyclicCohomologyNovikov1990}*{Theorem 5.4} states that, for any $\phi \in H^{2q}(\Gamma ; \bC)$, the equalities
\begin{align}
\begin{split}
    \alpha _\phi (M) &= \frac{(-1)^n}{(2\pi i )^q} \langle \hat{A}(M) \xi_M^*(\phi) , [M] \rangle,\\
    \Sgn_\phi (M) &= \frac{(-1)^n}{(2\pi i )^q} \langle L(M) \xi_M^*(\phi) , [M] \rangle 
    \end{split}
    \label{eq:ConnesMoscovici}
\end{align} 
hold, where $n:=\dim M$. Note that our convention of the pairing is different from the one in \cite{connesCyclicCohomologyNovikov1990} by the normalization factor $q!/(2q)!$. 

\begin{lem}\label{lem:poly}
The following diagram commutes;  
\[
\xymatrix{
\HP^*(\bC[\pi]) \ar[r]^{\kappa_\pi} \ar[d]^{\tau_\sigma^*} & 
\HP^*(\cA \pi) \ar[d]^{\tau_\sigma ^*} \\
\HP^*(\bC[\Gamma ]) \ar[r] ^{\kappa_\Gamma} & \HP^*(\cA\Gamma ).
}
\]
\end{lem}
\begin{proof}
Let $\CC^\bullet_{\mathrm{pol}}(\bC[\Gamma])$ denotes the subcomplex of $\CC^\bullet _{\mathrm{pol}} (\bC [\Gamma])$ consisting of cyclic cochains of polynomial growth. 
Since $\Gamma $ is hyperbolic, any polynomial growth cyclic cochain extends to $\cA \Gamma$ and $\CC^\bullet_{\mathrm{pol}} (\bC[\Gamma]) \to \CC^\bullet (\bC[\Gamma ])$ is a quasi-isomorphism. The splitting $\kappa_\Gamma $ is induced from $\CC^\bullet _{\mathrm{pol}}(\bC[\Gamma]) \to \CC^\bullet (\cA \Gamma)$. 

We generalize the notion of polynomial growth for cyclic cohomology groups of crossed products. 
For a cyclic cocycle $\phi \in \CC^\bullet (B \rtimes _\alg \Gamma ; \langle  g \rangle ) $, we define $\phi_{g_0, g_1,\cdots, g_n} \in \Hom (B^{\otimes n+1}, \bC)$ and $\phi_{1,g_1, \cdots, g_n} \in \Hom (B^{\otimes n}, \bC)$ as 
\begin{align*}
    \phi_{g_0,g_1,\cdots, g_n} (b_0,b_1, \cdots, b_n)&: = \phi(b_0u_{g_0},b_1u_{g_1}, \cdots, b_nu_{g_n}), \\
    \phi_{1,g_1,\cdots, g_n} (b_1, \cdots, b_n)&: = \phi(1,b_1u_{g_1}, \cdots, b_nu_{g_n}).
\end{align*}
We say that $\phi$ is of polynomial growth if $\| \phi_{g_0, \cdots , g_n}\| \leq C(1+ \ell(g_0) \cdots \ell(g_n))^N$ and $\| \phi_{1,g_1, \cdots , g_n}\| \leq C(1+ \ell(g_1) \cdots \ell(g_n))^N$ hold for some $C>1$ and $N \in \bZ_{>0}$. Then, for any $x_0 = \sum b_{0,g}u_g, \cdots, x_n = \sum_{g}b_{n,g}u_g \in B \rtimes_\alg \Gamma$, we have a bound
\begin{align*}
     &| \phi (x_0, \cdots, x_n) | \\
     \leq & \sum_{g_0\cdots g_n \in \langle g \rangle } \| \phi_{g_0,\cdots, g_n} \| \cdot \| b_{0,g_0} \| \cdots \| b_{n,g_n} \| \\
     \leq & \sum_{g_0\cdots g_n \in \langle g \rangle } C^N(1 + \ell(g_0) \cdots \ell(g_n))^{N} \| b_{0,g_0} \| \cdots \| b_{n,g_n} \| \\ 
     \leq & \sum_{g_0\cdots g_n \in \langle g \rangle } (C(1 + \ell (g_0))^{N}\| b_{0,g_0}\|) \cdot \cdots \cdot (  C (1 + \ell (g_n))^{N} \| b_{0,g_n}\|) \\
     =& \| m_{\langle g \rangle } (z_0z_1 \cdots , z_n) \|_{\ell^1 \Gamma }, 
\end{align*}
where $z_i := \sum_g  C(1 + \ell (g))^N\| b_{i, g}\| \cdot u_g \in \bC[\Gamma ]$ for $i =0, \cdots, n$ and $m_{\langle g \rangle}$ is the projection from $\bC[\Gamma]$ to the subalgebra $\bC[\langle g \rangle]$ supported on $\langle g \rangle $. As is shown in \cite{puschniggNewHolomorphicallyClosed2010}*{Proposition 5.5}, the right hand side is bounded by a defining seminorm $\| x_0 \otimes \cdots \otimes x_n\|$ of $\cA (B, \Gamma)^{\otimes (n+1)}$. 
Hence $\phi$ extends to a cyclic cocycle on $\cA (B, \Gamma)$. Now, for $h \in \pi$, the diagram of exact sequences 
\[
\xymatrix{
\CC_{\mathrm{pol}}^\bullet (\bK \rtimes_\alg \Gamma ) \ar[r] \ar[d] & 
\CC_{\mathrm{pol}}^\bullet (C(\cB,\cK_\sigma ) \rtimes_\alg \Gamma ) \ar[r] \ar[d] & 
\CC_{\mathrm{pol}}^\bullet (C_0(\cB_0,\cK_\sigma ) \rtimes_\alg \Gamma )  \ar[d] \\
\CC^\bullet (\cA (\bK , \Gamma )) \ar[r] & 
\CC^\bullet (\cA (C(\cB,\cK_\sigma ) ,\Gamma )) \ar[r] & 
\CC^\bullet (\cA (C_0(\cB_0,\cK_\sigma ),  \Gamma ) )
}
\]
induces the commutative diagram
\[
\xymatrix{
\HP^*_{\mathrm{pol}}(\bK \rtimes_\alg \Gamma ; \langle h \rangle) \ar[r]^{\partial \hspace{3ex}} \ar[d]^{\kappa_\Gamma } & \HP^*_{\mathrm{pol}}(C_0(\cB_0, \cK) \rtimes_\alg \Gamma ; \langle h \rangle) \ar[d] \ar[r] & \HP^*_{\mathrm{pol}}(\bC[\pi]; \langle h \rangle) \ar[d]^{\kappa_\pi} \\
\HP^*(\cA (\bK ,\Gamma) ; \langle h \rangle) \ar[r]^{\partial \hspace{3ex}}  & 
\HP^*(\cA(C_0(\cB_0, \cK) , \Gamma) ; \langle h \rangle)  \ar[r] & 
\HP^*(\cA\pi; \langle h \rangle).
}
\]
Here we also use the fact that the restriction of a polynomial growth cocycle on $(C_0(\cB_0,\bK_\sigma) \rtimes_\alg \Gamma $ to $S\bC[\pi]$ is of polynomial growth with respect to $\ell_\pi$.
This finishes the proof of the lemma. 
\end{proof}

\begin{thm}\label{thm:pair_tranfer}
Let $\Gamma$ be a hyperbolic group and let $\pi$ be its finitely presented subgroup. 
Let $\phi \in \HC^{q-2}_{e}(\bC[\pi] )$ and $\psi \in \HC^{q-2}_{\del}(\bC[\pi] )$. Then, for $M$, $N$, $g_M$, $g_N$, $g_M$ and $f_N$ as in the statement of \cref{thm:second_tr},  the following equalities hold: 
\begin{enumerate}
    \item $2\pi i \alpha_{\sigma \cdot \phi } (M) = \alpha_\phi (N)$, 
    \item $2\pi i \varrho_{\sigma \psi} (g_M) = \varrho_\psi (g_N)$,
    \item $\pi i \Sgn_{\sigma \cdot \phi } (M) = \Sgn_\phi (N)$,
    \item $\pi i \varrho_{\sigma \psi}^{\sgn } (f_M) = \varrho_\psi^\sgn (f_N)$.
\end{enumerate}  
\end{thm}
\begin{proof}
This follows from \cref{cor:cocycle_dual,lem:Chern,lem:poly}.
\end{proof}

\begin{rmk}
\cref{thm:pair_tranfer} (1), (3) are also proved by the Connes--Moscovici higher index theorem \eqref{eq:ConnesMoscovici} as 
\begin{align*}
    \alpha_{\sigma \cdot \phi } (M) =& \frac{(-1)^n}{(2\pi i )^q}\langle f^*(\sigma \cdot \phi ) \hat{A}(M) , [M] \rangle = \frac{(-1)^n}{(2\pi i )^q}\langle \langle \mathrm{PD}[N] \cdot f^*(\phi)\cdot \hat{A}(M) , [M] \rangle\\
    =& \frac{1}{2\pi i } \cdot \frac{(-1)^{n-2}}{(2\pi i )^{q-1}}\langle \langle f^*(\phi)\cdot \hat{A}(N) , [N] \rangle =  \frac{1}{2\pi i}\alpha_\phi(N).
\end{align*}
The same is also true for the higher signature. 
\end{rmk}

\begin{exmp}
Let $M$, $N$, $g_M$, $g_N$, $f_M$ and $f_N$ be as in \cref{exmp:bundle2}. 
Assume that the delocalized $\eta$-invariants $\eta_{\langle h \rangle }(g_N)$ and $\eta_{\langle h \rangle }(f_N)$ do not vanish (recall that they are the same thing as the higher $\rho$-number of the cyclic cocycle $\mathrm{tr}_h \in \HC^0(\bC[\pi])$ corresponding to $\langle h \rangle \in \langle \pi \rangle $ as is shown in   \cite{xieDelocalizedEtaInvariants2019}*{Theorem 1.1}). 
Now the higher $\rho$-number of $g_M$ with respect to the delocalized cyclic cocycle $\sigma \cdot \mathrm{tr}_{\langle h \rangle}$ is calculated as
\begin{align*}
    \varrho_{\sigma \cdot \mathrm{tr}_{h }} (g_M) &= (2\pi i)^{-1} \varrho_{\mathrm{tr}_{h}} (g_N) = (2\pi i)^{-1}\eta_{\langle h \rangle }(g_N), \\
    \varrho_{\sigma \cdot \mathrm{tr}_{h }}^\sgn  (f_M) &= (\pi i)^{-1} \varrho_{\mathrm{tr}_{h}}^\sgn  (f_N) = (\pi i)^{-1}\eta_{ \langle h \rangle  }(f_N).
\end{align*} 
They are non-trivial examples in which the higher $\rho$-number of higher degree is determined.
\end{exmp}

\appendix
\section{Secondary external product via pseudo-local C*-algebra}\label{section:product}
In this appendix, we define the secondary external product in coarse C*-algebra K-theory, the product of the analytic structure set and the K-homology group, in terms of the pseudo-local coarse C*-algebra $D^*(\widetilde{M})^\Gamma$, instead of Yu's localization algebra studied in \cite{zeidlerPositiveScalarCurvature2016}. 
The construction is inspired from the definition of the Kasparov product \cite{kasparovOperatorFunctorExtensions1980}.

\begin{rmk}\label{rmk:odd_RealK}
We use the K-theory of $\bZ_2$-graded C*-algebras by Van Daele~\cite{vandaeleKtheoryGradedBanach1988} for coarse C*-algebras with Clifford algebra symmetry.  
In general, for a $\bZ_2$-graded (Real) C*-algebra $A$, its (Real) K-theory $\K_1(A)$ is defined to be the set of homotopy classes of (Real) odd self-adjoint unitaries on $A$. 
We define the coarse C*-algebras $C^*_{p,q}(\widetilde{M})^\Gamma$, $D^*_{p,q}(\widetilde{M})^\Gamma$, $Q^*_{p,q}(\widetilde{M})^\Gamma$ consisting of operators which is graded-commutative with the action of Clifford algebras. 
Then $\K_1(C^*_{p,q}(\widetilde{M})^\Gamma) \cong \K_{1-p+q}(C^*(\widetilde{M})^\Gamma)$ holds, and the same is also true for $D^*$ and $Q^*$. 

For a spin manifold $M$ with $n:=\dim M =8m - q$ (where $q = 0, \cdots, 7$), the Dirac operator on the spinor bundle of $\mathrm{Spin}(M) \times_{\mathrm{Spin}_n} \Delta_{8m}$ is equipped with an additional symmetry of $\Cl_{0,q}$, and hence determines a K-theory class $[\chi (\slashed{D}_M)] \in \K_{1}(Q^*_{0,q}(\widetilde{M})^\Gamma)$ (cf. \cref{para:Dirac}). Similarly, the signature operator on an odd-dimensional manifold determines a complex K-theory class $[\chi(D_M^\sgn)] \in \K_1(Q^*_{0,1}(\widetilde{M})^\Gamma)$ (we refer to \cite{rosenbergSignatureOperator2006}*{Definition and Notation 1.1}).
\end{rmk}

\begin{defn}
The external product
\[{\cdot} \boxtimes {\cdot} \colon  \K_1(D^*_{p_1,q_1}(M_1)^{\Gamma _1}) \otimes \K_1(Q^*_{p_2,q_2}(\widetilde{M_2})^{\Gamma_2}) \to \K_1(D^*_{p,q}(\widetilde{M}_1 \times \widetilde{M}_2)^{\Gamma _1 \times \Gamma _2}), \]
where $p=p_1+p_2$ and $q=q_1+q_2$, is defined as
\[[F_1] \boxtimes [F_2]:= [ f(F_1 \hotimes 1 + (1-F_1^2) \hotimes F_2) ],  \]
where $f(x):=x/|x|$. 
\end{defn}
This definition makes sense because the operator $F_1 \hotimes 1 + (1-F_1^2) \hotimes F_2$ is invertible in $D^*_{p+q}(\widetilde{M}_1 \times \widetilde{M}_2)^{\Gamma _1 \times \Gamma _2}$. Indeed, this is seen as
\[ (F_1 \hotimes 1 + (1-F_1^2) \hotimes F_2)^2 = F_1 ^2 \hotimes 1 + (1-F_1^2) \hotimes F_2^2 \geq F_1^2 \hotimes  1 >0. \]
This also shows that $[F_1] \boxtimes [F_2]$ is well-defined independent of the choice of representatives $F_1$ and $F_2$.

\begin{lem}\label{lem:external_operator}
For $i=1,2$, let $D_i$ be a $\Gamma_i$-invariant $\bZ_2$-graded elliptic first-order differential operator on $M_i$ anticommuting with a $\bZ_2$-graded representation of $\Cl_{p_i,q_i}$. Moreover, we assume that $D_1$ is invertible. Then we have
\[ [\chi(D_1)] \boxtimes [\chi(D_2) ] = [\chi (D_1 \hotimes 1 + 1 \hotimes D_2)]. \]
\end{lem}
\begin{proof}
This is a standard argument in Kasparov theory. We just refer the reader to \cite{baajTheorieBivarianteKasparov1983} or \cite{blackadarTheoryOperatorAlgebras1998}*{Proposition 18.10.1}. 
\end{proof}

\cref{lem:external_operator} shows that
\begin{align*}
    [\chi(\slashed{D}_{M_1})] \boxtimes [\chi (\slashed{D}_{M_2})] &= [\chi (\slashed{D}_{M_1 \times M_2})], \\
    [\chi(D_{M_1}^\sgn )] \boxtimes [\chi (D_{M_2}^\sgn )] &= 2^\epsilon [\chi (D^\sgn _{M_1 \times M_2})], 
\end{align*}
where $\epsilon = 1$ if both $\dim M_1$ and $\dim M_2$ are odd, and otherwise $\epsilon = 0$. In particular, when $M_2=\bR$, these equalities means that
\begin{align*}
    \rho(g_{N\times \bR}) &= \rho(g_N ) \boxtimes [\bR] ,\\
    \rho_\sgn(f_{N \times \bR}  ) &= 2^{\epsilon } \rho_\sgn (f_N ) \boxtimes [\bR]_\sgn , 
\end{align*}
where $\epsilon$ is $0$ or $1$ if $\dim N$ is even or odd. Therefore, the coarse Mayer--Vietoris boundary map
\begin{align*}
    \mathclap{
    \partial_{\mathrm{MV}} \colon \K_*(D^*_{p,q}(\widetilde{N} \times \bR )^\pi ) \to \K_*(D^*_{p,q}(\widetilde{N} \subset \widetilde{N} \times \bR)^\pi )}
\end{align*}
sends the higher $\rho$-invariants of $N \times \bR$ as
\begin{align*}
    \partial_{\mathrm{MV}} (\rho (g_{N \times \bR})) &= \partial_{\mathrm{MV}} (\rho(g_N ) \boxtimes [\bR]) = \rho (g_N) \boxtimes \partial_{\mathrm{MV}} [\bR] = \rho(g_N), \\
    \partial_{\mathrm{MV}} (\rho_\sgn (f_{N \times \bR})) &= \partial_{\mathrm{MV}} (2^\epsilon \rho(f_N ) \boxtimes [\bR]) = 2^\epsilon \rho (f_N) \boxtimes \partial_{\mathrm{MV}} [\bR] = 2^\epsilon \rho(g_N).
\end{align*}

This completes the proof of \cref{lem:boundary} (3), (4).

\section{Cyclic homology of a crossed product and group homology}\label{section:Chern}
In this appendix, we give a more detailed discussion on the proof of \cref{lem:cyc_unit}.
We show that the exact sequence of cyclic chain complexes 
\begin{align}
   0 \to \ker \theta  \to \CC_*(C(\cB, \cK_\sigma ) \rtimes_\alg  \Gamma )[\![v^{\pm 1}]\!] \xrightarrow{\theta } \CC_*(\bC[\Gamma])[\![v^{\pm 1}]\!] \to 0 \label{eq:cyclic_exact}
\end{align}
is quasi-isomorphic to the exact sequence of chain complexes
\begin{align}
    0 \to C_\bullet (\Gamma , \Omega^\bullet _0(\cB_0 ) )[\![v^{\pm 1}]\!] &\to C_\bullet (\Gamma , \Omega^\bullet (\cB ) )[\![v^{\pm 1}]\!] 
    \to C_\bullet (\Gamma , \bC)[\![v^{\pm 1}]\!]  \to 0 \label{eq:exact_diff2}
\end{align}
given in  \eqref{eq:exact_diff} (note that the excision property of cyclic homology states that the inclusion $\CC_*(C_0(\cB_0, \cK_\sigma ) \rtimes _\alg \Gamma )[\![v^{\pm 1}]\!] \to \ker \theta $ is quasi-isomorphic).
To this end, we construct maps from \eqref{eq:cyclic_exact} and \eqref{eq:exact_diff2} to another exact sequence of complexes
\begin{align}
    0 \to \underline{\Omega}{}_0(\cB_0 \rtimes \Gamma )[u^{\pm 1}]' \to \underline{\Omega}(\cB \rtimes \Gamma )[u^{\pm 1}]' \to \Omega(\pt \rtimes \Gamma )[u^{\pm 1}]' \to 0. \label{eq:simp_dR}
\end{align}
Here, for a vector space $V$,  $V'$ stands for the algebraic dual vector space $\Hom (V,\bC)$.

We start with the definition of \eqref{eq:simp_dR}. 
Following \cite{angelCyclicCocyclesTwisted2013}*{Subsection 2.3}, 
let $\underline{\Omega}{}_\bullet(\cB)$ denotes the dual cocomplex of $\Omega^\bullet (\cB)$, which is equipped with the degree $-1$ differential $d_{\mathrm{dR}}$. 
we define the $\bZ$-graded vector space $\underline{\Omega}^\bullet  (\cB \rtimes \Gamma  )[u^{\pm 1}]$, where $u$ is a degree $2$ formal symbol, as
\begin{align*}
    &\underline{\Omega}^{r,s} (\cB \rtimes \Gamma  ) \\
    :=& \Big\{ (\omega _{(n)}) \in \prod_{n \geq 0} \underline{\Omega}^r (\cB \times \Gamma ^n) \otimes \Omega^s(\Delta^n) \mid (\id \times \delta^i )^*\omega_{(n)} =  (\delta_i \times \id)^*  \omega_{(n-1)} \Big\}_{\textstyle ,}  
\end{align*}
where $\delta_i \colon \cB \times \Gamma^{n} \to \cB \times \Gamma^{n-1}$ and $\delta^i \colon \Delta^{n-1} \to \Delta^{n}$ denote the $i$-th face maps. 
We write $\omega (g_1, \cdots, g_n)$ for the restriction of $\omega $ to $\cB \times \{ g_1, \cdots, g_n \} \times \Delta^n$. 
The twisted simplicial de Rham differential is defined as $d_\Theta := ud_{\mathrm{dR}} + d_\Delta + \Theta $, where $d_\Delta$ is the de Rham differential on $\Omega^s(\Delta^n)$ and $\Theta \in \underline{\Omega}^{1,2}(\cB \rtimes \Gamma)$ is the differential form 
\[\Theta (g_1, , \cdots, g_n) := \sum_{1 \leq i \leq j \leq k} 2\alpha (g_1 \cdots g_i, g_{i+1} \cdots g_j)dt_idt_j , \]
where $\alpha (g,h) = \sigma (g,h)^{-1}d\sigma (g,h) \in \Omega ^1(\cB)$. 
The dual of the short exact sequence $0 \to \Omega_0^\bullet (\cB_0) \to \Omega ^\bullet (\cB) \to \bC \to 0$ gives rise to 
\[ 0 \to \underline{\Omega}^\bullet (\pt \rtimes \Gamma )[u^{\pm 1}] \to \underline{\Omega}(\cB \rtimes \Gamma )[u^{\pm 1}] \to \underline{\Omega}{}_0^\bullet (\cB_0 \rtimes \Gamma )[u^{\pm 1}] \to 0,  \]
which is dual to \eqref{eq:simp_dR}. 

\begin{rmk}\label{rmk:PR}
We review the Packer--Raeburn construction \cite{packerTwistedCrossedProducts1989}*{Theorem 3.4}.
Let $\cH_\sigma$ be the Hilbert bundle $\cB \times \ell^2(\Gamma)$, on which $\Gamma $ acts as 
\[u_g \xi := ((\gamma_g \otimes \lambda_g) \circ v(g)) \xi  \]
for any $\xi \in C(\cB , \cH_\sigma)$, where 
\[ v(g) :=  \mathop{\mathrm{diag}}(\sigma(g,h)^*)_{h \in \Gamma } \in C(\cB, \bB(\ell^2\Gamma )).\] 
Let $\cK_\sigma $ denote the associated compact operator algebra bundle on $\cB$. 
Then, the relation $u_gu_h = \sigma(g,h,x)^* u_{gh}$ holds, i.e., $\{ u_g\} $ is a $\sigma^*$-twisted unitary representation of the groupoid $\cB \rtimes \Gamma$. This implements a twisted $\Gamma$-equivariant Morita equivalence between $(C(\cB), \sigma)$ and the untwisted $\Gamma$-C*-algebra $C(\cB, \cK_\sigma)$. In particular, $C(\cB) \rtimes_\sigma \Gamma$ is Morita equivalent to $C(\cB, \cK_\sigma) \rtimes \Gamma$.
\end{rmk}

\begin{lem}\label{lem:pairing_homologies}
There are homomorphisms from the exact sequences \eqref{eq:cyclic_exact} and \eqref{eq:exact_diff2} to \eqref{eq:simp_dR}, which induces isomorphism of homology. 
\end{lem}
\begin{proof}
The pairing $\langle \cdot, \cdot \rangle$ of $ C_\bullet (\Gamma, \Omega^\bullet (\cB))[\![v ^{\pm 1} ]\!]$ and $\underline{\Omega} (\cB \rtimes \Gamma )^\bullet [u^{\pm 1}] $ is defined by 
\[\Big\langle \sum_{g_1, \cdots, g_m} \xi _{g_1, \cdots, g_m} v^k ,  (\omega_{(n)}) u^l \Big\rangle := \delta_{k,-l} \sum_{g_1, \cdots, g_m} \int _{\cB \times   \Delta^n} \xi_{g_1, \cdots, g_m} \wedge \omega_{g_1, \cdots, g_m} \]
satisfies $\langle \xi u, d _{\mathrm{dR}}\omega \rangle = \langle d_{\mathrm{dR}} \xi, \omega \rangle$, $\langle \xi , d_{\Delta} \omega \rangle = \langle \delta_\Gamma \xi , \omega \rangle$ and $\langle \xi, \Theta \omega \rangle = \langle \Theta \xi, \omega \rangle$. Moreover, this pairing gives rise to a commutative diagram
\[
\xymatrix{
C_\bullet (\Gamma , \Omega^\bullet_0 (\cB_0))[\![v^{\pm 1}]\!] \ar[r] \ar[d] & 
C_\bullet (\Gamma , \Omega^\bullet (\cB))[\![v^{\pm 1}]\!] \ar[r] \ar[d] & 
C_\bullet (\Gamma , \bC )[\![v^{\pm 1}]\!] \ar[d] \\
\underline{\Omega}{}_0(\cB_0 \rtimes \Gamma )[u^{\pm 1}]' \ar[r] &
\underline{\Omega}(\cB \rtimes \Gamma )[u^{\pm 1}]' \ar[r] & 
\Omega(\pt \rtimes \Gamma )[u^{\pm 1}]' .
}
\]
Moreover the left and the right vertical maps induce the isomorphism of homology. 
Indeed, the homology groups of two complexes at the left (resp.~the right) are both isomorphic to the group homology $H_{[*-1]}(\pi; \bC)$ (resp.~$H_{[*]}(\Gamma, \bC)$).

In \cite{angelCyclicCocyclesTwisted2013}, Angel constructed a pairing of  \eqref{eq:cyclic_exact} and \eqref{eq:simp_dR} as a variation of the JLO character. Although the space $\cB$ is not a manifold, the same definition also works in our setting. Moreover, due to the $1$-dimensionality of $\cB$, the construction is partly simplified.  

The simplicial connection and curvature forms of $\cH_\sigma$ is defined by gluing the connections and curvature forms 
\begin{align*}
    \nabla_u^k(g_1, \cdots, g_k) &:= d_{\mathrm{dR}} + u^{-1}d_\Delta + t_1 A(g_1) + t_2 A(g_1g_2) + \cdots + t_k A(g_1 \cdots g_k),   \\
    \vartheta_{(k)}(g_1,\cdots ,g_k) &:= \sum_{1 \leq i \leq j \leq k} \alpha (g_1 \cdots g_i, g_{i+1} \cdots g_j)(t_idt_j -t_jdt_i).
\end{align*}
The JLO homomorphism 
\[ \cT_{\mathrm{JLO}} \colon (C_\bullet (\Gamma , \CC_\bullet (C^\infty (\cB, \cK_\sigma )))[\![v^{\pm 1}]\!], b + uB+\delta_\Gamma') \to (\underline{\Omega}^\bullet  (\cB \rtimes \Gamma )[u^{\pm 1}]', d_\Theta) \]
is defined as 
\begin{align*}
    &\langle (\omega_{(n)}) u^k, \cT_{\mathrm{JLO}}(((\tilde{a}_0 \otimes \cdots \otimes a_p),(g_1, \cdots, g_q))v^l) \rangle \\
    =& \delta_{k,l} \int _{\cB \times \Delta^q} \omega_{g_1, \cdots, g_q} \wedge \int_{\Delta ^p} \mathrm{Tr}(\tilde{a_0}e^{s_0 \vartheta_{(q)}}(\nabla_u^q a_1)e^{s_1 \vartheta_{(q)}} 
    \cdots e^{s_p \vartheta_{(q)} } (\nabla_u^q a_q)) ds_1 \cdots ds_q.
\end{align*} 
It is shown in the same way as \cite{angelCyclicCocyclesTwisted2013}*{Theorem 5.5} that it is a chain map. 

Moreover, in \cite{angelCyclicCocyclesTwisted2013}*{Subsection 5.3} (we also refer to \cite{nistorGroupCohomologyCyclic1990}*{2.5}), a quasi-isomorphism 
\[\Psi_A \colon (C_\bullet (\Gamma, \CC_\bullet (A))[u^{\pm 1}], b + uB + \delta_\Gamma' )  \to (\CC_\bullet (A \rtimes _\alg \Gamma )[u^{\pm 1}] , b + uB) \]
is constructed for any $\Gamma$-algebra $A$. By the construction, this $\Psi_A$ is functorial.  
Therefore, $\Psi_A \circ \cT_{\mathrm{JLO}}$ gives rise to a commutative diagram
\[\mathclap{
\xymatrix@C=2em{
\ker \theta  \ar[r] \ar[d] & 
\CC_\bullet (C^\infty (\cB, \cK_\sigma ) \rtimes_\alg \Gamma , e)[\![v^{\pm 1}]\!] \ar[r] \ar[d] & 
\CC_\bullet (\bK \rtimes \Gamma , e )[\![v^{\pm 1}]\!] \ar[d] \\
\underline{\Omega}{}_c(\cB_0 \rtimes \Gamma )[u^{\pm 1}]' \ar[r] &
\underline{\Omega}(\cB \rtimes \Gamma )[u^{\pm 1}]' \ar[r] & 
\Omega(\pt \rtimes \Gamma )[u^{\pm 1}]' .
}}
\]
Since the left and the right vertical maps are both isomorphic, this finishes the proof of the lemma.   
\end{proof}

\bibliographystyle{alpha}
\bibliography{ref.bib}
\end{document}